\documentclass[a4paper]{article}
\widowpenalty
\clubpenalty
\usepackage[english]{babel}
\usepackage[utf8]{inputenc}
\usepackage[T1]{fontenc}
\usepackage{lmodern}
\usepackage[numbers]{natbib}
\usepackage{amsmath}
\usepackage{amsfonts}
\usepackage{amssymb}
\usepackage{amsthm}
\theoremstyle{plain}
\newtheorem{theorem}{Theorem}[section]

\newtheorem{lemma}[theorem]{Lemma}

\newtheorem{remark}{Remark}

\theoremstyle{definition}

\usepackage{graphicx}
\usepackage[font=small,format=hang,labelfont=bf,figurename=Fig.,tablename=Tab.]{caption}
\usepackage{subcaption}
\usepackage[bookmarks,bookmarksnumbered,colorlinks=true,linktocpage]{hyperref}
\usepackage{algorithm}
\usepackage{enumerate}
\DeclareMathOperator*{\Span}{span}
\DeclareMathOperator*{\argmin}{argmin}
\renewcommand{\leq}{\leqslant}
\renewcommand{\geq}{\geqslant}
\newcommand{\IR}{\mathbb{R}}
\newcommand{\IN}{\mathbb{N}}

\usepackage{listings}
\usepackage{xcolor}
\definecolor{mygreen}{rgb}{0,0.6,0}
\definecolor{mygray}{rgb}{0.97,0.97,0.97}
\definecolor{mymauve}{rgb}{0.58,0,0.82}
\usepackage[a4paper]{geometry}
\numberwithin{equation}{section}
\usepackage[noend]{algorithmic}

\DeclareMathOperator{\misfit}{\mathcal{J}}
\DeclareMathOperator{\supp}{supp}
\DeclareMathOperator{\dist}{dist}
\newcommand{\tolAngle}{{\varepsilon_\theta}}
\newcommand{\tolTrunc}{{\varepsilon_\Psi}}
\title{Adaptive Spectral Inversion for Inverse Medium Problems}
\author{
	Yannik G.\ Gleichmann\,\footnote{Department of Mathematics and Computer Science, University of Basel, Basel, Switzerland;
	\href{mailto:yannik.gleichmann@unibas.ch}{yannik.gleichmann@unibas.ch}}
	\qquad
	Marcus J.\ Grote\,\footnote{Department of Mathematics and Computer Science, University of Basel, Basel, Switzerland;
	Corresponding author: \href{mailto:marcus.grote@unibas.ch}{marcus.grote@unibas.ch}}
	\\
	\tiny{}
}
\date{\today}
\begin{document}
\maketitle
\begin{abstract}
A nonlinear optimization method is proposed for the solution of inverse medium problems
with spatially varying properties.
To avoid the prohibitively large number of unknown control variables resulting from
standard grid-based representations, the misfit is instead minimized in a small subspace
spanned by the first few eigenfunctions of a judicious elliptic operator, which itself depends
on the previous iteration. By repeatedly adapting 
both the dimension and the basis of the search space, regularization is inherently incorporated
at each iteration without the need for extra Tikhonov penalization.
Convergence is proved under an angle condition, which is included
into the resulting \emph{Adaptive Spectral Inversion} (ASI) algorithm. 
The ASI approach compares favorably to standard grid-based inversion using $L^2$-Tikhonov regularization when applied to an elliptic inverse problem. The improved accuracy
resulting from the newly included angle condition is further demonstrated via numerical experiments from
time-dependent inverse scattering problems.

\bigskip
\noindent
\textbf{Keywords:} Inverse problem, inverse scattering problem, adaptive eigenspace inversion, nonlinear optimization

\end{abstract}
\section{Introduction}
The solution of inverse medium problems entails the reconstruction of a medium's spatially varying properties, $u(x)$, inside a bounded region of space $\Omega\subset \IR^d$ from partially available, often noisy observations of a state variable $y$.
This is inverse to the direct or forward problem of determining $y$ for a given medium $u(x)$, where typically $y$ satisfies a governing partial differential equation, be it stationary or time-dependent, which involves $u(x)$. More abstractly, 
if we denote the solution to the forward problem by a {\em forward operator} $F: H_1 \to H_2$ acting between two Hilbert spaces $H_1,H_2$, the associated inverse problem then consists in solving
\begin{align}\label{eq:invpb}
	F(u) = y,
\end{align}
for a given right-hand side $y^\dagger\in H_2$, that is, to determine a solution $u^\dagger \in H_1$ such that $y^\dagger = F(u^\dagger)$.

Due to measurement errors, only approximate (noisy) data $y^\delta$ is available in practice,
\begin{align}\label{eq:noise}
	\|y^\dagger - y^\delta\|_{H_2} \leq \delta,
\end{align}
for a noise level $\delta \geq 0$.
Then, the inverse medium problem consists in finding $u^{\dagger,\delta} \in H_1$ that satisfies
\begin{align}\label{eq:IP.operator}
	F(u) = y^\delta,
\end{align}
which we reformulate as a least-squares minimization problem for the {\em misfit} $\misfit^\delta: H_1 \to \IR$ as
\begin{align}\label{eq:IP.misfit}
	u^{\dagger,\delta} = \argmin_{u \in H_1} \misfit^\delta(u) := \argmin_{u \in H_1} \frac{1}{2} \|F(u) - y^\delta\|^2_{H_2}.
\end{align}

In general, \eqref{eq:IP.operator} or \eqref{eq:IP.misfit} is ill-posed and cannot be solved directly without further regularization.
Numerical methods for obtaining stable solutions to \eqref{eq:IP.operator} or \eqref{eq:IP.misfit} essentially fall into three categories. 
First, starting from an initial guess $u^{(0)}$, \emph{iterative regularization methods}, such as the Landweber iteration \cite{hanke1995aConvergence, landweber1951anIterationFormula}, a special case of gradient descent, improve upon $u^{(m)} $ at iteration $m$ by using the derivative of $\misfit^\delta$ \cite{bakushinsky2004iterative, kaltenbacher2008iterative}.
Due to semi-convergence, the iteration needs to be stopped judiciously as the regularization parameter corresponds to the reciprocal of $m$.
Second, \emph{Tikhonov regularization} adds a regularization term $R(u)$ to $\misfit^\delta(u)$ \cite{engl1996regularization, kirsch1996an}. 
Hence, instead of solving \eqref{eq:IP.misfit}, one now seeks a minimizer of $\misfit^\delta(u) + \alpha R(u)$, where $\alpha$ is the regularization parameter which must be chosen carefully.
Typical choices for the regularization term $R(u)$ include the $L^2$-norm, the $\operatorname{TV}$-norm, or the $H^1$-semi-norm, depending on the expected smoothness in $u^\dagger$.
Third, \emph{regularization by projection} \emph{(or discretization)} in preimage-space replaces in \eqref{eq:IP.misfit} or \eqref{eq:IP.operator} the infinite dimensional space $H_1$ by a finite dimensional subspace $\Psi^{(m)}\subset H_1$ and then solves for $u^{(m)} \in \Psi^{(m)}$ \cite{groetsch1988convergence,kaltenbacher2012aConvergence}.
Here, $\Psi^{(m)}$ typically corresponds to the underlying subspace of a finite difference or finite element discretization whose mesh-size, $h_m>0$, then acts as the regularization parameter.
Alternative projection methods use finite-dimensional subspaces in image-space \cite{hofmann2007regularization,  kaltenbacher2000regularization}, or even projection into both image- and preimage-space simultaneously \cite{hamarik2002on, natterer1977regularisierung}.

Unless an effective parametrization of $u(x)$ is known a priori, inverse medium problems usually lead to prohibitively high-dimensional search spaces where the number of unknown parameters (or control variables) is determined by the degrees of freedom of the underlying finite difference or finite element discretization. Alternatively, sparsity promoting strategies attempt to
remain sufficiently general while keeping the dimension of the search space small by
applying $\ell_1$-norm soft-thresholding, say, to promote a sparse representation in a fixed wavelet, curvelet, etc.\ basis or frame \cite{DDM2004,HH2008,LAH2012,LDNDR2010}.

An adaptive inversion method was first introduced in \cite[Section 5.4]{deBuhan2010logaritmic}, where the search space consists of the first few eigenfunctions of a judicious elliptic differential operator $L[u^{(m)}]$. 
As $L[u^{(m)}]$ itself depends on the previous iterate, $u^{(m)}$, so do its eigenfunctions;
hence, the current search space is repeatedly adapted at every iteration.
To avoid division by zero in the presence of vanishing gradients, De Buhan and Kray later incorporated a small fixed parameter $\varepsilon$ which led to the coercive elliptic operator $L_\varepsilon[u^{(m)}]$.
Restricting the search space to the first few eigenfunctions of $L_\varepsilon[u^{(m)}]$ has proved highly effective in a number of acoustic \cite{deBuhan2013aNewApproach, graff2019howToSolve}, electro-magnetic \cite{deBuhan2017numerical} and seismic \cite{faucher2020eigenvector}  inverse scattering problems, but also in optimal control \cite{clay2021anAdaptive}.

To take advantage of the regularizing effect of a low-dimensional search space, both the eigenfunctions
and the dimension of the search space were adapted in \cite{grote2017adaptive}.
Moreover, the connection
between $L_\varepsilon$ and TV-regularization motivated the use of different elliptic operators depending on the expected smoothness in the target medium \cite{grote2019adaptive}.
Recently, a first step was taken in developing a mathematical theory underpinning the remarkable accuracy of that decomposition for the approximation
of piecewise constant functions \cite{baffet2021adaptive}, which led to rigorous $L^2$-error estimates in \cite{baffet2022error}.

The remaining part of this paper is structured as follows.
In Section \ref{sec:giaa}, we introduce the general  \emph{Adaptive Inversion} iteration for the solution of inverse problems,
which proceeds by solving \eqref{eq:IP.misfit} in a sequence of finite dimensional subspaces $\Psi^{(m)} \subset H_1$, not necessarily known a priori. Here we identify a key {\it angle condition}, which yields convergence of the Adaptive Inversion iteration, and also prove that it is a genuine regularization method.
In Section \ref{sec:asd}, we present the \emph{Adaptive Spectral (AS) decomposition} and recall its
approximation properties \cite{baffet2022error,baffet2021adaptive}. By combining the
Adaptive Inversion iteration together with the AS decomposition, we propose in Section \ref{sec:asi}
the \emph{Adaptive Spectral Inversion (ASI)} Algorithm, which incorporates the new angle condition from Section \ref{sec:giaa} by including the sensitivities of the gradient into the construction of the search space at the subsequent iteration. Finally in Section \ref{sec:numerical.results}, we present several numerical experiments which illustrate the performance of the ASI method and verify the theory from Section \ref{sec:giaa}.
In particular, we apply the ASI method to a time-dependent inverse scattering problems which demonstrates that it yields a significant improvement over previous versions from \cite{grote2017adaptive, baffet2021adaptive}.
\section{Adaptive Inversion}\label{sec:giaa}
To determine a solution to the inverse problem \eqref{eq:IP.operator} for given data $y^\delta$,
we shall minimize the misfit $\misfit^\delta$ in \eqref{eq:IP.misfit} successively in a sequence of (closed) finite-dimensional subspaces $\Psi^{(m)}\subset H_1$, $m \geq 1$, 
until its gradient $D\misfit^\delta$ vanishes.
In doing so, we do not assume the entire sequence of subspaces $\{ \Psi^{(m)} \}_{m \geq 1}$ known a priori.
Hence, we consider the following adaptive inversion algorithm:

\begin{algorithm}[H]
\caption{Adaptive Inversion}
\label{algo:adaptive.inversion.algorithm}
\algsetup{indent=2em,linenosize=\tiny,linenodelimiter=.}
\begin{algorithmic}[1]
	\REQUIRE initial guess $u^{(0),\delta}$, search space $\Psi^{(1)} \subset H_1$, $m = 1$.
	\WHILE{$\|D\misfit^\delta(u^{(m),\delta})\| \neq 0$}
		\STATE \emph{Solve} 
			\begin{align}\label{eq:algo.optimization}
				u^{(m),\delta} = \argmin_{u \in \Psi^{(m)}} \misfit^\delta(u).
			\end{align}
		\STATE \emph{Determine} new search space $\Psi^{(m+1)}$, such that $u^{(m),\delta} \in \Psi^{(m+1)}$.
		\STATE $m \gets m+1$
	\ENDWHILE
\end{algorithmic}
\end{algorithm}
By ensuring that $u^{(m),\delta}$ also belongs to the new search space $\Psi^{(m+1)}$, we guarantee in every iteration that the misfit does not increase,
\begin{align}
	\misfit^\delta(u^{(m+1),\delta}) \leq \misfit^\delta(u^{(m),\delta}), \qquad \forall m \geq 1.
\end{align}
Therefore the sequence $\{u^{(m),\delta}\}_{m \geq 1}$ obtained by the Adaptive Inversion Algorithm \ref{algo:adaptive.inversion.algorithm} is a minimizing sequence of $\misfit^\delta$.
Without further assumptions, however, we do not know yet whether this sequence converges.
\subsection{Convergence}
To prove that the sequence $\{ u^{(m),\delta} \}_{m \geq 1}$ generated by the Adaptive Inversion Algorithm indeed converges, we proceed as follows.
First, we identify a key angle condition, which ensures that for fixed $\delta \geq 0$, $\|D\misfit^\delta(u^{(m),\delta})\| \to 0$ as $m \to \infty$.
Then, under suitable assumptions from convex optimization theory, we conclude that the sequence indeed converges to $u^{\dagger,\delta}$.
\begin{theorem}\label{thm:convergence.gradient}
Let the misfit $\misfit^\delta$ be Fréchet differentiable with derivative $D\misfit^\delta(u)$, $u \in H_1$ and Lipschitz-continuous for every direction $d \in H_1$, i.e.\
\begin{align}\label{eq:lipschitz.gateaux}
	| D\misfit^\delta(u + v)d - D\misfit^\delta(u)d | \leq L \|v\|_{H_1} \|d\|_{H_1} \qquad \forall u,v \in H_1, \quad L > 0,
\end{align}
and further assume that $\misfit^\delta(u) \geq C > - \infty$ for all $u \in H_1$.
Then, if the corrections $d^{(m),\delta} = u^{(m+1),\delta} - u^{(m),\delta}$ satisfy the \emph{angle condition}
\begin{align}\label{eq:angle.condition}
	|D\misfit^\delta(u^{(m),\delta})d^{(m),\delta} | \geq
	\tolAngle \|d^{(m),\delta}\|_{H_1} \|D\misfit^\delta(u^{(m),\delta})\|, \qquad 0 < \tolAngle < 1
\end{align}
uniformly in $m$, we have
\begin{align}
	\|D\misfit^\delta(u^{(m),\delta})\| \to 0, \qquad m \to \infty.
\end{align}
\end{theorem}
\begin{proof}
First, we claim that
\begin{align}\label{eq:ArmijoEstimate}
	\misfit^\delta(u^{(m+1),\delta}) - \misfit^\delta(u^{(m),\delta})\leq
	- \frac{\alpha_{m}}{2} \mu_mD\misfit^\delta(u^{(m),\delta})d^{(m),\delta},
\end{align}
where
\begin{align}\label{eq:AlphaMax}
	\alpha_m
	= \frac{\mu_m}{L}\frac{D\misfit^\delta(u^{(m),\delta})d^{(m),\delta}}{\|d^{(m),\delta}\|^{2}} > 0,
	\qquad \mu_m = \operatorname{sign}(D\misfit^\delta(u^{(m),\delta})d^{(m),\delta}).
\end{align}
As both $u^{(m+1),\delta},u^{(m),\delta} \in \Psi^{(m+1)}$, we also have $d^{(m),\delta} \in \Psi^{(m+1)}$.
Since $\misfit^\delta({u}^{(m+1),\delta}) \leq \misfit^\delta(u)$ for all $u \in \Psi^{(m+1)}$, we may choose $u = u^{(m),\delta} - \alpha_{m} \mu_m d^{(m),\delta}$.
The linearity of the derivative thus yields
\begin{align}
		&\misfit^\delta(u^{(m+1),\delta}) - \misfit^\delta(u^{(m),\delta}) + \alpha_{m} \mu_m D\misfit^\delta(u^{(m),\delta})d^{(m),\delta} \\
		&\hspace{2em}\leq \misfit^\delta(u^{(m),\delta} - \alpha_{m} \mu_m d^{(m),\delta}) - \misfit^\delta(u^{(m),\delta}) + \alpha_{m} \mu_m D\misfit^\delta(u^{(m),\delta})d^{(m),\delta} \\
		&\hspace{2em}= \alpha_{m}\mu_m \int_{0}^{1}\left( D\misfit^\delta(u^{(m),\delta})
						- D\misfit^\delta(u^{(m),\delta} - \tau \alpha_m \mu_m d^{(m),\delta}) \right) d^{(m),\delta} d\tau \\
		&\hspace{2em}\leq \frac{\alpha_{m} L}{2} \|d^{(m),\delta}\|^{2} \alpha_m
		= \frac{\alpha_{m}}{2} \mu_m D\misfit^\delta(u^{(m),\delta})d^{(m),\delta},
\end{align}
which proves \eqref{eq:ArmijoEstimate}.

Since $\misfit^\delta$ is bounded from below and decreases in every iteration, there exists a constant $C_{0} \geq 0$ such that by \eqref{eq:ArmijoEstimate}
\begin{align}
	- C_{0} &\leq C - \misfit^\delta(u^{(1),\delta}) 
			\leq \misfit^\delta(u^{(M+1),\delta}) - \misfit^\delta(u^{(1),\delta}) \label{eq:boundedness.misfit.below.above.begin}\\
			&= \sum_{m=1}^{M} \left(\misfit^\delta(u^{(m+1),\delta}) - \misfit^\delta(u^{(m),\delta})\right) \label{eq:InequalityArmijoLemma}
			\leq - \sum_{m=1}^{M} \frac{\alpha_{m}}{2} \mu_m D\misfit^\delta(u^{(m),\delta})d^{(m),\delta} \\
			&= - \frac{1}{2L} \sum_{m=1}^M \frac{(D\misfit^\delta(u^{(m),\delta}))^2}{\|d^{(m),\delta}\|^2}\leq 0. \label{eq:boundedness.misfit.below.above.end}
\end{align}
Next, we define the angle $\theta_{m}$ via
\begin{align}
	\cos (\theta_{m}) = \frac{\mu_mD\misfit^\delta(u^{(m),\delta})d^{(m),\delta}}{\|D\misfit^\delta(u^{(m),\delta})\| \|d^{(m),\delta}\|} > 0 
\end{align}
and rewrite \eqref{eq:boundedness.misfit.below.above.begin} -- \eqref{eq:boundedness.misfit.below.above.end} using \eqref{eq:AlphaMax} as
\begin{align}
	-C_0 \leq -\frac{1}{2L} \sum_{m=1}^{M} \cos^{2}(\theta_{m}) \|D\misfit^\delta(u^{(m),\delta})\|^{2} \leq 0.
\end{align}
Taking the limit $M \rightarrow \infty$ then yields the well known \emph{Zoutendijk condition} \cite{nocedal2006numerical}
\begin{align}
	\sum_{m=1}^{\infty} \cos^{2}(\theta_{m}) \|D\misfit^\delta(u^{(m),\delta})\|^{2} < \infty,
\end{align}
which immediately implies
\begin{align}
	\cos^{2}(\theta_{m}) \|D\misfit^\delta(u^{(m),\delta})\|^{2} \to 0, \qquad m \to \infty.
\end{align}
Since we optimize for $u^{(m),\delta}$ such that the angle condition \eqref{eq:angle.condition} holds uniformly in $m$, i.e.\ $\cos(\theta_m) \geq \tolAngle$ for all $m \geq 1$, with $0 < \tolAngle < 1$, we thus conclude that
\begin{align}
	\| D\misfit^\delta(u^{(m),\delta}) \| \to 0, \qquad m \to \infty.
\end{align}
\end{proof}
Theorem \ref{thm:convergence.gradient} implies that the Adaptive Inversion Algorithm yields a minimizing sequence $\{u^{(m),\delta}\}_{m\geq1}$ such that the gradient of the misfit tends to zero;
hence, the algorithm is well-defined and converges.
Without further assumptions, however, this does not imply that the sequence $\{u^{(m),\delta}\}_{m\geq1}$ converges (weakly) to a (local) minimizer or accumulation point.
Under further standard assumptions from convex optimization theory, one can even prove convergence to a minimizer of \eqref{eq:IP.misfit}, as in \cite[Corollary 11.30]{bauschke2017convex}. 
Those assumptions, however, seldom hold in practice for inverse medium problems.
\subsection{Regularization}
In the presence of perturbed noisy data, it makes little sense to improve the approximate solution $u^{(m),\delta}$ at iteration $m$ beyond the error $\delta$ in the observations. 
Instead for $\tau > 1$, we stop the Adaptive Inversion Algorithm at iteration $m_\ast(\delta)$ when the \emph{discrepancy principle} is satisfied:
\begin{align}\label{eq:discrepancy.principle}
	m_\ast(\delta) = m_\ast = \min \{m \in \IN: \; \|F(u^{(m),\delta}) - y^\delta\|_{H_2} \leq \tau \delta \}.
\end{align}
If we assume that the search spaces $\Psi^{(m)}$ satisfy
\begin{align}\label{eq:dimension.condition.basis}
	\| (I - \Pi_{\Psi^{(m)}}) u^{\dagger} \|_{H_1} \to 0, \qquad m \to \infty,
\end{align}
where $\Pi_{\Psi^{(m)}}$ denotes the projection into $\Psi^{(m)}$ and $u^\dagger$ the exact (noise-free) solution, the following lemma guarantees that the Adaptive Inversion Algorithm will always satisfy \eqref{eq:discrepancy.principle} after a finite number of steps.
\begin{lemma} \label{claim:discrepancy.principle.well.defined}
Suppose that the forward operator $F$ is continuous and that the spaces $\{\Psi^{(m)}\}_{m\geq1}$ satisfy \eqref{eq:dimension.condition.basis}.
Then, for $\tau > 1$, there exists for every $\delta \geq 0$ an index $m_\ast(\delta)$ such that the discrepancy principle \eqref{eq:discrepancy.principle} is satisfied.
\end{lemma}
\begin{proof}
By minimality of $u^{(m),\delta}$, the residual is bounded by 
\begin{align}
	\|F(u^{(m),\delta}) - y^\delta\|_{H_2} &\leq \|F(\Pi_{\Psi^{(m)}}u^{\dagger}) - y^\delta\|_{H_2} \\
		&\leq \|F(\Pi_{\Psi^{(m)}}u^{\dagger}) - F(u^{\dagger})\|_{H_2}
			+ \|F(u^{\dagger}) - y^\delta\|_{H_2} \\
		&\leq \|F(\Pi_{\Psi^{(m)}}u^{\dagger}) - F(u^{\dagger})\|_{H_2} + \delta.
\end{align}
From the continuity of $F$ and \eqref{eq:dimension.condition.basis} we obtain
\begin{align}
	\|F(u^{(m),\delta}) - y^\delta\|_{H_2} \leq \tau \delta, \qquad \tau > 1,
\end{align}
for $m \geq M$ sufficiently large.
Thus $m_\ast(\delta) := M$ is always well defined.
\end{proof}

Next, we consider the sequence $\{ u^{(m_\ast),\delta} \}_{\delta \geq 0}$ obtained from the Adaptive Inversion Algorithm when stopped via the discrepancy principle \eqref{eq:discrepancy.principle}.
To show that it converges to the exact (noise-free) solution $u^\dagger$ as $\delta \to 0$, that is that
\begin{align}
	u^{(m_\ast),\delta} \to u^\dagger, \qquad \delta \to 0,
\end{align}
we assume that the forward operator $F$ is continuous with Fréchet derivative $DF(u)$ and that the standard \emph{tangential cone condition} (aka \emph{Scherzer condition}),
\begin{align}\label{eq:scherzer.condition}
	\| F(u) - F(v) - DF(u)(u-v) \|_{H_2}
	\leq \eta \|F(u) - F(v)\|_{H_2},
	\qquad \forall u,v \in H_1,
\end{align}
holds for some $\eta \in (0,1)$.
Moreover, we assume that $F$ is weakly sequentially closed.

Under all the above assumptions, we conclude from \cite[Theorem 3.4]{kaltenbacher2012aConvergence} that the Adaptive Inversion Algorithm is a genuine regularization method in the sense of \cite[Definition 3.1]{engl1996regularization}:

\begin{theorem}\label{thm:regularization}
Let $\Psi^{(m)}$ satisfy \eqref{eq:dimension.condition.basis}, $F$ be weakly sequentially closed, Fréchet differentiable and also satisfy \eqref{eq:scherzer.condition}.
Then, the sequence $\{u^{(m_\ast),\delta}\}_{\delta \geq 0}$ obtained from the Adaptive Inversion Algorithm \ref{algo:adaptive.inversion.algorithm}, stopped at iteration $m_\ast = m_\ast(\delta)$ according to \eqref{eq:discrepancy.principle},
admits a subsequence converging to a minimizer of $\misfit^\delta$.
Moreover, if the minimizer is unique, then the sequence $\{u^{(m_\ast),\delta}\}_{\delta \geq 0}$ converges to the exact (noise-free) solution $u^\dagger$ of the inverse problem \eqref{eq:IP.operator}, that is
\begin{align}
	u^{(m_\ast),\delta} \to u^\dagger, \qquad \delta \to 0.
\end{align}
\end{theorem}

\begin{remark}
If the nullspace of the Fréchet derivative $DF(u)$ is trivial, the Scherzer condition \eqref{eq:scherzer.condition} immediately implies the uniqueness of $u^{\dagger}$. Indeed, from 
\eqref{eq:scherzer.condition} and the triangle inequality, we first deduce that
\begin{align}
	\| DF(u)(u-v)\| \leq (1+\eta)\|F(u)-F(v)\| \qquad \forall u,v \in H_1.
\end{align}
Given two distinct solutions $u^{\dagger,\delta} \neq v^{\dagger,\delta}$ to \eqref{eq:IP.operator}, we
then infer that 
\begin{align}
	0 = \|F(u^{\dagger,\delta}) - F(v^{\dagger,\delta})\|
	\geq \frac{1}{\eta + 1} \|DF(u^{\dagger,\delta})(u^{\dagger,\delta} - v^{\dagger,\delta})\|,
\end{align}
which yields the sought contradiction and hence the uniqueness of $u^{\dagger}$.
Thus, if the nullspace of the Fréchet derivative $DF(u)$ is trivial $\forall u\in H_1$, then 
Theorem \ref{thm:regularization} immediately implies the convergence of $\{u^{(m_\ast),\delta}\}_{\delta \geq 0}$ to the exact solution. Clearly, those assumptions may not hold in practice, for instance, in a situation of  limited indirect observations.
\end{remark}

In summary, the Adaptive Inversion Algorithm generates a minimizing sequence $\{u^{(m),\delta}\}_{m \geq 1}$ for $\misfit^\delta$ and, under standard assumptions, Theorem \ref{thm:convergence.gradient} implies that the gradient of $\misfit^\delta$ tends to zero; hence, the algorithm is well defined and terminates.
In fact, under standard assumptions from convex optimization theory, there exists for every $\delta$ a minimizer $u^{\dagger,\delta}$ with $u^{(m),\delta} \to u^{\dagger,\delta}$ as $m \to \infty$.
Moreover, under standard assumptions for inverse problems, we deduce that the Adaptive Inversion Algorithm yields a genuine regularization method;
thus, $u^{(m_\ast),\delta}$ converges to the exact (noise-free) solution $u^{\dagger}$ of \eqref{eq:IP.operator} as $\delta \to 0$.
\section{Adaptive Spectral Decomposition}\label{sec:asd}
To apply the Adaptive Inversion Algorithm from Section \ref{sec:giaa}, we must specify the search spaces $\Psi^{(m)}$ in practice.
Here we shall consider low-dimensional search spaces spanned by the first few eigenfunctions of a judicious elliptic operator, which itself depends on the previous iterate.
By combining those search spaces with the Adaptive Inversion Algorithm, we shall devise in Section \ref{sec:asi} the Adaptive Spectral Inversion Algorithm for the solution of inverse medium problems.

Consider an arbitrary bounded function $u: \Omega \to \IR$ on a bounded Lipschitz domain $\Omega \subset \IR^d$, $d \geq 2$, which vanishes on the boundary.
Next, let $V_h \subset H_0^1(\Omega)$ be the finite element space of continuous, piecewise polynomials of degree $r \geq 1$, on a regular and quasi-uniform mesh with mesh-size $h>0$.
Denoting by $u_h \in V_h$ the standard $\mathbb{P}^r$-FE interpolant of $u$, we introduce the differential operator
\begin{align}\label{eq:as.operator}
	L_\varepsilon[u_h] v = - \nabla \cdot \left(\mu_\varepsilon[u_h] \nabla v \right),
	\quad
	\mu_\varepsilon[u_h] = \frac{1}{\sqrt{|\nabla u_h|^2 + \varepsilon^2}},
	\qquad \varepsilon > 0.
\end{align}
In contrast to $u$, which may be discontinuous, $u_h$ is continuous and piecewise polynomial with $\nabla u_h \in L^\infty(\Omega)$.
For $\varepsilon$ small, $L_\varepsilon[u_h]$ in \eqref{eq:as.operator} is well defined and for $h$ small even uniformly elliptic in $\Omega$ \cite{baffet2022error};
in practice, we always set $\varepsilon = 10^{-8}$.
Hence, there exists a non-decreasing sequence of strictly positive eigenvalues $\{\lambda_k\}_{k\geq1}$ with corresponding eigenfunctions $\{\varphi_k\}_{k\geq1}$ satisfying 
\begin{align}\label{eq:as.eigenfunctions}
	L_\varepsilon[u_h] \varphi_k = \lambda_k \varphi_k \quad\text{in}\; \Omega,
	\qquad
	\varphi_k = 0 \quad\text{on}\;\partial\Omega.
\end{align}
Moreover, the eigenfunctions of $L_\varepsilon[u_h]$ form an $L^2$-orthonormal basis of $L^2(\Omega)$.
\begin{figure}[ht!]
	\begin{subfigure}[c]{0.45\textwidth}
		\includegraphics[width=1\textwidth]{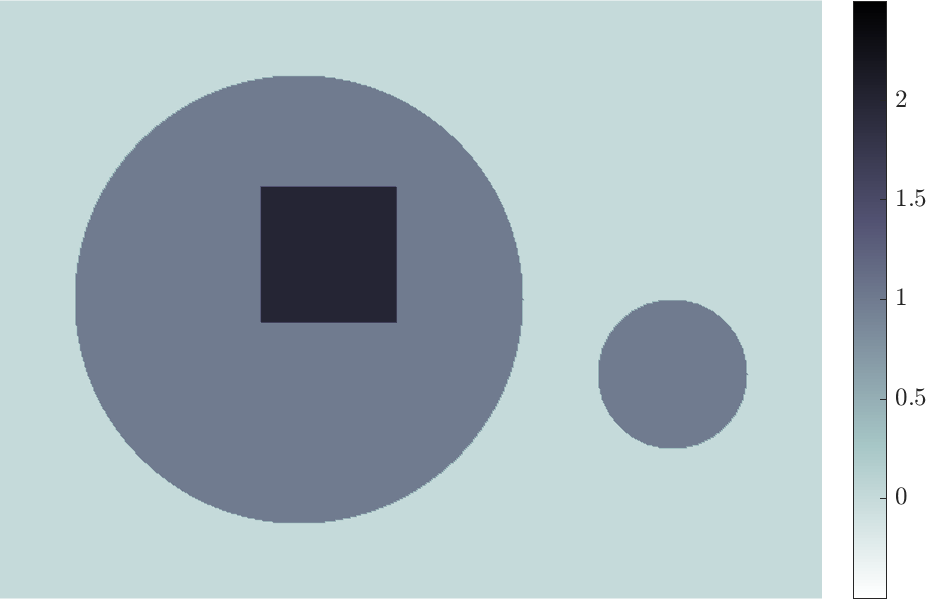}
		\caption{$u$ (or $u_h$)}
		\label{fig:asd.sets.ef.eig.medium}
	\end{subfigure}
	\hfill
	\begin{subfigure}[c]{0.45\textwidth}
		\includegraphics[width=1\textwidth]{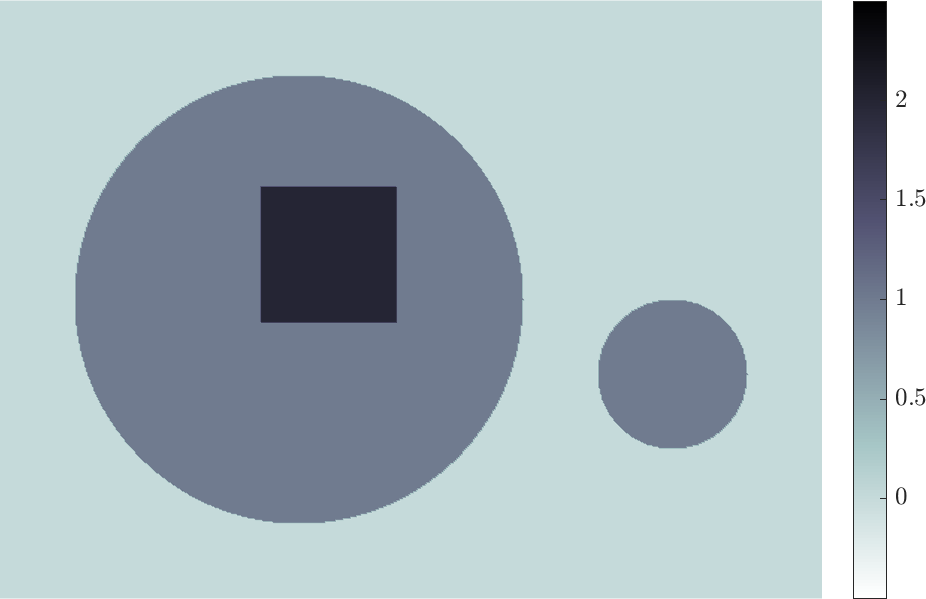}
		\caption{$\Pi_K[u_h]u$ for $K = 3$}
		\label{fig:asd.sets.ef.eig.asd}
	\end{subfigure}
	\caption{AS decomposition:
			A piecewise constant function $u:\Omega\to\IR$ and its projection $\Pi_K[u_h]u$ onto the first $K=3$ eigenfunctions from \eqref{eq:as.eigenfunctions}.}
\end{figure}
\begin{figure}[ht!]
	\begin{subfigure}[c]{0.32\textwidth}
		\includegraphics[width=1\textwidth]{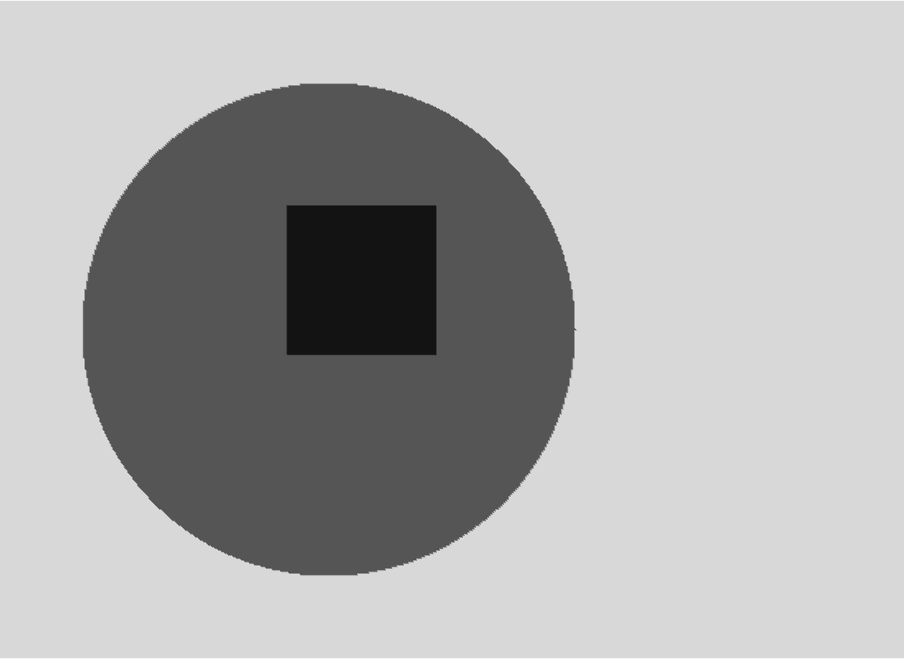}
		\caption{$\varphi_1$, $\lambda_1 \approx 1.5$}
		\label{fig:asd.sets.ef.eig.ef1}
	\end{subfigure}
	\hfill
	\begin{subfigure}[c]{0.32\textwidth}
		\includegraphics[width=1\textwidth]{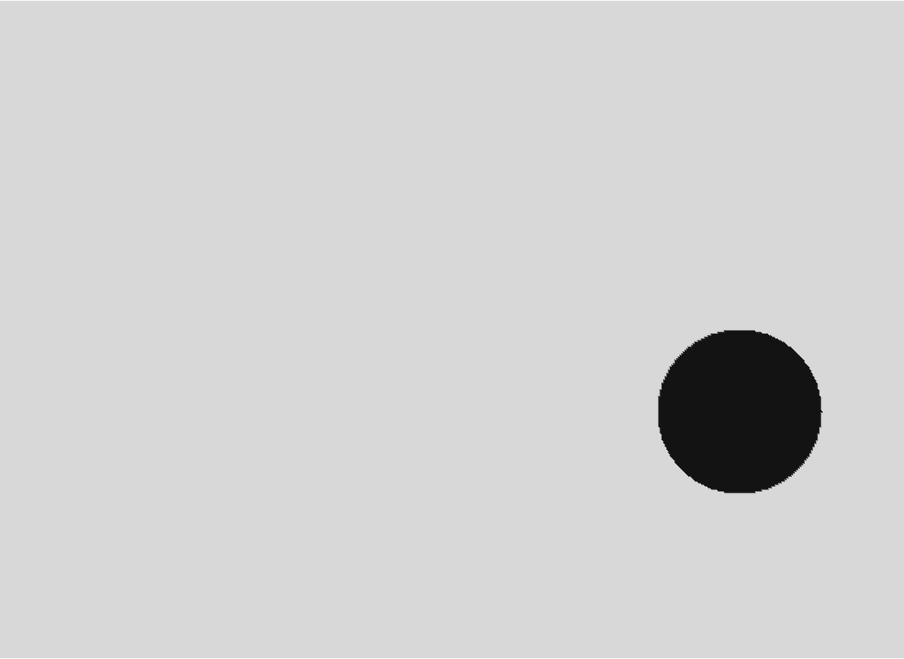}
		\caption{$\varphi_2$, $\lambda_2 \approx 4.7$}
		\label{fig:asd.sets.ef.eig.ef2}
	\end{subfigure}
	\hfill
	\begin{subfigure}[c]{0.32\textwidth}
		\includegraphics[width=1\textwidth]{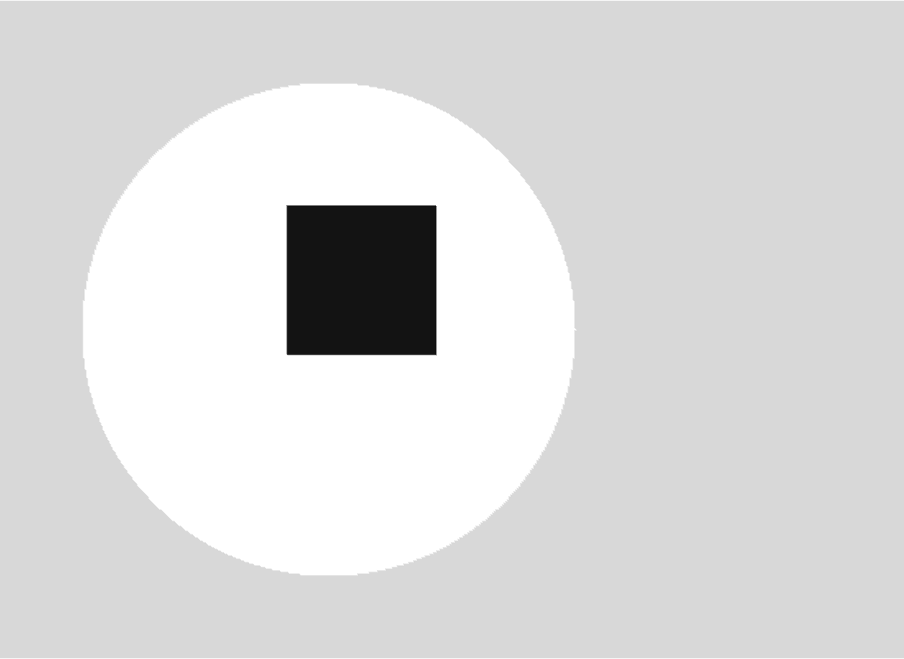}
		\caption{$\varphi_3$, $\lambda_3 \approx 5.3$}
		\label{fig:asd.sets.ef.eig.ef3}
	\end{subfigure}
	\\[2ex]
	\begin{subfigure}[c]{0.32\textwidth}
		\includegraphics[width=1\textwidth]{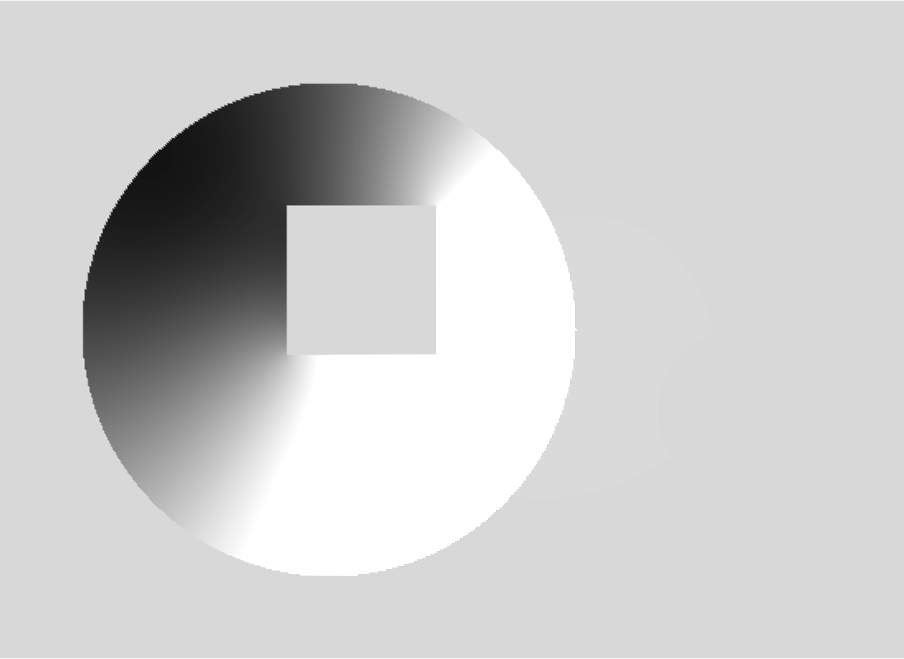}
		\caption{$\varphi_4$, $\lambda_4 \approx 1.0 \cdot 10^{8}$}
		\label{fig:asd.sets.ef.eig.ef4}
	\end{subfigure}
	\hfill
	\begin{subfigure}[c]{0.32\textwidth}
		\includegraphics[width=1\textwidth]{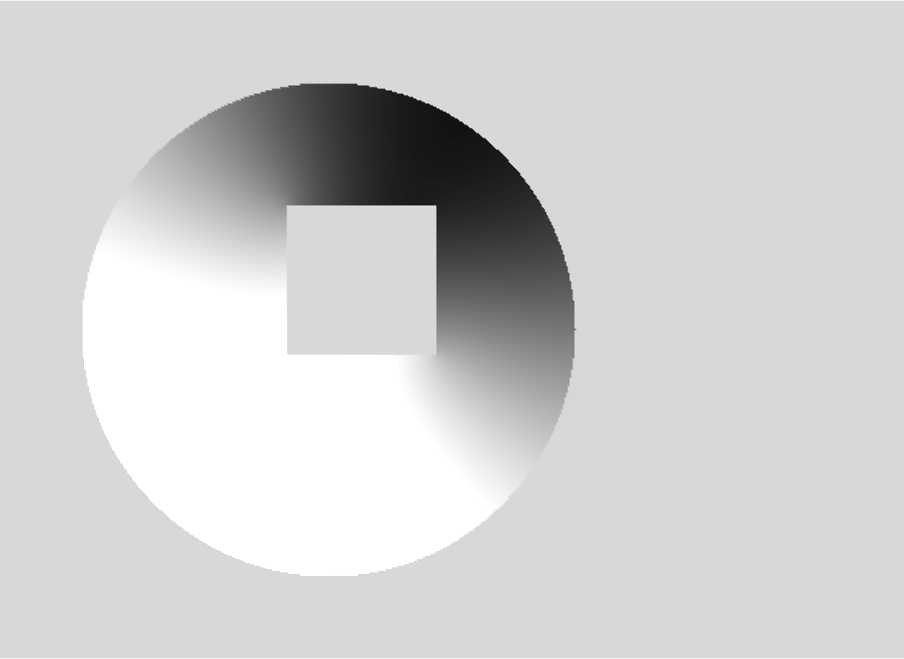}
		\caption{$\varphi_5$, $\lambda_5 \approx 1.1 \cdot 10^{8}$}
		\label{fig:asd.sets.ef.eig.ef5}
	\end{subfigure}
	\hfill\begin{subfigure}[c]{0.32\textwidth}
		\includegraphics[width=1\textwidth]{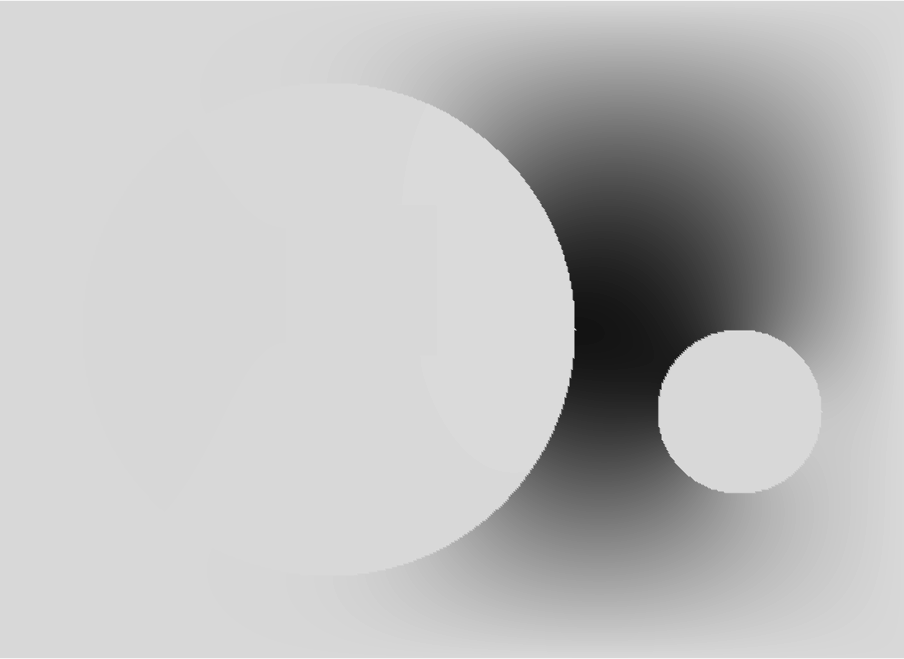}
		\caption{$\varphi_6$, $\lambda_6 \approx 1.3 \cdot 10^{8}$}
		\label{fig:asd.sets.ef.eig.ef6}
	\end{subfigure}
	\caption{AS decomposition:
			 The first six eigenfunctions $\varphi_k$, $k = 1, \ldots, 6$ of $L_\varepsilon[u_h]$, with its corresponding eigenvalues $\lambda_k$ satisfying \eqref{eq:as.eigenfunctions} for $u_h$ as in Figure \ref{fig:asd.sets.ef.eig.medium}.}
	\label{fig:asd.sets.ef.eig}
\end{figure}

For $K \geq 1$,  we let $\Phi_K = \Span\{\varphi_1,\ldots,\varphi_K\}$ and denote by $\Pi_K[u_h]$ the $L^2$-projection into $\Phi_K$:
\begin{align}\label{eq:as.decomposition}
	\Pi_K[u_h]: L^2(\Omega) \to \IR, \qquad (v - \Pi_K[u_h]v,\varphi)_{L^2(\Omega)} = 0 \qquad\forall\varphi\in\Phi_K, \forall v\in L^2(\Omega).
\end{align}
We call
\begin{align}
	\Pi_K[u_h]u = \sum_{k=1}^{K} \beta_k \varphi_k, \qquad \beta_k\in\IR
\end{align}
the (truncated) \emph{Adaptive Spectral (AS) decomposition}  of $u$.
Note that the operator $L_\varepsilon[u_h]$ depends on $u_h$ (and $\varepsilon$), and hence so do its eigenfunctions as well as $\Pi_K[u_h]$.

To illustrate the AS decomposition, consider the piecewise constant function $u:\Omega\subset\IR^2 \to \IR$ shown in Figure \ref{fig:asd.sets.ef.eig.medium}.
Next, we (numerically) solve the eigenvalue problem \eqref{eq:as.eigenfunctions} with $\varepsilon = 10^{-8}$ using a $\mathbb{P}^1$-FE discretization with mesh-size $h = 0.01$.
As shown in Figure \ref{fig:asd.sets.ef.eig}, the first three eigenfunctions $\varphi_k$, $k=1,2,3$ capture the inclusions rather well, whereas the subsequent eigenfunctions resemble eigenfunctions of the Laplacian with eigenvalues that scale as $1/\varepsilon$.
Next, we project $u$ into $\Phi_3 = \Span\{\varphi_1,\varphi_2,\varphi_3\}$.
As shown in Figure \ref{fig:asd.sets.ef.eig.asd}, the projection $\Pi_3[u_h]u$ is hardly distinguishable from $u$.

In \cite{baffet2022error,baffet2021adaptive}, rigorous $L^2$-error bounds were derived for the AS decomposition.
More precisely, consider a piecewise constant function with zero boundary values of the form
\begin{align}\label{eq:u.decomposed}
	u = \sum_{k=1}^K \alpha_k \chi_{A_k}, \qquad \alpha_k \in \IR\smallsetminus\{0\},
\end{align}
where $\chi_{A_k}$ is the characteristic function of Lipschitz domains $A_k \subset\subset \Omega$ with mutually disjoint and connected boundaries $\partial A_k$.

Next we partition $\Omega = \overline{M_h \cup D_h}$ into the two open sets
\begin{align}\label{eq:asd.sets}
	M_h = \bigcup_{k=1}^K \left\{x \in \Omega : \dist(x, \partial A_k) < h \right\},
	\qquad
	D_h = \Omega \smallsetminus \overline{M_h},
\end{align}
where $M_h$ is the open tube with diameter $h > 0$ around the jump discontinuities of $u$.
In Figure \ref{fig:asd.sets}, for instance, the sets $A_k$, $M_h$, and $D_h$ are shown for $u$ as in Figure \ref{fig:asd.sets.ef.eig.medium}.
\begin{figure}[ht!]
	\begin{subfigure}[c]{0.45\textwidth}
		\centering
		\includegraphics[width=1\textwidth]{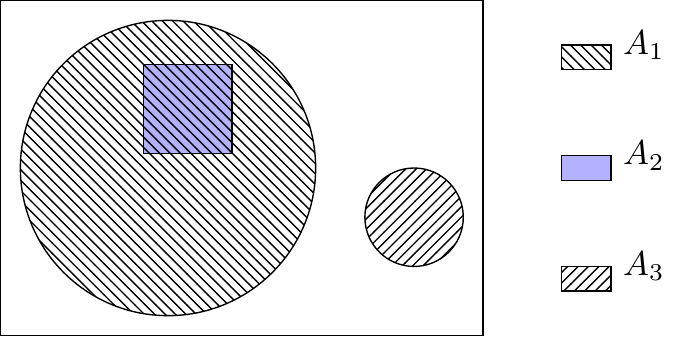}
	\end{subfigure}
	\hfill	
	\begin{subfigure}[c]{0.45\textwidth}
		\centering
		\includegraphics[width=1\textwidth]{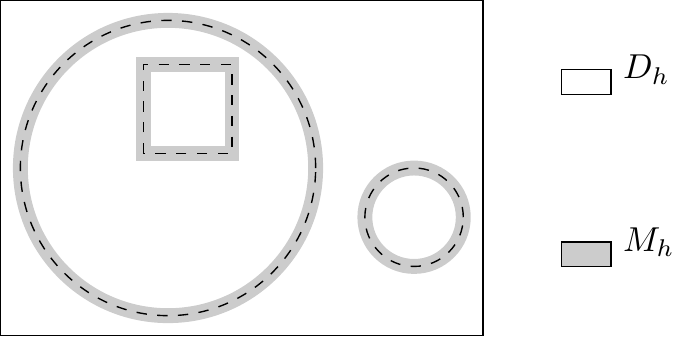}
	\end{subfigure}
	\caption{AS decomposition:
		The sets $A_k$, $k = 1,2,3$, $M_h$, and $D_h$ from \eqref{eq:asd.sets} for $u$ as in Figure \ref{fig:asd.sets.ef.eig.asd}.}
	\label{fig:asd.sets}
\end{figure}
Clearly, the FE-interpolant $u_h \in V_h$ of $u$ satisfies $u_h = u$ and $\nabla u_h = 0$ in $D_h$.
Moreover, for a sequence of regular and quasi-uniform meshes $(\mathcal{T}_h)_{h > 0}$, we have
\begin{align}\label{eq:admin.approx.properties}
	\lim_{h \to 0} \|u_h - u\|_{L^2(\Omega)} = 0,
	\qquad
	\nabla u_h \in L^\infty,
	\qquad
	\supp(\nabla u_h) \subset \overline{M_h},
\end{align}
and $h \|u_h\|_{L^\infty} \leq C$ with $C > 0$ independent of $h$, \cite[Proposition 2.2]{baffet2022error}.
In \cite{baffet2021adaptive}, those properties led to the following upper bound for the first $K$ eigenfunctions.
\begin{theorem}\label{thm:estimate.eigenfunctions}
Let $u$ be given by \eqref{eq:u.decomposed}, $u_h \in V_h$ be its FE interpolant, and $\varphi_k$, $k \geq 1$ the eigenfunctions of the AS operator $L_\varepsilon[u_h]$.
Then, for every $\varepsilon > 0$ and $h > 0$ sufficiently small, there exists a constant $C > 0$, independent of $\varepsilon$ and $h$, such that 
\begin{align}\label{eq:estimate.eigenfunctions}
	\|\nabla \varphi_k \|_{L^2(D_h)} \leq C \sqrt{\varepsilon}, \qquad k = 1, \ldots, K.
\end{align}
\end{theorem}
From Theorem \ref{thm:estimate.eigenfunctions} we conclude that the first $K$ eigenfunctions of $L_\varepsilon[u_h]$ are almost piecewise constant in $D_h$, that is away from any discontinuities, when $u$ is piecewise constant with $K$ inclusions.
This indicates that the AS decomposition \eqref{eq:as.decomposition} is able to approximate $u$ well throughout $\Omega$, which was proved in \cite{baffet2022error}:
\begin{theorem}\label{thm:estimate.asd}
Let $u$ be given by \eqref{eq:u.decomposed}, $u_h \in V_h$ be its FE interpolant, $\varphi_k$, $k \geq 1$ the eigenfunctions of the operator $L_\varepsilon[u_h]$ and $\Pi_K[u_h]u$ be its projection.
Then, for every $\varepsilon > 0$ and $h > 0$ sufficiently small, there exists a constant $C > 0$, independent of $\varepsilon$ and $h$, but possibly depending on $u$, such that 
\begin{align}
	\|u - \Pi_K[u_h]u\|_{L^2(\Omega)} \leq C \sqrt{h + \varepsilon}.
\end{align}
\end{theorem}
In summary, the AS decomposition $\Pi_K[u_h]u$ of $u$ approximates $u$ arbitrarily well, when $u$ is piecewise constant and consists of $K$ inclusions.
Note that more general versions of Theorem \ref{thm:estimate.eigenfunctions} and Theorem \ref{thm:estimate.asd} were proved in \cite[Theorem 5]{baffet2021adaptive} and \cite[Theorem 3.6]{baffet2022error}, which also apply to piecewise constant functions with nonzero boundary values.

Finally, we remark that the operator $L_\varepsilon[u]$ corresponds to the Fréchet derivative of the regularized TV-functional, $\operatorname{TV}_\varepsilon(u) = \int_\Omega\sqrt{|\nabla u|^2 + \varepsilon^2}$, which reduces to the standard TV-functional $\operatorname{TV}(u) = \int_\Omega |\nabla u|$ for $\varepsilon = 0$ -- see \cite[Remark 1]{grote2017adaptive}.
In fact, for $v \in H_0^1$, $v \approx u_h$, the ``energy-norm'' associated to the operator $L_\varepsilon[u_h]$ essentially corresponds to
the TV-``energy'',
\begin{align}\label{eq:approx.tv.functional}
	(L_\varepsilon[u_h]v,v)_{L^2(\Omega)} = 
	\int_\Omega \mu_\varepsilon[u_h]|\nabla v|^2
		= \int_\Omega \frac{|\nabla v|^2}{\sqrt{|\nabla u_h|^2 + \varepsilon^2}}
		\approx \int_\Omega |\nabla v|
		= \operatorname{TV}(v).
\end{align}
The AS decomposition also bears a remarkable resemblance to the spectral decomposition of the nonlinear $\operatorname{TV}$-functional \cite{burger2016spectral, gilboa2016nonlinear}.
\section{Adaptive Spectral Inversion}\label{sec:asi}
We shall now combine the simple Adaptive Inversion iteration from Section \ref{sec:giaa} with the AS decomposition \eqref{eq:as.decomposition} from Section \ref{sec:asd}.
Hence, at each iteration, we shall determine the new search space $\Psi^{(m+1)}$ from the current minimizer $u^{(m),\delta} \in \Psi^{(m)}$ using the eigenfunctions $\varphi_k$ of the elliptic operator $L_\varepsilon[u^{(m),\delta}]$ in \eqref{eq:as.eigenfunctions}.
Here, we recall from Theorems \ref{thm:estimate.eigenfunctions} and \ref{thm:estimate.asd} in Section \ref{sec:asd} that a piecewise constant medium $u$ with $K$ inclusions is approximated with high accuracy by the first $K$ eigenfunctions of $L_\varepsilon[u_h]$ -- 
see also \cite{baffet2022error, baffet2021adaptive} for further details.
 
First, we merge the current search space $\Psi^{(m)}$ with those eigenfunctions and reduce its dimension, while ensuring that $u^{(m),\delta}$ is still well represented in the merged space.
This first step corresponds to the algorithm previously used in \cite{baffet2021adaptive} -- See Remark \ref{rem:asi0} below for further details.

Next, to incorporate the new angle condition \eqref{eq:angle.condition} required for Theorem \ref{thm:convergence.gradient}, we shall include in $\Psi^{(m+1)}$ the most \emph{sensitive} eigenfunctions $\varphi_k$ reordered according to their \emph{sensitivities}:
\begin{align}\label{eq:sensitivities}
	|\sigma_1|
	\geq |\sigma_2|
	\geq \ldots
	\geq |\sigma_k|
	\geq \ldots \ ,
	\qquad \sigma_k = D\misfit^\delta(u^{(m),\delta}) \varphi_k.
\end{align}
We shall now construct a subspace
$\Phi_{N_\theta} = \Span\{\varphi_1, \ldots, \varphi_{N_\theta}\} \subset L^2(\Omega)$, which contains at least one $d \in \Phi_{N_\theta}$ that satisfies the angle condition \eqref{eq:angle.condition}, for a fixed $0 < \tolAngle < 1$.

\begin{lemma}\label{lem:linfinity.criterion}
Let $N_\infty$ be the largest index such that 
\begin{align}\label{eq:linfinity.criterion}
	|D\misfit^\delta(u^{(m),\delta})\varphi_k| \geq \tolAngle \|D\misfit^\delta(u^{(m),\delta})\|,\qquad k = 1, \ldots, N_\infty,
\end{align}
holds.
Then, $d = \sum_{k=1}^{N_\infty} \sigma_k \varphi_k$ satisfies the angle condition \eqref{eq:angle.condition}.
\end{lemma}
\begin{proof}
Since the eigenfunctions $\varphi_k$ are orthonormal, we have $\|d\|_{L^2(\Omega)} = \|\sigma\|_{\ell^2}$ which yields
\begin{align}
	|D\misfit^\delta(u^{(m),\delta})d| &= \left|\sum_{k=1}^{N_\infty} \sigma_k D\misfit^\delta(u^{(m),\delta})\varphi_k \right|
		= \sum_{k=1}^{N_\infty} |D\misfit^\delta(u^{(m)})\varphi_k|^2 \\
		&\geq \tolAngle \|D\misfit^\delta(u^{(m),\delta})\| \sum_{k=1}^{N_\infty} |D\misfit^\delta(u^{(m),\delta})\varphi_k| \\ 
		&= \tolAngle \|D\misfit^\delta(u^{(m),\delta})\| \|\sigma\|_{\ell^1} 
		\geq \tolAngle \|D\misfit^\delta(u^{(m),\delta})\| \|\sigma\|_{\ell^2} \\
		&= \tolAngle  \|d\|_{L^2(\Omega)} \|D\misfit^\delta(u^{(m),\delta})\|.
\end{align}
Thus, the angle condition \eqref{eq:angle.condition} is satisfied.
\end{proof}
Note that for $\tolAngle <1$ there always exists $N_\infty \geq 1$ sufficiently large such that \eqref{eq:linfinity.criterion} is satisfied.

\begin{lemma}\label{lem:l2.criterion}
Let $N_2$, be the smallest index such that 
\begin{align}\label{eq:l2.criterion}
	\sqrt{\sum_{k=1}^{N_2} |D\misfit^\delta(u^{(m),\delta}) \varphi_k|^2} \geq \tolAngle \|D\misfit^\delta(u^{(m),\delta})\|
\end{align}
holds. Then, $d = \sum_{k=1}^{N_2} \sigma_k \varphi_k$ satisfies the angle condition \eqref{eq:angle.condition}.
\end{lemma}
\begin{proof}
Similarly, since the eigenfunctions $\varphi_k$ are orthonormal, we have
\begin{align}
	|D\misfit^\delta(u^{(m),\delta})d| &= \sum_{k=1}^{N_2} |D\misfit^\delta(u^{(m),\delta})\varphi_k|^2 \\
		&= \|d\|_{L_2(\Omega)}
			\sqrt{\sum_{k=1}^{N_2} |D\misfit^\delta(u^{(m),\delta})\varphi_k|^2} \\
		&\geq \tolAngle \|d\|_{L^2(\Omega)} \|D\misfit^\delta(u^{(m),\delta})\|,
\end{align}
because of \eqref{eq:l2.criterion}.
Hence, the angle condition \eqref{eq:angle.condition} is satisfied.
\end{proof}
We refer to \eqref{eq:linfinity.criterion} as the $\ell^\infty$-criterion and to \eqref{eq:l2.criterion} as the $\ell^2$-criterion.
If we let 
\begin{align}
\label{eq:NthetaI}
	N_\theta = \max\{N_\infty,N_2\},
\end{align}
we ensure that there exists at least one element in the subspace $\Phi_{N_\theta}$ that satisfies the newly introduced angle condition \eqref{eq:angle.condition}.
Clearly, \eqref{eq:linfinity.criterion} is more stringent than \eqref{eq:l2.criterion} and thus $N_2 \leq N_\infty$.
As we only compute  in practice a finite number of eigenfunctions $\varphi_k$, $1 \leq k \leq N$, ordered according to their eigenvalues, none might in fact satisfy \eqref{eq:linfinity.criterion}. Then we set $N_\infty = 0$ and thereby $N_\theta = N_2$ in \eqref{eq:NthetaI}.
In the unlikely case that no $N_2 \leq N$ satisfies \eqref{eq:l2.criterion}, one needs to increase $N$ until \eqref{eq:l2.criterion} holds to ensure the existence of $d \in \Phi_{N_\theta}$ which satisfies  \eqref{eq:angle.condition}.
This yields the full Adaptive Spectral Inversion (ASI) Algorithm below.
\begin{algorithm}[H]
\caption{Adaptive Spectral Inversion}
\label{algo:ASI}
\algsetup{indent=2em,linenosize=\tiny,linenodelimiter=.}
\begin{algorithmic}[1]
	\REQUIRE ~
		\\ initial guess $u^{(0),\delta}$, search space $\Psi^{(1)}$
		\\ set  $m = 1$, $\tau_0 \geq 1$
	\ENSURE reconstruction $u^{(m_\ast),\delta}$
	\WHILE{$\| D\misfit^\delta(u^{(m),\delta})\| \neq 0$} \label{algo:stop.criterion}
		\STATE \emph{Minimize} $\misfit^\delta$ in the current search space $\Psi^{(m)}$, $\dim \Psi^{(m)} = K_{m}$:
				\begin{align}
					\label{eq:minimizer.subproblem}
					u^{(m),\delta} = \argmin_{u \in \Psi^{(m)}} \misfit^\delta(u).
				\end{align}
				\\[1ex]
		\STATE \emph{Compute} $\tau_{m} = \frac{1}{\delta}\|F(u^{(m),\delta}) - y^\delta\|$.
				\\[1ex]
		\IF{$\tau_{m} \leq \tau_0$}
			\STATE Discrepancy principle \eqref{eq:discrepancy.principle} is satisfied,
				return $u^{(m_\ast),\delta} = u^{(m-1),\delta}$.
				\\[1ex]
		\ENDIF
		\STATE\label{algo:determine}
				\textit{Determine} AS basis $\Phi^{(m+1)}$ from $L_\varepsilon[u^{(m),\delta}]$ via \eqref{eq:as.eigenfunctions}.
				\\[1ex]
		\STATE\label{algo:merge}
				\textit{Merge} AS basis and current search space: $\hat{\Psi}^{(m+1)} = \Phi^{(m+1)} \cup \Psi^{(m)}$.
				\\[1ex]
		\STATE\label{algo:truncate}
				\textit{Truncate} $\hat{\Psi}^{(m+1)}$ while maintaining the accuracy of $u^{(m),\delta}$.
				This yields $\tilde{\Psi}^{(m+1)}$.
				\\[1ex]
		\STATE\label{algo:add.sensitivities}
				\textit{Add sensitivities.}
				Include AS basis functions $\varphi_k \in \tilde{\Phi}^{(m+1)} \subset \Phi^{(m+1)}$ with maximal $|\sigma_k|$ in \eqref{eq:sensitivities}.
				This yields the new search space $\Psi^{(m+1)} = \tilde{\Psi}^{(m+1)} \cup \tilde{\Phi}^{(m+1)}$ with $\dim \Psi^{(m+1)} = K_{m+1}$.
				\\[1ex]
		\STATE $m \gets m + 1$
	\ENDWHILE
\end{algorithmic}
\end{algorithm}
By combining the previous search space $\Psi^{(m)}$ with promising new search directions determined from the current iterate $u^{(m),\delta}$, Steps \ref{algo:determine} -- \ref{algo:truncate} construct a low-dimensional subspace $\tilde{\Psi}^{(m+1)}$ able to represent $u^{(m),\delta}$ up to a small error.
In Step \ref{algo:add.sensitivities}, we then add the basis functions which will probably contribute most to the reduction of the misfit.
Below, we discuss in detail Steps \ref{algo:determine} -- \ref{algo:add.sensitivities} of the ASI Algorithm \ref{algo:ASI}.
\paragraph{Step \ref{algo:determine}: Determine $\Phi^{(m+1)}$.}
Once the minimizer $u^{(m),\delta}$ of $\misfit^\delta$ in the current search space $\Psi^{(m)}$ of dimension $\dim \Psi^{(m)} = K_{m}$ has been found, we compute the first few eigenfunctions of the elliptic operator $L_\varepsilon[u^{(m),\delta}]$, that is, we numerically solve the (linear, symmetric and positive definite) eigenvalue problem \eqref{eq:as.eigenfunctions}, which yields the new AS basis $\Phi^{(m+1)} = \Span\left\{ \varphi_1, \ldots, \varphi_{K_{m}}\right\}$.
Here we arbitrarily set the dimension of $\Phi^{(m+1)}$ to that of $\Psi^{(m)}$, though any other sufficiently large number of eigenfunctions could be chosen.
\paragraph{Step \ref{algo:merge}: Merge.}
Next we merge the previous search space $\Psi^{(m)}$ with $\Phi^{(m+1)}$ as $\hat{\Psi}^{(m+1)} = \Psi^{(m)} \cup \Phi^{(m+1)}$ and compute an $L^2$-orthonormal basis $\hat{\psi}_k$, $k = 1, \ldots, \hat K_{m+1}$, of $\hat{\Psi}^{(m+1)}$ via modified Gram-Schmidt.
Since the dimension $\hat{K}_{m+1}$ of the merged subspace $\hat{\Psi}^{(m+1)}$ may be as large as $2 K_{m}$, we now need to truncate $\hat{\Psi}^{(m+1)}$ to keep the number of control variables small.
\paragraph{Step \ref{algo:truncate}: Truncate.}
To truncate $\hat{\Psi}^{(m+1)}$ and retain only those basis functions $\hat{\psi}_k$ essential for representing $u^{(m),\delta}$, we proceed in two steps.
First, we compute an indicator $v$ close to $u^{(m),\delta}$, but with minimal $\operatorname{TV}$-``energy'', to remove noise and preserve edges.
In doing so, we keep the computational costs low by minimizing the linearized $\operatorname{TV}$-functional \eqref{eq:approx.tv.functional}.
Thus, the indicator $v$ satisfies
\begin{align}\label{eq:indicator.tvnorm}
\begin{split}
	\min_{v \in \hat{\Psi}^{(m+1)}}\quad &\int_\Omega \mu_\varepsilon[u^{(m),\delta}] |\nabla v|^2 \\
	\text{s.t.} \quad & \|v - u^{(m),\delta}\|_{L^2(\Omega)} \leq \tolTrunc \|u^{(m),\delta}\|_{L^2(\Omega)},
\end{split}
\end{align}
for a prescribed truncation tolerance $\tolTrunc > 0$; typically, $\tolTrunc = 5\%$.
Since \eqref{eq:indicator.tvnorm} is a quadratic optimization problem with quadratic inequality constraints, computing $v$ is cheap.

Second, we reduce the dimension of $\hat{\Psi}^{(m+1)}$ by discarding all basis functions that do not contribute much to the indicator $v$.
Let $\gamma_k$ denote the Fourier coefficients $\gamma_k$ of $v$,
\begin{align}
	\gamma_k = (v,\hat{\psi}_k)_{L^2(\Omega)}, \qquad k = 1, \ldots, \hat{K}_{m+1},
\end{align}
sorted in decreasing order and sort the $L^2$-orthonormal basis functions $\hat{\psi}^{(m+1)}_k$ accordingly.
Next, we determine the index
\begin{align}
	N_0 = \min\left\{1 \leq K \leq \hat{K}_{m+1}:
	\; \sum_{k=K+1}^{\hat{K}_{m+1}} \gamma_k^2 \leq \tolTrunc^2 \|\gamma\|^2_{\ell^2} \right\},
	\label{eq:truncation.tolerance.truncate}
\end{align}
such that the relative $L^2$-error in the Fourier expansion truncated at $N_0$ is below $\tolTrunc$.
To avoid drastic changes in the dimension of the search space, we now calculate $\rho = N_0 / K_m$.
For prescribed $\rho_0,\rho_1$ such that $0 < \rho_0 \leq 1 \leq \rho_1$, typically $\rho_0 = 0.8$ and $\rho_1 = 1.2$, we choose the dimension $\tilde{K}_{m+1}$ of the truncated space $\tilde{\Psi}^{(m+1)}$ as follows:
\begin{enumerate}[(i)]
	\item If $\rho \in [\rho_0,\rho_1]$, set $\tilde{K}_{m+1} = N_0$.
	\item If $\rho < \rho_0$, the dimension decreases too fast. Set $\tilde{K}_{m+1} = \left\lceil \rho_0 K_{m} \right\rceil$ and halve  $\tolTrunc$ to avoid a rapid decrease in the dimension at the next iteration.
	\item\label{item:growth} If $\rho > \rho_1$, the dimension increases too fast. Set $\tilde{K}_{m+1} = \left\lceil \rho_1 K_{m} \right\rceil$ and double $\tolTrunc$ to avoid a rapid increase in the dimension at the next iteration.
\end{enumerate}
This yields the truncated space as $\tilde{\Psi}^{(m+1)} = \Span\left\{ \hat{\psi}_1, \ldots, \hat{\psi}_{\tilde{K}_{m+1}} \right\}$.
\paragraph{Step \ref{algo:add.sensitivities}: Add sensitivities.} 
Finally, we determine the most sensitive AS basis functions $\tilde{\Phi}^{(m+1)} = \Span\left\{\varphi_1, \ldots, \varphi_{N_{\theta}}\right\}$ from $\Phi^{(m+1)}$, reorderd according to their sensitivities as in 
\eqref{eq:sensitivities}. In doing so, we use
the $\ell^\infty$- and $\ell^2$-criteria \eqref{eq:linfinity.criterion} and \eqref{eq:l2.criterion} from Lemmas \ref{lem:linfinity.criterion} and \ref{lem:l2.criterion}, respectively, to determine $N_\theta$.
In the unlikely event that the AS space $\Phi^{(m+1)}$ from Step \ref{algo:determine} contains no eigenfunction that satisfies \eqref{eq:linfinity.criterion} or \eqref{eq:l2.criterion}, we can either increase the dimension of $\Phi^{(m+1)}$, or set $\tilde{\Phi}^{(m+1)} = \emptyset$ (and simply proceed).
By combining $\tilde{\Psi}^{(m+1)}$ and $\tilde{\Phi}^{(m+1)}$, we eventually obtain the subsequent search space
\begin{align}
	\Psi^{(m+1)} = \tilde{\Psi}^{(m+1)} \cup \tilde{\Phi}^{(m+1)} 
		= \Span\left\{ \psi_1, \ldots, \psi_{K_{m+1}}\right\}.
\end{align}

Note that the sensitivity based selection procedure in Step \ref{algo:add.sensitivities} 
 only ensures that the angle condition \eqref{eq:angle.condition} is satisfied by at least
 one element $d$ in $\tilde{\Phi}^{(m+1)} \subset \Psi^{(m+1)}$, but not necessarily
 by the defect 
 $d^{(m),\delta} = u^{(m),\delta} - u^{(m-1),\delta}$ itself. 
The latter, stronger condition would require verifying \eqref{eq:angle.condition} at every iteration, and possibly rejecting
the new minimizer of \eqref{eq:minimizer.subproblem} when not satisfied, while further increasing the
search space. Instead,  Lemmas \ref{lem:linfinity.criterion} and \ref{lem:l2.criterion} guarantee
that at least one element in the search space satisfies the angle condition, thereby 
making it rather likely that it will also be satisfied by the correction $d^{(m),\delta}$.

\begin{remark}\label{rem:asi0}
In the above Adaptive Spectral Inversion (ASI) Algorithm, Steps \ref{algo:stop.criterion} -- \ref{algo:truncate} correspond to the ASI Algorithm previously introduced in \cite{baffet2021adaptive}, yet with the added growth
control on $K_m$ in Step \ref{algo:truncate} (\ref{item:growth}). Step \ref{algo:add.sensitivities}, however, where further eigenfunctions are included based on their sensitivites \eqref{eq:sensitivities}, is new and will prove crucial for detecting small-scale features in the medium. Henceforth we denote by $\text{ASI}_{0}$ the above ASI Algorithm without Step \ref{algo:add.sensitivities}, similar to that from \cite{baffet2021adaptive}, to distinguish it from the present ASI Algorithm.  
\end{remark}
\section{Numerical Results}\label{sec:numerical.results}
To illustrate the accuracy and usefulness of the ASI Algorithm from Section \ref{sec:asi}, we shall now apply it to two inverse medium problems of the form:
Find $u \in H_1, y \in H_2$ that satisfy
\begin{align}
	\min_{u,y}\quad  &\frac{1}{2} \|y - y^\delta\|_{H_2}^2 \label{eq:num.full.misfit}\\
	\text{s.t.}\quad &A[u] y - f = 0,	\label{eq:num.constraint}
\end{align}
for a given source $f$ and noisy data $y^\delta \in H_2$ as in \eqref{eq:noise}.
Each inverse problem is governed by a distinct \emph{forward problem} \eqref{eq:num.constraint}
whose solution, for any given medium\footnote{We refer to a function $u:\IR^2 \to \IR$ as a medium.} $u$, is $y = A[u]^{-1}f$. Thus, we can eliminate the constraint \eqref{eq:num.constraint}, which leads to the equivalent unconstrained minimization problem:
Find $u^{\dagger,\delta} \in H_1$ such that
\begin{align}\label{eq:num.min.problem}
	u^{\dagger,\delta} = \argmin_{u \in H_1} \frac{1}{2}\|y[u] - y^\delta\|_{H_2}^2.
\end{align}

For each inverse problem, we shall attempt to recover separately the two different (unknown) media shown in Figure \ref{fig:InverseProblems.uDagger}.
The first consists of six discs, each with a different value and radius.
The second consists of three inclusions:
an open wedge with a sharp $90^\circ$ interior angle, a convex drop-like inclusion with a sharp tip, and a kite-shaped inclusion, which is non-convex with a smooth boundary.
Since the unknown medium $u^\dagger$ must be strictly bounded away from zero, we write $u^\dagger = 1 + u^\dagger_0$ for $u^\dagger_0$ with zero boundary.
Then, we apply the ASI Algorithm \ref{algo:ASI} to minimize the misfit
\begin{align}
	\min_{u} \misfit^\delta(1+u)
\end{align}
with $u$ equal to zero at the boundary.
\begin{figure}[ht!]
\begin{subfigure}{0.49\textwidth}
	\centering
		\includegraphics[width=0.8\textwidth]{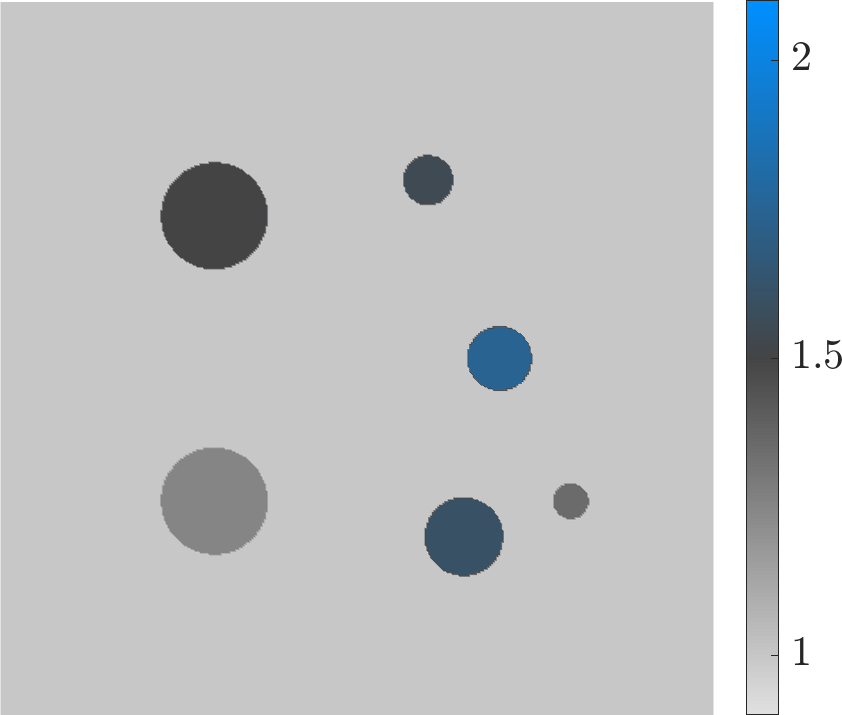}
	\end{subfigure}
	\hfill
	\begin{subfigure}{0.49\textwidth}
	\centering
		\includegraphics[width=0.8\textwidth]{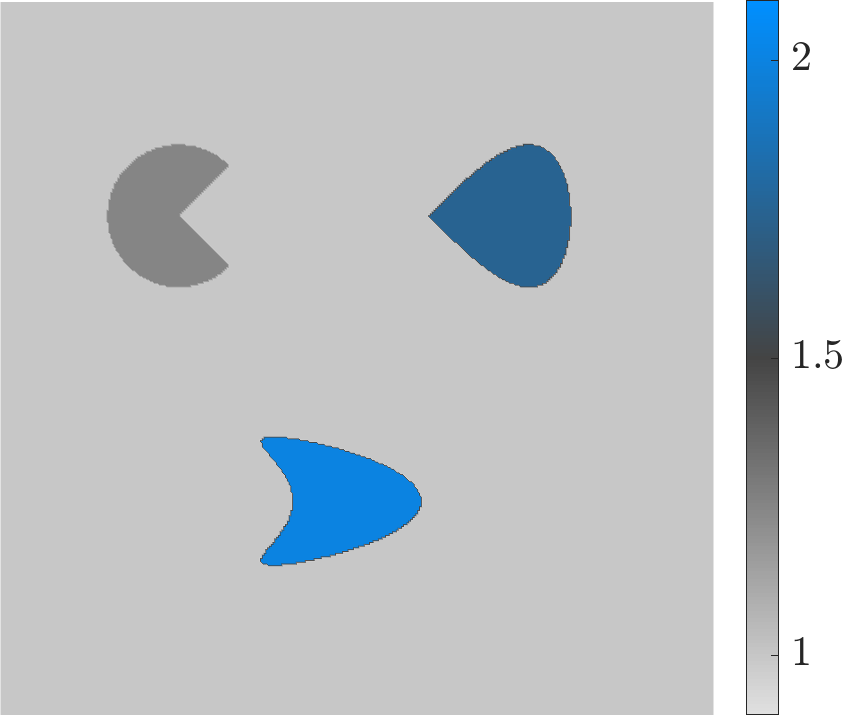}
	\end{subfigure}
	\caption{Inverse Problem: The two media $u^\dagger$ used in the numerical experiments.
				Left: six discs;
				right: three inclusions.}
	\label{fig:InverseProblems.uDagger}
\end{figure}

In all cases, we apply the ASI Algorithm \ref{algo:ASI} from Section \ref{sec:asi} with the following fixed parameter settings:
\begin{align}
	\rho_0 = 0.8, \quad \rho_1 = 1.2, \quad \tolAngle = 10^{-4}, \quad \varepsilon_\Psi = 0.05, \quad \tau_0 = 1.
\end{align}
As initial guess, we always choose $u^{(0),\delta} = 1$ constant and set the initial search space $\Psi^{(1)}$ equal to the first $K_1 = 50$ or $K_1 = 100$ $L^2$-orthonormal eigenfunctions of the Laplacian sorted in non-decreasing order w.r.t.\ their eigenvalues.

To assess the accuracy of the ASI method, we shall monitor the following quantities:
The dimension $K_m$ of the search space $\Psi^{(m)}$, the relative error 
\begin{align}\label{eq:rel.error}
	e_{m} = \frac{\|u^{(m),\delta} - u^\dagger\|_{H_1}}{\|u^\dagger\|_{H_1}},
\end{align}
the ratio $\tau_{m}$ from \eqref{eq:discrepancy.principle},
\begin{align}\label{eq:tau.discrepancy.principle}
	\tau_{m} = \frac{\|y[u^{(m),\delta}] - y^\delta\|_{H_2}}{\delta},
\end{align}
and the total number of iterations $m_\ast$, where $m = m_\ast+1$ is the first index such that $\tau_m \leq \tau_0$; hence, the discrepancy principle \eqref{eq:discrepancy.principle} is then satisfied with $\tau = \tau_{m_\ast}$.
\subsection{Elliptic Inverse Problem}\label{sec:IP.elliptic}
First, we consider as forward problem the elliptic differential equation
\begin{align}\label{eq:elliptic.forward.problem}
\begin{aligned}
	- \nabla \cdot (u \nabla y) &= f \qquad &&\text{in}\;\Omega, \\
	y &= 0 \qquad &&\text{on}\; \partial\Omega,
\end{aligned}
\end{align}
in the unit square $\Omega = (0,1)^2$ with constant right-hand side $f = 100$;
hence, the observations $y^\delta$ are available throughout $\Omega$.
In \cite{kaltenbacher2000regularization,kaltenbacher2008iterative,scherzer1993optimal}, the one-dimensional version of \eqref{eq:elliptic.forward.problem} was considered and shown to satisfy the Scherzer condition \eqref{eq:scherzer.condition}.

Here we shall compare the ASI Algorithm \ref{algo:ASI} from Section \ref{sec:asi} with a standard grid-based inversion method using Tikhonov regularization.
We initialize the ASI method with the constant initial guess $u^{(0),\delta} = 1$ and let $\Psi^{(1)}$ equal the first $K_{1} = 100$ Laplace eigenfunctions.
The forward problem \eqref{eq:elliptic.forward.problem} is solved 
in $H_1 = H_2 = V_h(\Omega)$, where $V_h$ corresponds to the subspace of 
continuous, piecewise linear $\mathbb{P}^1$ finite elements (FE) with mesh size $h>0$.
Here, we use 
for both $u$ and $y$ the same 
fixed triangular mesh with vertices located on a $400 \times 400$ equidistant Cartesian grid.
Despite the large number ($160'000$) of dof's in the FE representation, we recall that the ASI Algorithm determines the $m$-th iterate only in the much smaller subspace $\Psi^{(m)}$ of dimension $K_m$.
In the $m$-th step of the ASI Algorithm the minimizer $u^{(m),\delta}$ to \eqref{eq:minimizer.subproblem} is determined using standard BFGS together with Armijo \cite{nocedal2006numerical} instead of Wolfe-Powell line search to keep the number of (expensive) gradient evaluation small.

For the standard Tikhonov $L^2$-regularization approach, we minimize 
\begin{align}\label{ip:elliptic.tikhonov}
	\min_{u \in L^2(\Omega)} \frac{1}{2}\|y[u] - y^\delta\|_{L^2(\Omega)}^2 + \frac{\alpha}{2} \|u\|_{L^2(\Omega)}^2,
\end{align}
where $\alpha > 0$ denotes the regularization parameter.
Again, we discretize $u$ and $y$ with $\mathbb{P}^1$-FE on the same triangular mesh as for the ASI method, yet due to the resulting large number of unknowns, we now solve \eqref{ip:elliptic.tikhonov} using standard \emph{limited memory BFGS} (L-BFGS) \cite{nocedal2006numerical} together with Armijo line search.
Following \cite[Section $11.2$]{engl1996regularization}, we set the regularization parameter $\alpha_n = 2^{-n}$ at the $n$-th L-BFGS iteration.

To avoid any potential inverse crime, the exact data $y^\dagger = y[u^\dagger]$ was computed from a 
$20 \%$ finer mesh and perturbed at each grid point $x_j$ as
 \begin{align}
	y^\delta(x_j) = y^\dagger(x_j) + \hat\delta\, \eta_j, \qquad x_j \in \Omega,
\end{align}
for a \emph{noise level} $\hat\delta \geq 0$.
Here, $\eta_j$ corresponds to Gaussian noise normalized such that
the data misfit with $y^\dagger, y^\delta \in V_h(\Omega)$ satisfies exactly
\begin{align}
	\|y^\dagger - y^\delta\|_{L^2(\Omega)} = \hat\delta = \delta.
\end{align}
\begin{table}[ht!]
	\begin{center}
		\caption{Elliptic inverse problem, six discs:
				the relative error $e_{m_\ast}$, the total number of iterations $m_\ast$, and the dimension $K_{m_\ast}$ of the search space are shown for the ASI method and $L^2$-Tikhonov regularization.}
		\label{tab:elliptic.discs}
		\begin{tabular}{l||cccc|cccc}
			Method & \multicolumn{4}{c|}{ASI} & \multicolumn{4}{c}{$L^2$-regularization} \\\hline\hline
			$\delta = \hat\delta$ & $1\%$ & $2\%$ & $5\%$ & $10\%$ & $1\%$ & $2\%$ & $5\%$ & $10\%$ \\\hline
			$e_{m_\ast}$ & $2.7\%$ & $2.8\%$ & $3.9\%$ & $5.3\%$ & $5.3\%$ & $6.1\%$ & $8.0\%$ & $8.7\%$\\
			$m_\ast$ & $50$ & $19$ & $15$ & $10$ & $5.3$ & $39$ & $15$ & $11$\\
			$K_{m_\ast}$ & $67$ & $95$ & $76$ & $51$ & \multicolumn{4}{c}{$\mathrm{nDof} = 160'000$} \\
			$\tau_{m_\ast}$ & $1.005$ & $1.0001$ & $1.000$ & $1.0003$ & $1.000$ & $1.0008$ & $1.005$ & $1.001$
		\end{tabular}
	\end{center}	
\end{table}
In Table \ref{tab:elliptic.discs} we list the relative error $e_{m_\ast}$ at the final iteration $m_\ast$ for the ASI method and {$L^2$-Tikhonov} regularization together with the dimension $K_{m_\ast}$ of the search space $\Psi^{(m_\ast)}$.
As expected, both methods require fewer iterations $m_\ast$ to achieve \eqref{eq:discrepancy.principle} as the noise level $\delta$ increases, while the relative error obtained by the ASI method is always smaller than that achieved by standard $L^2$-Tikhonov regularization.
Even with $\delta = 10 \%$ the ASI method yields a smaller relative error than standard Tikhonov regularization using smaller $\delta$.
In Figure \ref{fig:elliptic.discs}, we compare the reconstructed media $u^{(m_\ast),\delta}$ for the two methods.
Clearly, the reconstruction obtained by the ASI method displays sharper contrasts and crisper edges while the coefficients inside each disc are more accurate.
Moreover, for $L^2$-Tikhonov regularization, the background appears noisy and the coefficients inside the inclusions are not accurately reconstructed.
Remarkably, the better reconstruction obtained by the ASI method is achieved with as few as $K_{m_\ast} = 100$ control variables only, in contrast to grid-based Tikhonov regularization with more than $160'000$ unknowns.

In Figure \ref{fig:elliptic.discs.data}, we observe that the relative error $e_{m}$ and $\|D\misfit^{\delta}(u^{(m),\delta})\|$ decrease throughout all iterations, as expected from Theorem \ref{thm:convergence.gradient} in Section \ref{sec:giaa}, while $\tau_{m}$ tends to $1$.
Note that the dimension $K_{m}$, or equivalently the number of control variables, remains small for the ASI method regardless of $\delta$, which keeps the computational effort low.
\begin{figure}[ht!]
	\begin{subfigure}{0.32\textwidth}
		\centering
		\includegraphics[width=1\textwidth]{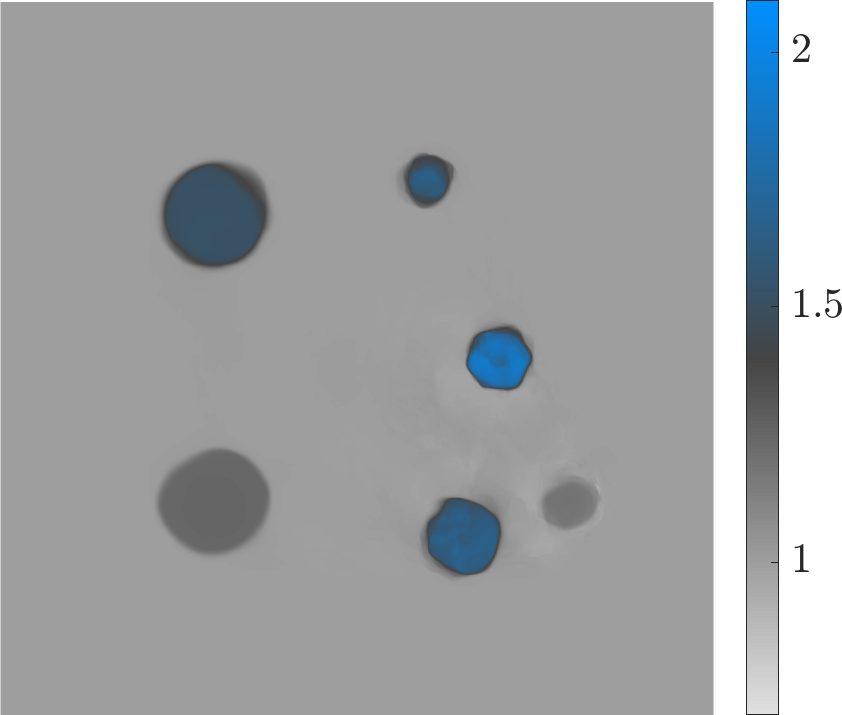}
		\caption{ASI: $u^{(m_\ast),\delta}$ for $\hat\delta = 1\%$}
	\end{subfigure}
	\hfill
	\begin{subfigure}{0.32\textwidth}
		\centering
		\includegraphics[width=1\textwidth]{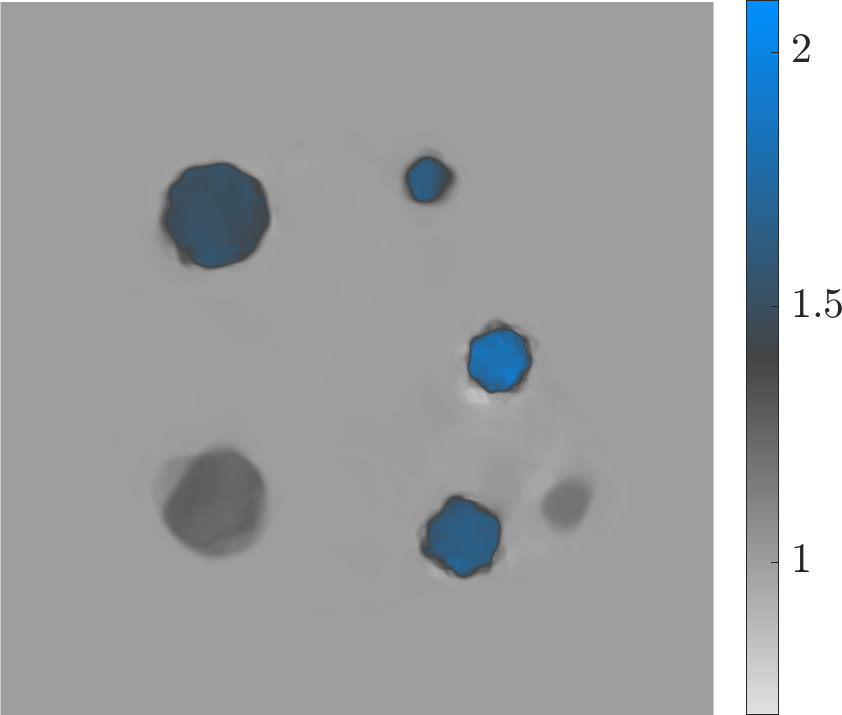}
		\caption{ASI: $u^{(m_\ast),\delta}$ for $\hat\delta = 2\%$}
	\end{subfigure}
	\hfill
	\begin{subfigure}{0.32\textwidth}
		\centering
		\includegraphics[width=1\textwidth]{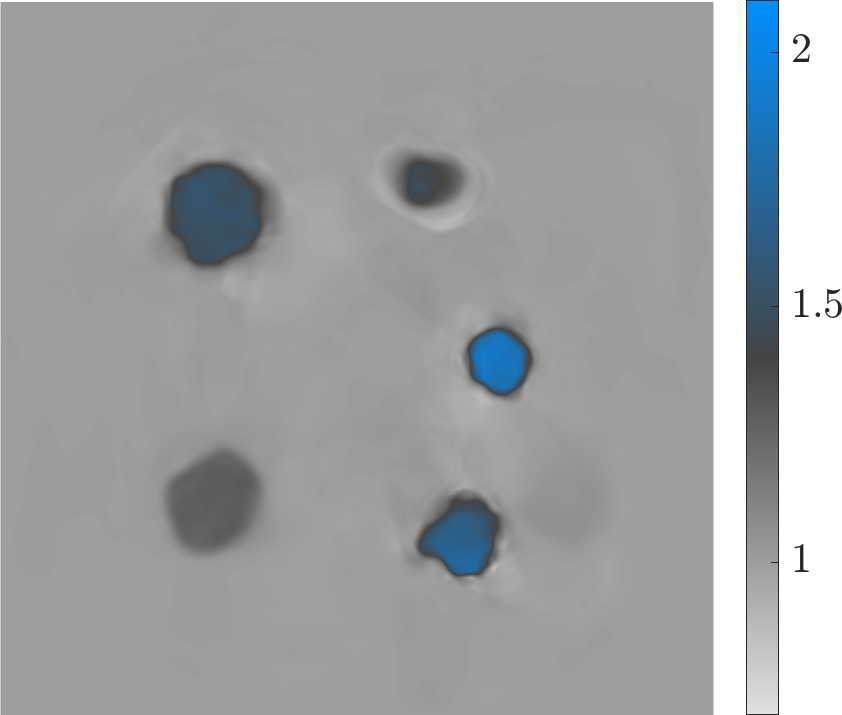}
		\caption{ASI: $u^{(m_\ast),\delta}$ for $\hat\delta = 5\%$}
	\end{subfigure}
	\\[2ex]
	\begin{subfigure}{0.32\textwidth}
		\centering
		\includegraphics[width=1\textwidth]{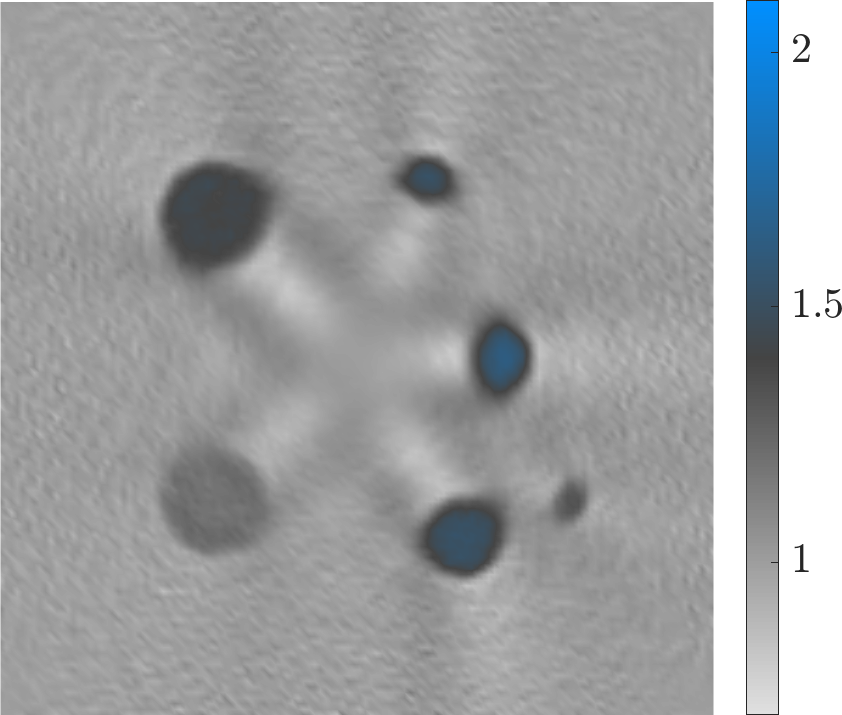}
		\caption{Tikhonov: $u^{(m_\ast),\delta}$ for $\hat\delta = 1\%$}
	\end{subfigure}
	\hfill
	\begin{subfigure}{0.32\textwidth}
		\centering
		\includegraphics[width=1\textwidth]{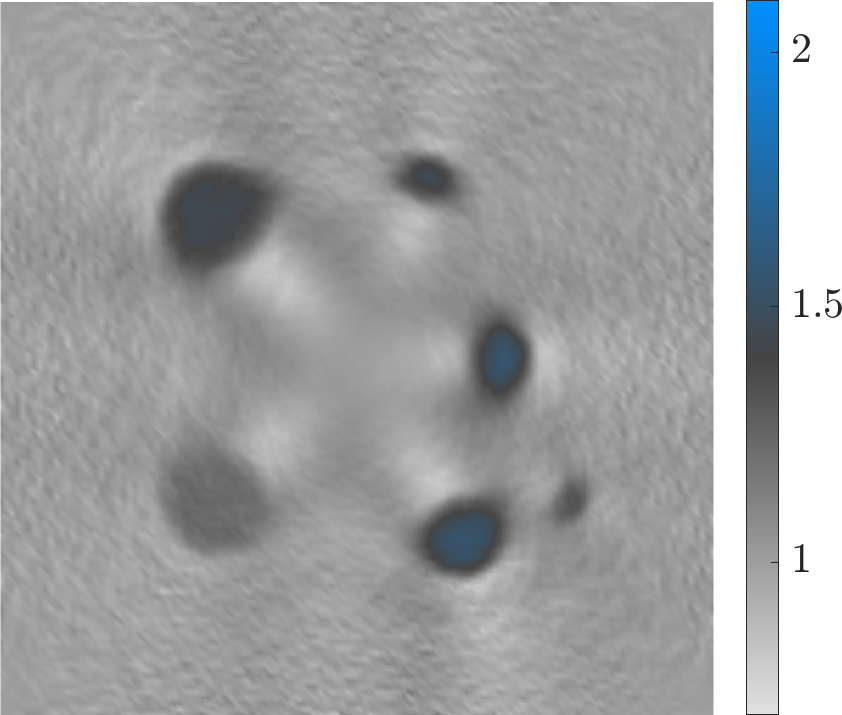}
		\caption{Tikhonov: $u^{(m_\ast),\delta}$ for $\hat\delta = 2\%$}
	\end{subfigure}
	\hfill
	\begin{subfigure}{0.32\textwidth}
		\centering
		\includegraphics[width=1\textwidth]{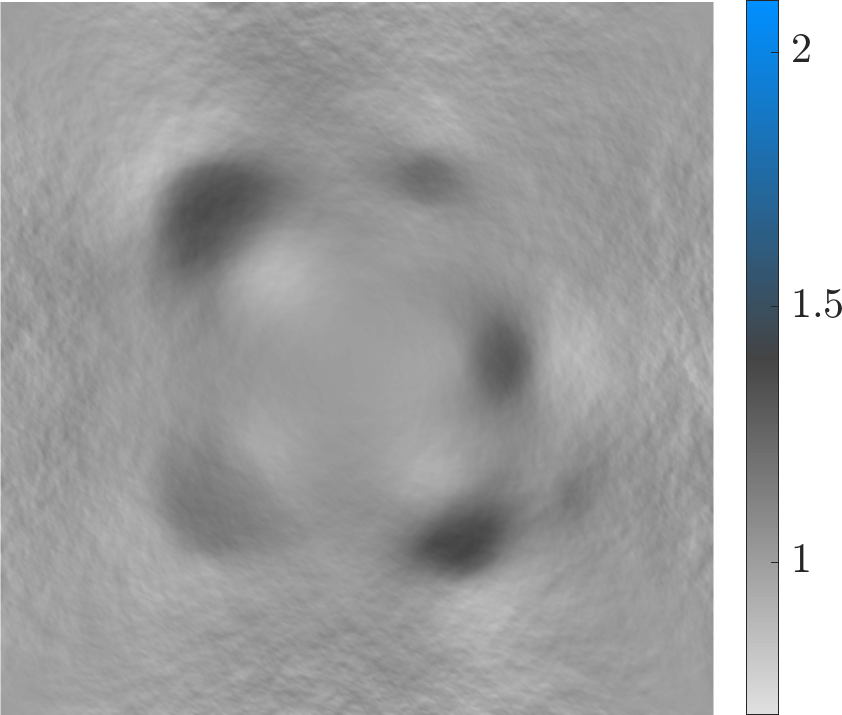}
		\caption{Tikhonov: $u^{(m_\ast),\delta}$ for $\hat\delta = 5\%$}
	\end{subfigure}
	\caption{Elliptic inverse problem, six discs:
				Reconstructed media using the ASI method (top) or standard $L^2$-Tikhonov regularization (bottom) for different noise levels $\hat\delta$.}
	\label{fig:elliptic.discs}
\end{figure}
\begin{figure}[ht!]
	\begin{subfigure}[c]{0.49\textwidth}
		\centering
		\includegraphics[width=0.8\textwidth]{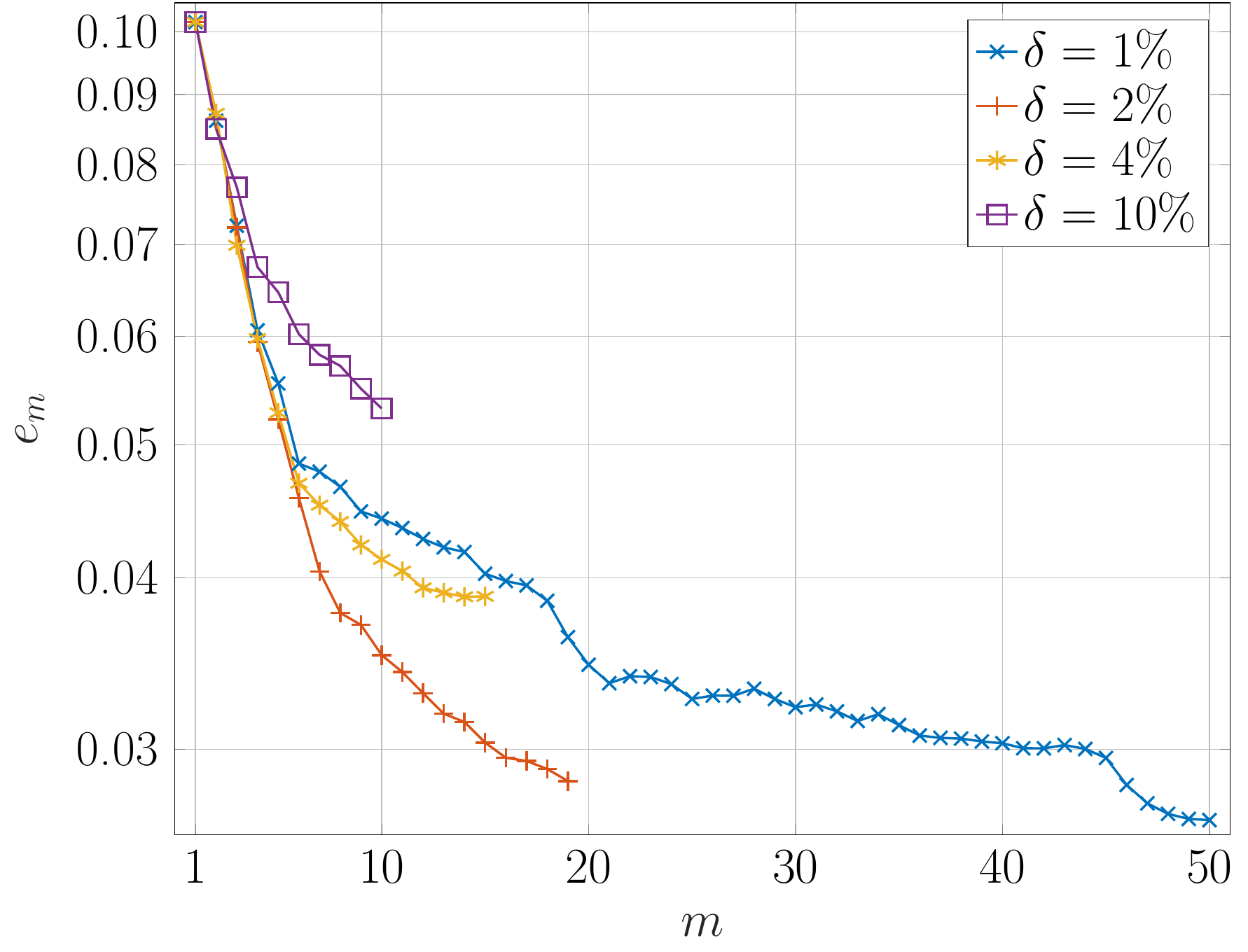}
		\caption{relative $L^2$ error $e_{m}$}
	\end{subfigure}
	\hfill
	\begin{subfigure}[c]{0.49\textwidth}
		\centering
		\includegraphics[width=0.8\textwidth]{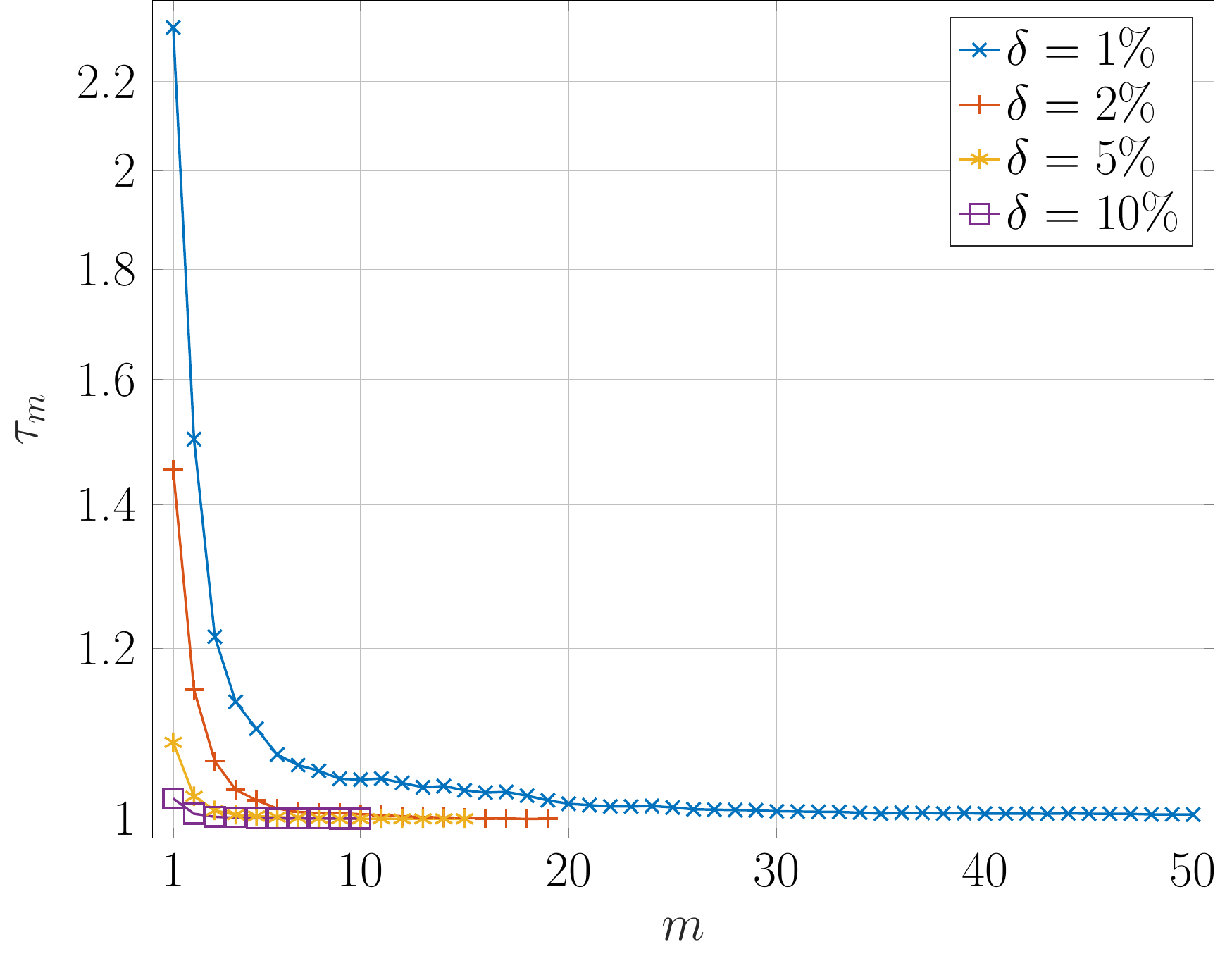}
		\caption{ratio $\tau_{m}$ from the discrepancy principle}
		\label{subfig:elliptic.discs.tau}
	\end{subfigure}
	\\[2ex]
	\begin{subfigure}[c]{0.49\textwidth}
		\centering
		\includegraphics[width=0.8\textwidth]{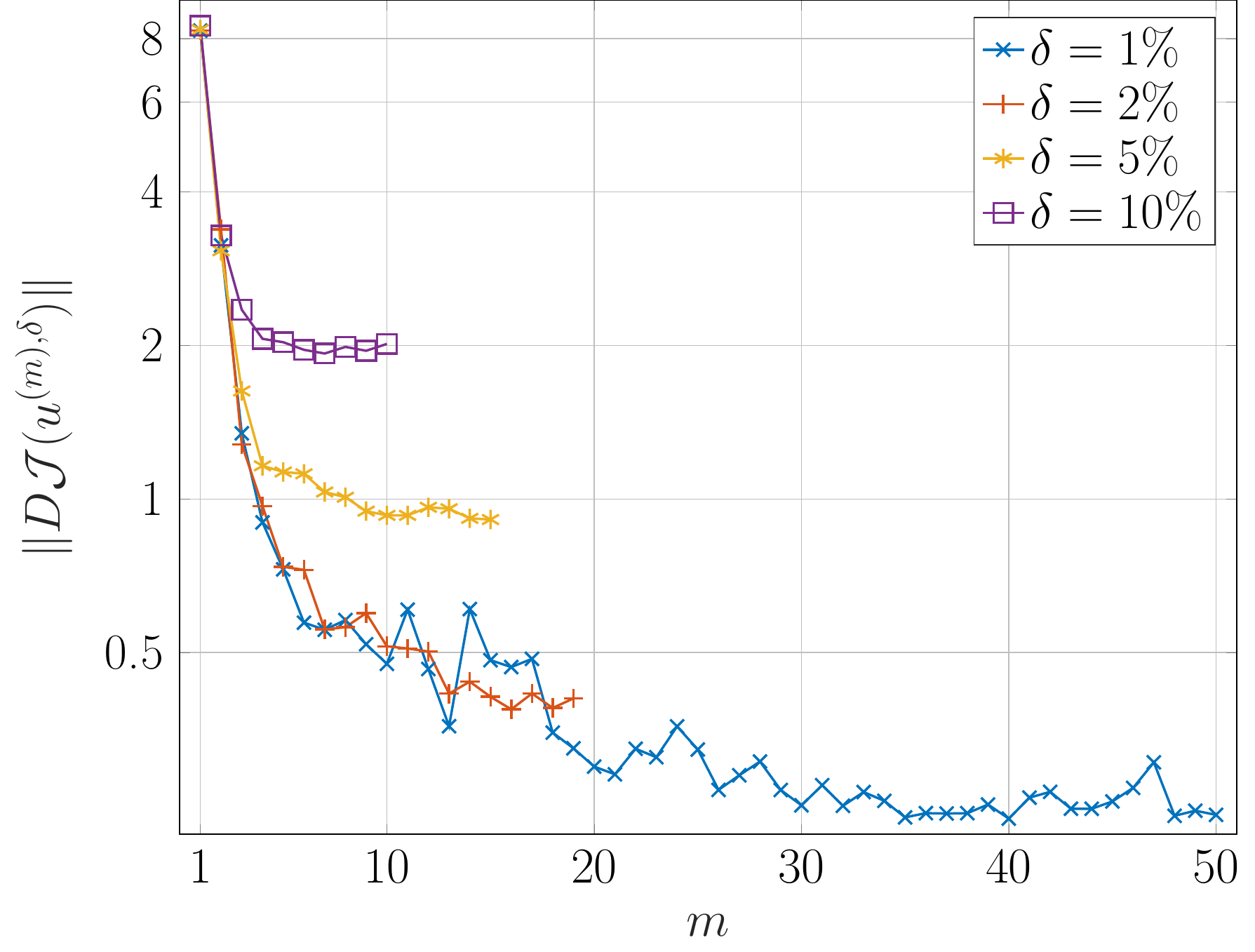}
		\caption{$\|D\misfit(u^{(m),\delta})\|$}
	\end{subfigure}
	\hfill
	\begin{subfigure}[c]{0.49\textwidth}
		\centering
		\includegraphics[width=0.8\textwidth]{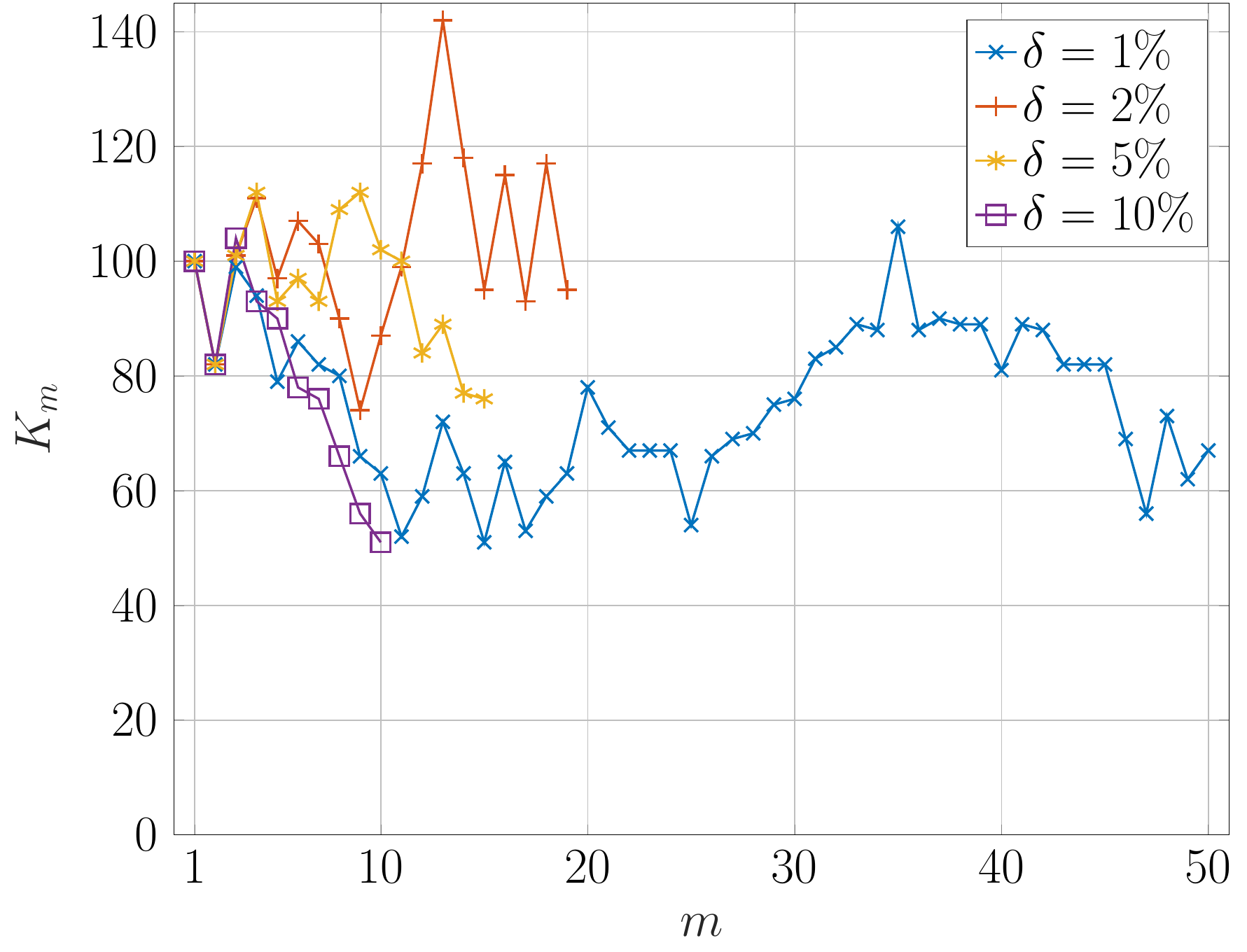}
		\caption{dimension $K_{m}$}
	\end{subfigure}
	\caption{Elliptic inverse problem, six discs:
				The relative error \eqref{eq:rel.error}, the ratio $\tau_{m}$ from the discrepancy principle \eqref{eq:tau.discrepancy.principle}, the norm of the gradient, and the dimension of the search space at iteration $m$ for different noise levels $\delta$.}
	\label{fig:elliptic.discs.data}
\end{figure}

Next, we consider the (unknown) medium $u^\dagger$ shown in the right frame of Figure \ref{fig:InverseProblems.uDagger}, which consists of three distinct inclusions.
Again, the ASI Algorithm from Section \ref{sec:asi} is able to recover the medium at the various noise levels $\delta$, as shown in Figure \ref{fig:elliptic.three.inclusions}:
all three inclusions are clearly visible with good contrast and sharp edges, except for the reentrant corner of the open wedge with $5\%$ noise.
$L^2$-Tikhonov regularization, however, results in more noisy and blurred reconstructions, where the inclusions become hardly visible beyond $5 \%$ noise.

From Table \ref{tab:elliptic.three.inclusions}, we again infer that the relative error $e_{m_\ast}$ of the ASI method alway remains below that obtained with $L^2$-Tikhonov regularization:
Even with $\hat\delta = 10\%$ noise, the ASI method remains more accurate than $L^2$-Tikhonov regularization with as little as $1\%$ noise.
Moreover, the number of control variables $K_{m_\ast}$ used in the ASI method never exceeds $160$, in comparison to approximately $160'000$ control variables used in the nodal FE representation for $L^2$-Tikhonov regularization.
As a consequence, the computational effort of the ASI method always remains well below that with standard Tikhonov regularization.

In Figure \ref{fig:elliptic.three.inclusions.data}, we observe that the relative error $e_{m}$ and $\|D\misfit^{\delta}(u^{(m),\delta})\|$ decrease with each iteration, while the ratio $\tau_{m}$ tends to $1$ and the number of basis functions $K_{m}$ of the search space $\Psi^{(m)}$, i.e.\ the number of control variables, remains small, which again translates to a low computational effort. Note that $\tau_{m} \approx 1$ implies that $u^{(m),\delta}$ (nearly) yields an optimal data misfit because $\|y^\dagger - y^\delta\|_{L_2(\Omega)} \leq \delta$.
\begin{table}[ht!]
	\begin{center}
		\caption{Elliptic inverse problem, three inclusions:
		the relative error $e_{m_\ast}$, the total number of iterations $m_\ast$, and the dimension $K_{m_\ast}$ of the search space are shown for the ASI method and $L^2$-Tikhonov regularization.}
		\label{tab:elliptic.three.inclusions}
		\begin{tabular}{l||cccc|cccc}
			Method & \multicolumn{4}{c|}{ASI} & \multicolumn{4}{c}{$L^2$-regularization} \\\hline\hline
			$\delta = \hat\delta$ & $1\%$ & $2\%$ & $5\%$ & $10\%$ & $1\%$ & $2\%$ & $5\%$ & $10\%$ \\\hline
			$e_{m_\ast}$ & $3.3\%$ & $3.7\%$ & $5.0\%$ & $6.4\%$ & $7\%$ & $8.3\%$ & $12.9\%$ & $13.0\%$ \\
			$m_\ast$ & $50$ & $25$ & $12$ & $8$ & $69$ & $48$ & $15$ & $15$ \\
			$K_{m_\ast}$ & $153$ & $120$ & $110$ & $72$ & \multicolumn{4}{c}{$\mathrm{nDof} = 160'000$} \\
			$\tau_{m_\ast}$ & $1.001$ & $1.0001$ & $1.0001$ & $1.000$ & $1.0004$ & $1.0001$ & $1.029$ & $1.0014$
		\end{tabular}
	\end{center}
\end{table}
\begin{figure}[ht!]
	\begin{subfigure}{0.32\textwidth}
		\centering
		\includegraphics[width=1\textwidth]{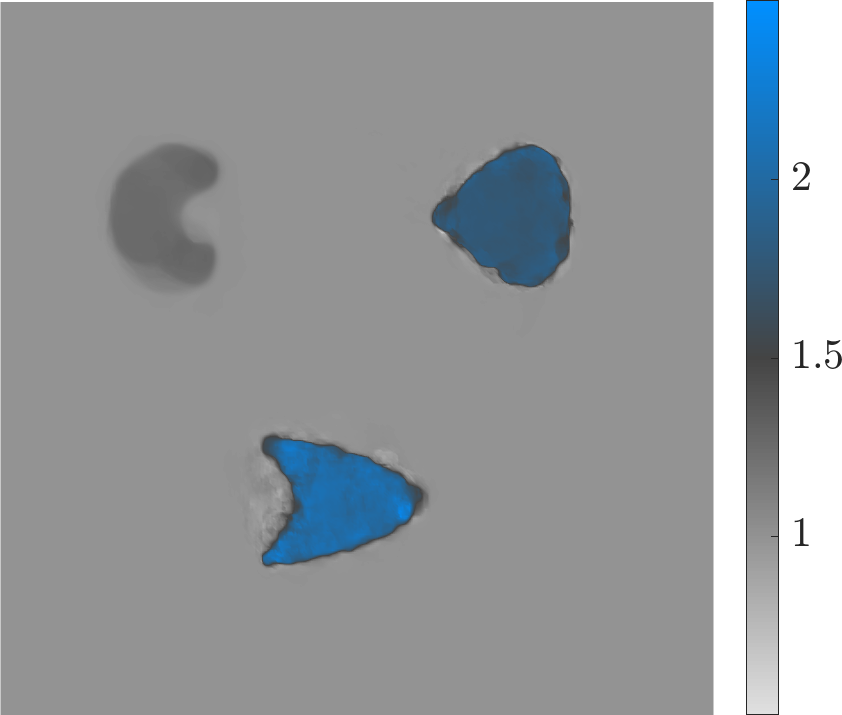}
		\caption{ASI: $u^{(m_\ast),\delta}$ for $\hat\delta = 1\%$}
	\end{subfigure}
	\hfill
	\begin{subfigure}{0.32\textwidth}
		\centering
		\includegraphics[width=1\textwidth]{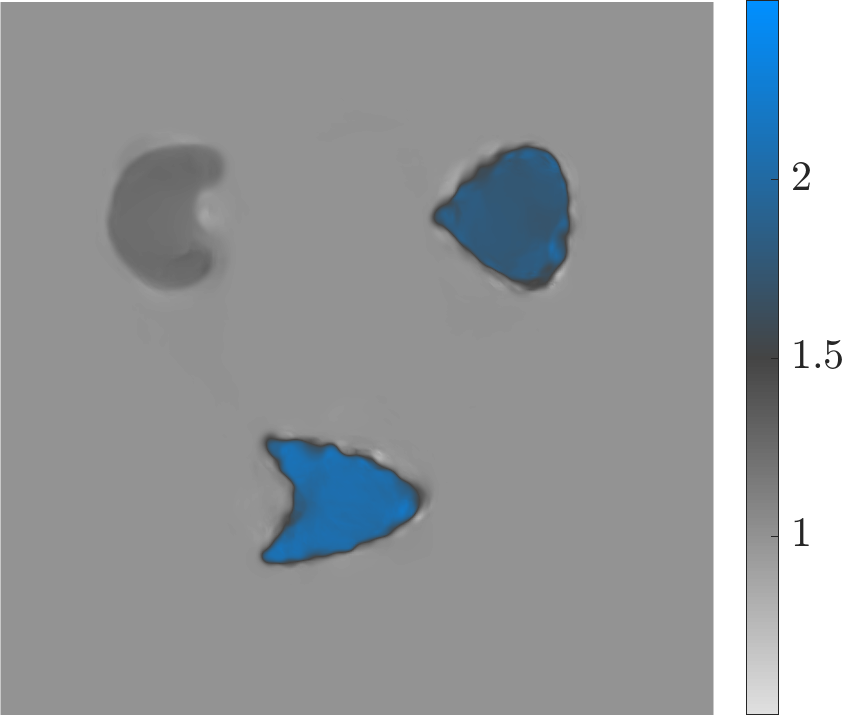}
		\caption{ASI: $u^{(m_\ast),\delta}$ for $\hat\delta = 2\%$}
	\end{subfigure}
	\hfill
	\begin{subfigure}{0.32\textwidth}
		\centering
		\includegraphics[width=1\textwidth]{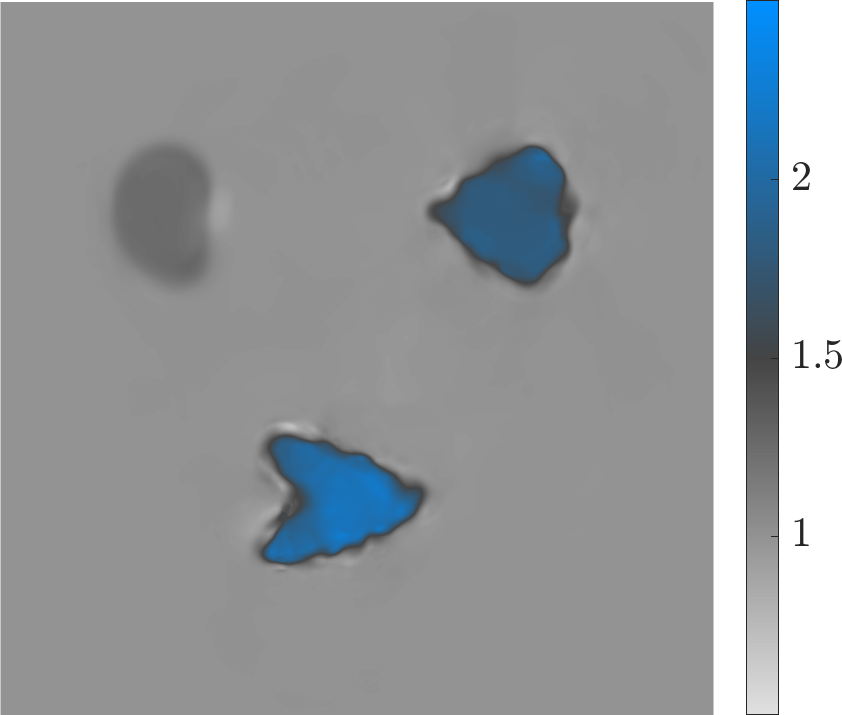}
		\caption{ASI: $u^{(m_\ast),\delta}$ for $\hat\delta = 5\%$}
	\end{subfigure}
	\\[2ex]
	\begin{subfigure}{0.32\textwidth}
		\centering
		\includegraphics[width=1\textwidth]{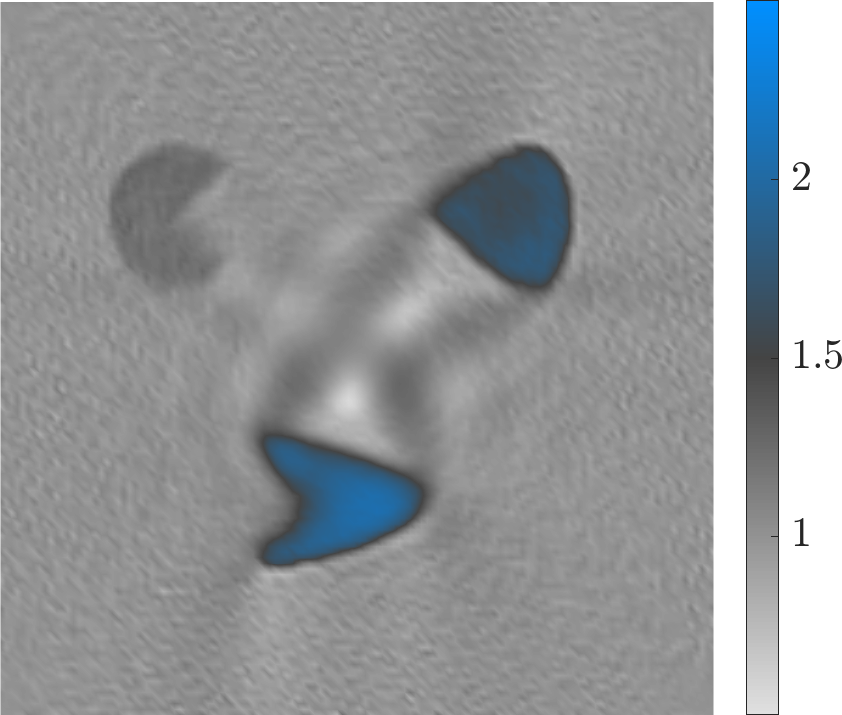}
		\caption{Tikhonov: $u^{(m_\ast),\delta}$ for $\hat\delta = 1\%$}
	\end{subfigure}
	\hfill
	\begin{subfigure}{0.32\textwidth}
		\centering
		\includegraphics[width=1\textwidth]{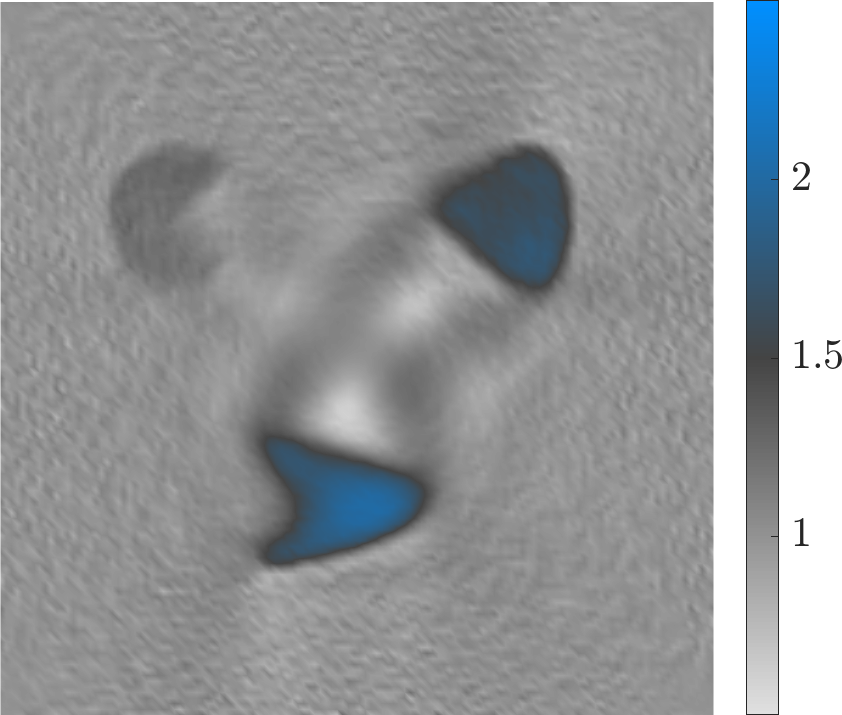}
		\caption{Tikhonov: $u^{(m_\ast),\delta}$ for $\hat\delta = 2\%$}
	\end{subfigure}
	\hfill
	\begin{subfigure}{0.32\textwidth}
		\centering
		\includegraphics[width=1\textwidth]{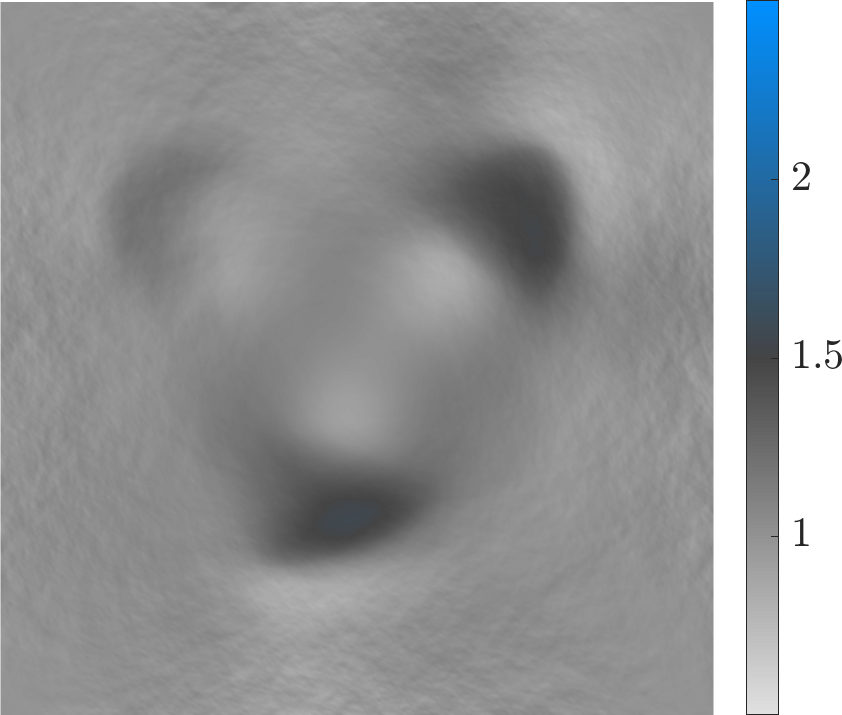}
		\caption{Tikhonov: $u^{(m_\ast),\delta}$ for $\hat\delta = 5\%$}
	\end{subfigure}
	\caption{Elliptic inverse problem, three inclusions:
				reconstructed medium using the ASI method (top) or standard $L^2$-Tikhonov regularization (bottom) for different noise levels $\hat\delta$.}
	\label{fig:elliptic.three.inclusions}
\end{figure}
\begin{figure}[ht!]
	\begin{subfigure}[c]{0.49\textwidth}
		\centering
		\includegraphics[width=0.8\textwidth]{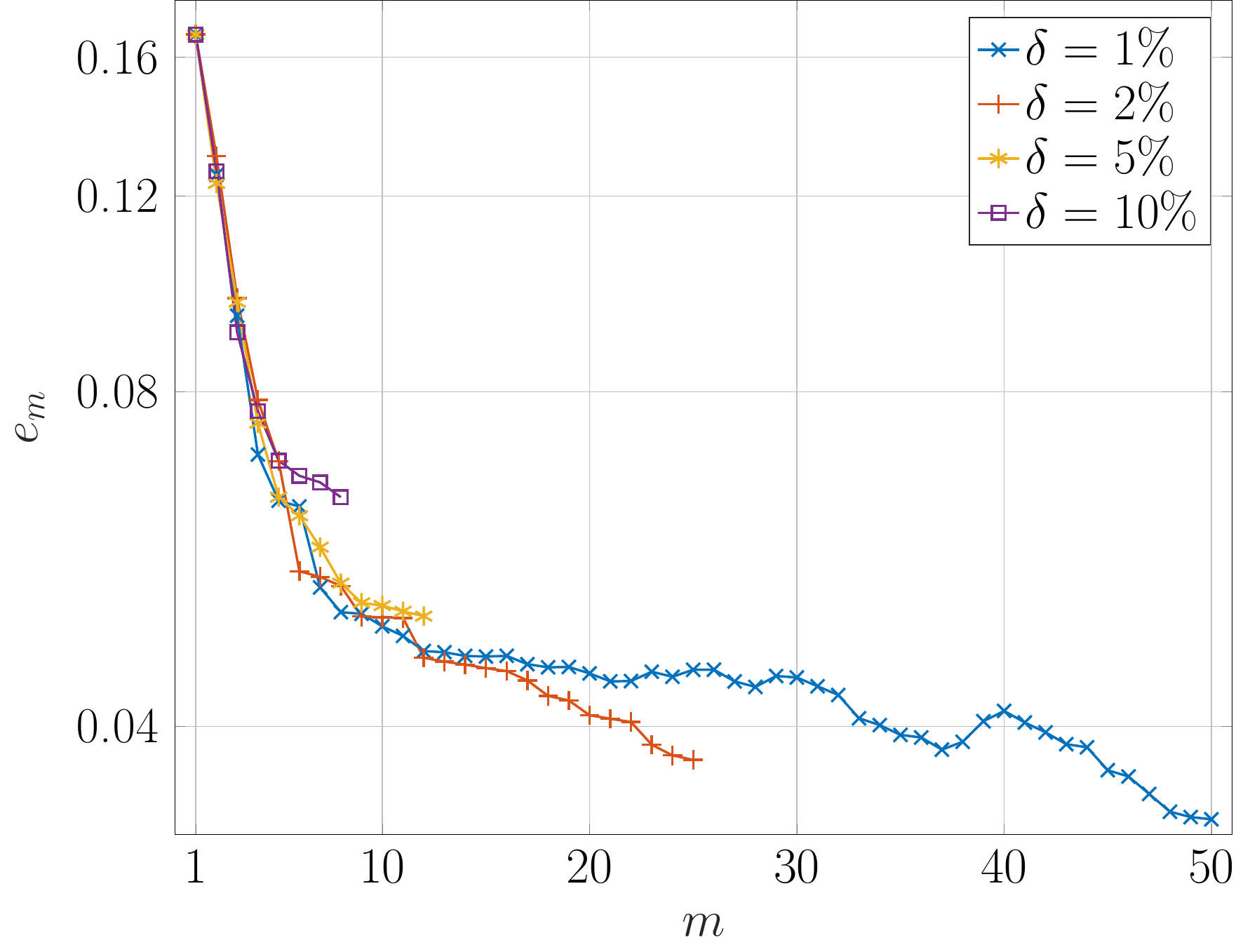}
		\caption{relative $L^2$ error $e_{m}$}
	\end{subfigure}
	\hfill
	\begin{subfigure}[c]{0.49\textwidth}
		\centering
		\includegraphics[width=0.8\textwidth]{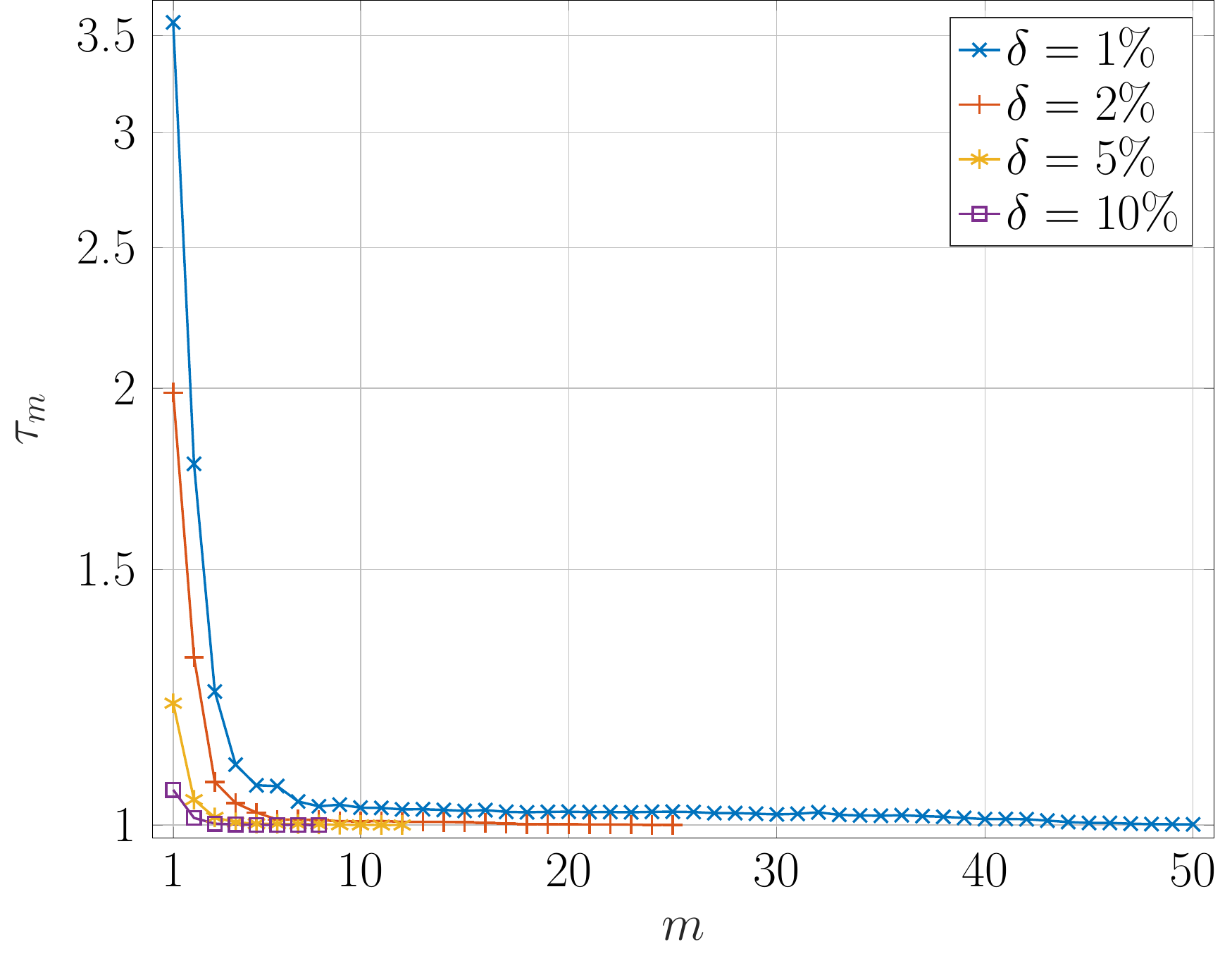}
		\caption{ratio $\tau_{m}$ from the discrepancy principle}
	\end{subfigure}
	\\[2ex]
	\begin{subfigure}[c]{0.49\textwidth}
		\centering
		\includegraphics[width=0.8\textwidth]{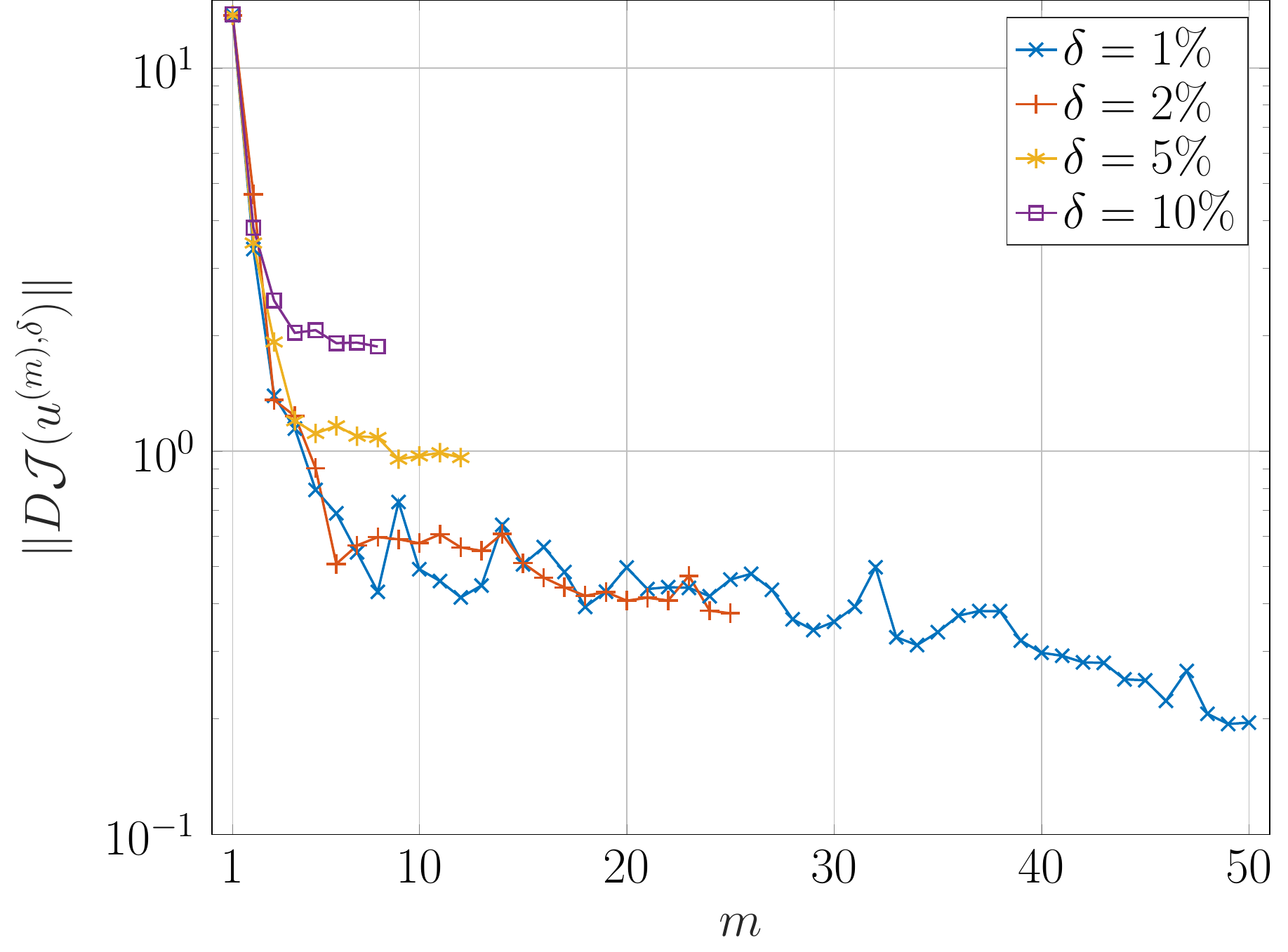}
		\caption{$\|D\misfit(u^{(m),\delta})\|$}
	\end{subfigure}
	\hfill
	\begin{subfigure}[c]{0.49\textwidth}
		\centering
		\includegraphics[width=0.8\textwidth]{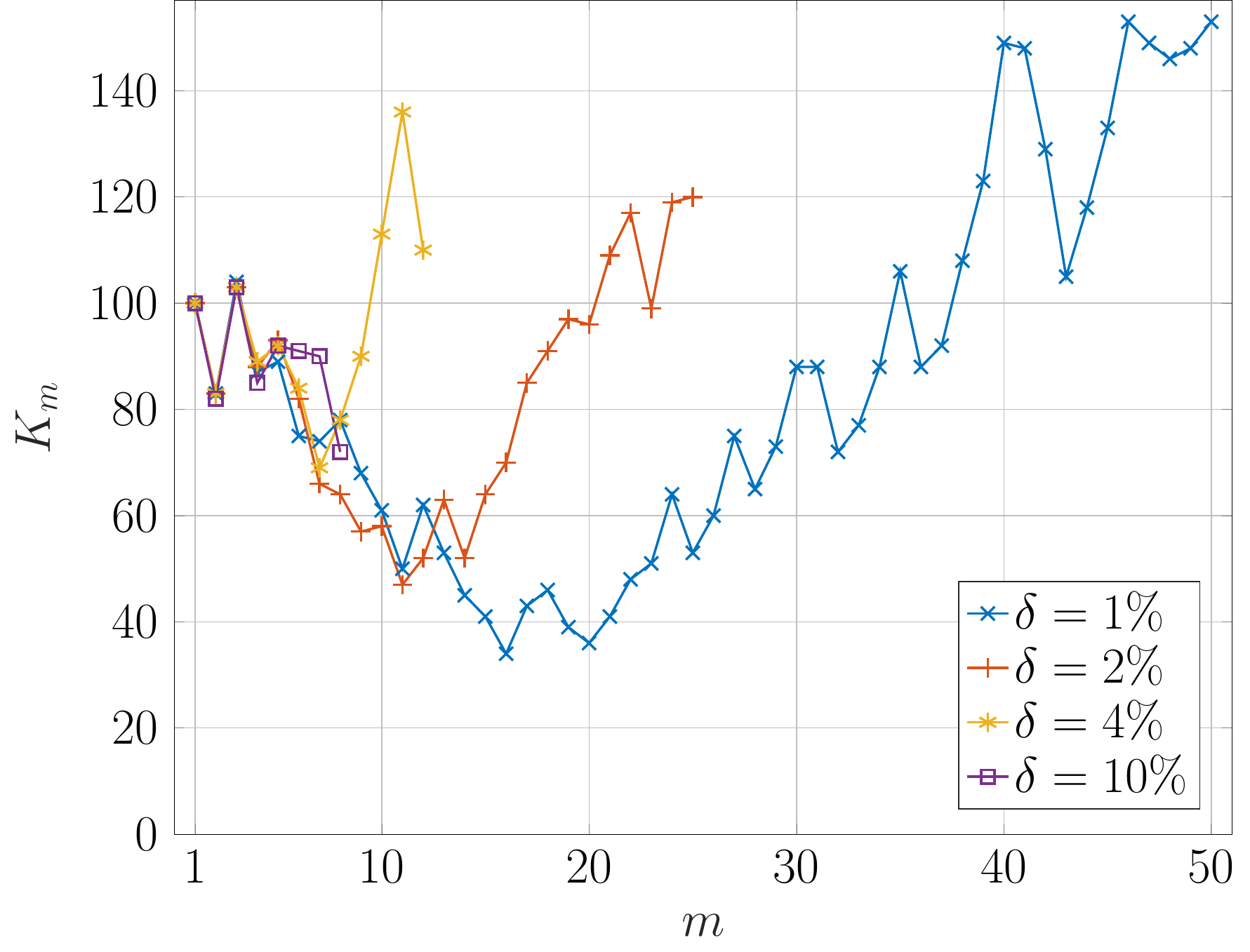}
		\caption{dimension $K_{m}$}
	\end{subfigure}
	\caption{Elliptic inverse problem, three inclusions:
				The relative error \eqref{eq:rel.error}, the ratio $\tau_{m}$ from the discrepancy principle \eqref{eq:tau.discrepancy.principle}, the norm of the gradient, and dimension of the search space, for every iteration $m$ for different noise levels $\hat\delta > 0$.}
	\label{fig:elliptic.three.inclusions.data}
\end{figure}
\subsection{Time Dependent Inverse Scattering Problem}\label{sec:IP.wave}
Next, we consider wave scattering from an unknown spatially distributed medium illuminated by surrounding point sources.
Hence, the forward problem \eqref{eq:num.constraint} now corresponds to the time-dependent wave equation in $\Omega = (0,1)^2$,
\begin{align}
	\label{eq:wave.equation}
	\begin{aligned}
		\frac{\partial^2}{\partial t^2} y_\ell(x,t) - \nabla \cdot \left( u(x) \nabla y_\ell(x,t) \right)
			&= f_\ell(x,t),  &&\quad x \in \Omega,\; t \in (0,T), \\
			 y_\ell(x,0) = \tfrac{\partial}{\partial t}y_\ell(x,0) &= 0, &&\quad x \in \Omega, \\
		\tfrac{\partial}{\partial t}y_\ell(x,t) + \sqrt{u(x)} \tfrac{\partial}{\partial n}y_\ell(x,t) &= 0, &&\quad x \in\partial\Omega,\; t\in (0,T),
	\end{aligned}
\end{align}
with homogeneous initial conditions and first-order absorbing boundary conditions.
Here, $u(x)$ denotes the squared wave speed whereas the sources $f_\ell(x,t) = g_\ell(x) r(t)$, $\ell = 1,\ldots,N_s$, correspond to smoothed Gaussian point sources
\begin{align}
	g_\ell(x) = \kappa e^{\frac{(x - x_\ell)^2}{s}},
	\qquad x \in \Omega,
\end{align}
in space, centered about distinct locations $x_\ell \in \Omega$, with $s = 10^{-2}$ and $\kappa = 200$, and a Ricker wavelet \cite{deBuhan2013aNewApproach,lines2005aTimeDomain} in time,
\begin{align}
	r(t) = (1 - 2 \pi^2 (\nu - t)^2) e^{- \pi^2 (\nu t - 1)^2},
	\qquad t \in [0,T],
\end{align}
with central frequency $\nu = 10$.
In Figure \ref{fig:wave.forward.problem}, snapshots of the solution to the forward problem \eqref{eq:wave.equation} are shown at different times for the source located at the top left corner.
\begin{figure}[ht!]
	\begin{subfigure}{0.32\textwidth}
		\centering
		\includegraphics[width=1\textwidth]{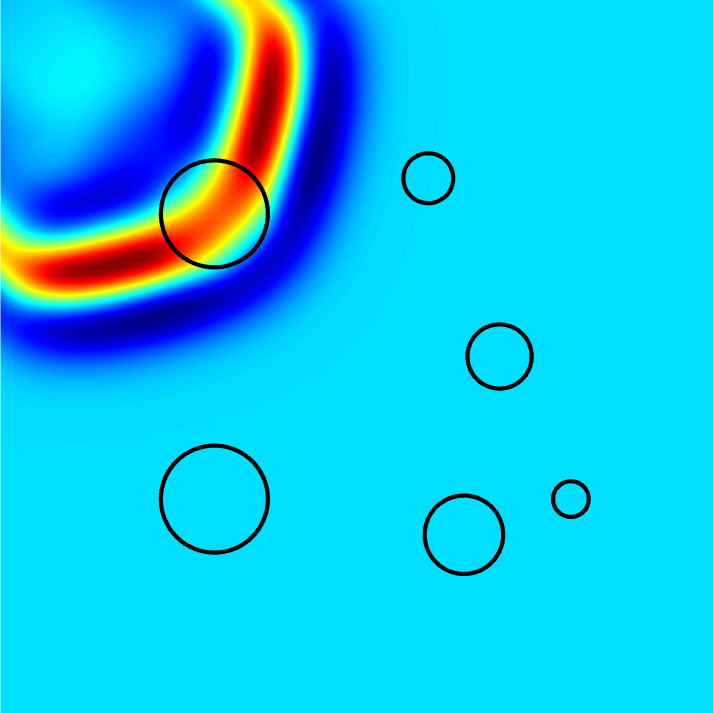}
		\caption{$t = 0.4$}
	\end{subfigure}
	\hfill
	\begin{subfigure}{0.32\textwidth}
		\centering
		\includegraphics[width=1\textwidth]{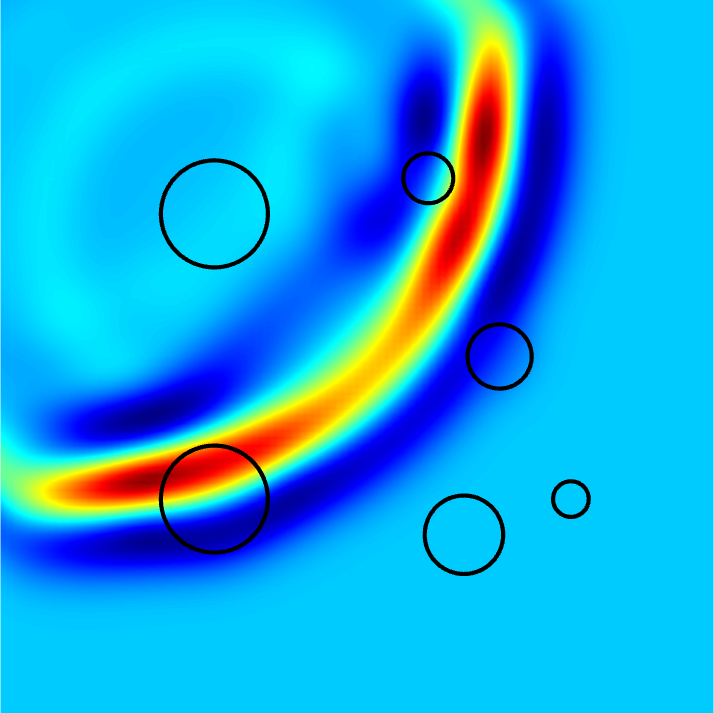}
		\caption{$t = 0.7$}
	\end{subfigure}
	\hfill
	\begin{subfigure}{0.32\textwidth}
		\centering
		\includegraphics[width=1\textwidth]{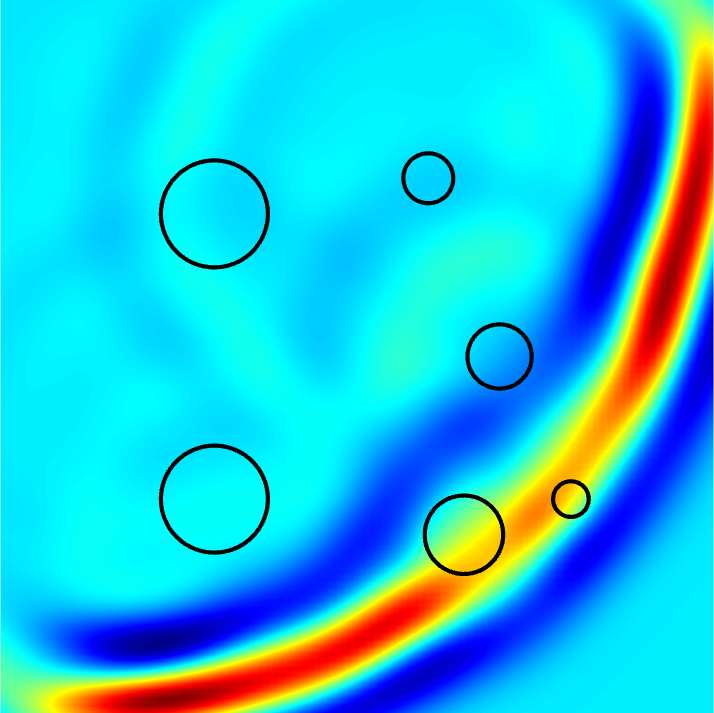}
		\caption{$t = 1.0$}
	\end{subfigure}
	\caption{Snapshots of the solution to the wave equation \eqref{eq:wave.equation}.}
	\label{fig:wave.forward.problem}
\end{figure}

The forward problem \eqref{eq:wave.equation} is solved by a standard Galerkin-FE method, where we again discretize $u(x)\in V_{h}$ using piecewise linear $\mathbb{P}^1$-FE on a triangular mesh with vertices located on a $400 \times 400$ equidistant Cartesian grid.
For $y(\cdot,t) \in \widetilde{V}_h$ in \eqref{eq:wave.equation}, however, we use quadratic $\mathbb{P}^2$-FE with mass-lumping \cite{cohen2001higher,mulder2001higher} on a separate triangular mesh with about $10$ elements per wavelength, resulting in approximately $60'000$ nodes.
For the time integration, we use the standard (fully explicit) leapfrog method with time-step $\Delta t \approx 4.5 \cdot 10^{-4}$.

To generate the (synthetic) observations, we now place $N_s = 32$ sources located at $x_\ell$ equidistributed near the boundary and illuminate the medium, one source at a time.
In contrast to the previous example from Section \ref{sec:IP.elliptic}, here the data $y^\delta_\ell$ is only available at the boundary nodes, yet for all discrete time-steps $t_n = n \Delta t$, $n = 0, 1, \ldots, N_T$, until the final time $T = 2$ when the incident wave has essentially left $\Omega$.
Hence, we set $H_1 = V_h(\Omega)$ and $H_2 = \widetilde{V}_h(\partial\Omega) \times \{t_n\}_{n = 0}^{N_T}$ in \eqref{eq:num.min.problem}, where the misfit
\begin{align}\label{eq:wave.misfit.all.sources}
	\misfit^\delta(u) =
	\frac{1}{2} \sum_{\ell = 1}^{N_s} \|y_\ell[u] - y_\ell^\delta\|_{L^2(\partial\Omega \times (0,T))}^2
\end{align}
now accounts for the data $y_\ell^\delta$ from multiple sources $\ell = 1, \ldots, N_s$.

To avoid any potential inverse crime, the exact data $y^\dagger_\ell = y_\ell[u^\dagger]$ was computed from a $20 \%$ finer mesh.
The perturbed (noisy) data $y_\ell^\delta\in H_2$ was then obtained at each boundary node $x_j$ and time-step
$t_n$ as
 \begin{align}
	y_\ell^\delta(x_j,t_n) = y_\ell^\dagger(x_j,t_n)\cdot
		(1 + \hat\delta\cdot \left(\eta_\ell)_{j,n}\right),
		\qquad (x_j,t_n) \in \partial\Omega \times (0,T),
\end{align}
for a \emph{noise level} $\hat\delta \geq 0$. Here, $(\eta_\ell)_{j,n}$ corresponds to normally distributed Gaussian noise with
\begin{align}
	\sum_{\ell = 1}^{N_s} \|y_\ell^\dagger - y_\ell^\delta\|_{L^2(\partial\Omega \times (0,T))} \leq \delta.
\end{align}

To reduce computational cost during the inverse iteration, we do not minimize \eqref{eq:wave.misfit.all.sources} directly, but instead use a standard \emph{sample average approximation} (SAA) \cite{haber2012anEffective}:
at each iteration, we combine all sources $f_\ell$ into a single ``super-shot'' $f^{(m)}$,
\begin{align}
	f^{(m)} = \sum_{\ell = 1}^{N_s} \xi_\ell^{(m)} f_\ell,
\end{align}
where $\xi_\ell^{(m)} = \pm 1$ follow a Rademacher distribution with zero mean.
Thus, at each iteration $m$, we solve the minimization problem 
\begin{align}
	\min_{u \in \Psi^{(m)}} \frac{1}{2} \| y[u] - y^{(m),\delta} \|_{L^2(\partial\Omega \times (0,T))},
\end{align}
with corresponding boundary observations
\begin{align}
	y^{(m),\delta} = \sum_{\ell = 1}^{N_s} \xi_\ell^{(m)} y_\ell^\delta.
\end{align}

First, we compare the ASI Algorithm from Section \ref{sec:asi} with the former $\text{ASI}_{0}$  Algorithm, 
where we omit Step \ref{algo:add.sensitivities} and thus ignore the most sensitive AS basis functions required for the angle condition, see Remark \ref{rem:asi0}.
To do so, we consider the (unknown) medium $u^\dagger$ consisting of six discs shown in Figure \ref{fig:InverseProblems.uDagger}.
In Table \ref{tab:wave.discs}, we observe that the ASI and $\text{ASI}_{0}$ Algorithms perform similarly in terms of the relative error $e_{m_\ast}$ and dimension of the search space $K_{m_\ast}$.
However, when we compare the reconstructed media $u^{(m_\ast),\delta}$ from Figure \ref{fig:wave.discs}, we observe that the $\text{ASI}_{0}$ Algorithm (bottom row) fails to reconstruct the smallest disc with increasing noise, whereas the ASI Algorithm (top row) always recovers even that smallest disc.

In Figure \ref{fig:wave.discs.data},  the relative error $e_{m}$ decreases
throughout all iterations, while $K_{m}$, and hence the number of control variables, remains small, which keeps the overall computational cost low.
As expected from Theorem \ref{thm:convergence.gradient}, the norm of the gradient $\|D\misfit^\delta(u^{(m),\delta})\|$ decreases and the ratio $\tau_{m}$ from the discrepancy principle \eqref{eq:tau.discrepancy.principle} tends to 1.
\begin{table}[ht!]
	\begin{center}
		\caption{Inverse scattering problem, six discs:
				the relative error $e_{m_\ast}$, the total number of iterations $m_\ast$, and the dimension $K_{m_\ast}$ of the search space are shown for the ASI and $\text{ASI}_{0}$ method.}
		\label{tab:wave.discs}
		\begin{tabular}{l||cccc|cccc}
			Method & \multicolumn{4}{c|}{ASI} & \multicolumn{4}{c}{$\text{ASI}_{0}$} \\\hline\hline
			$\hat\delta$ & $1\%$ & $2\%$ & $5\%$ & $10\%$ & $1\%$ & $2\%$ & $5\%$ & $10\%$ \\\hline
			$e_{m_\ast}$ & $2.6\%$ & $2.4\%$ & $2.0\%$ & $1.8\%$ & $2.4\%$ & $3.4\%$ & $4.4\%$ & $3.2\%$ \\
			$m_\ast$ & $50$ & $48$ & $50$ & $32$ & $50$ & $50$ & $31$ &  $49$ \\
			$K_{m_\ast}$ & $81$ & $133$ & $85$ & $69$ & $83$ & $28$ & $57$ & $83$ \\
			$\tau_{m_\ast}$ & $1.076$ & $1.006$ & $1.0003$ & $1.0001$ & $1.047$ & $1.26$ & $1.061$ & $1.017$ \\
			$\delta$ & $0.18\%$ & $0.35\%$ & $0.88\%$ & $1.76\%$ & $0.18\%$ & $0.35\%$ & $0.88\%$ & $1.76\%$
		\end{tabular}
	\end{center}	
\end{table}
\begin{figure}[ht!]
	\begin{subfigure}{0.32\textwidth}
		\centering
		\includegraphics[width=1\textwidth]{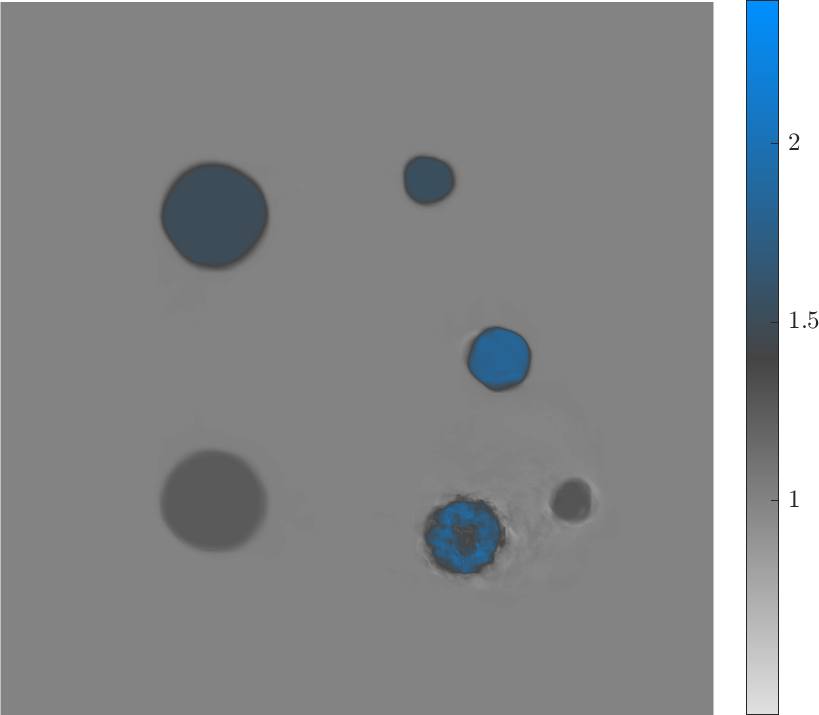}
		\caption{ASI: $u^{(m_\ast),\delta}$ for $\hat\delta = 1\%$}
	\end{subfigure}
	\hfill
	\begin{subfigure}{0.32\textwidth}
		\centering
		\includegraphics[width=1\textwidth]{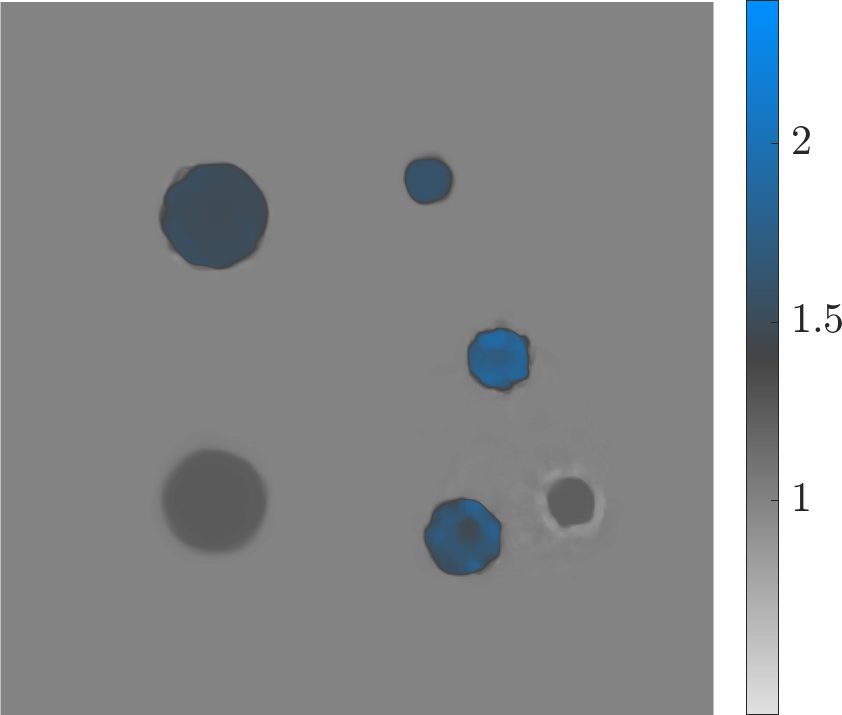}
		\caption{ASI: $u^{(m_\ast),\delta}$ for $\hat\delta = 2\%$}
	\end{subfigure}
	\hfill
	\begin{subfigure}{0.32\textwidth}
		\centering
		\includegraphics[width=1\textwidth]{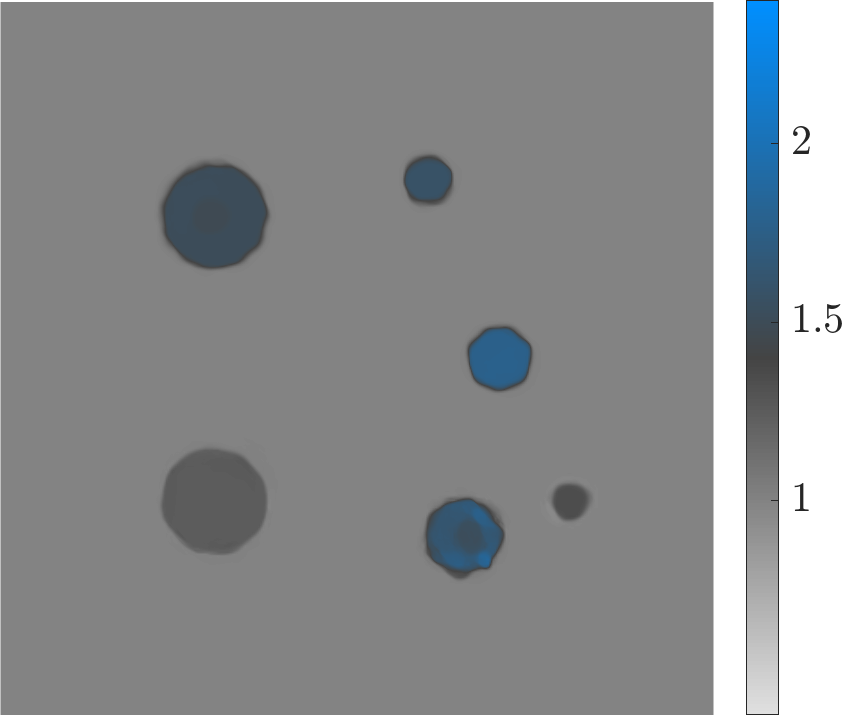}
		\caption{ASI: $u^{(m_\ast),\delta}$ for $\hat\delta = 5\%$}
	\end{subfigure}
	\\[2ex]
	\begin{subfigure}{0.32\textwidth}
		\centering
		\includegraphics[width=1\textwidth]{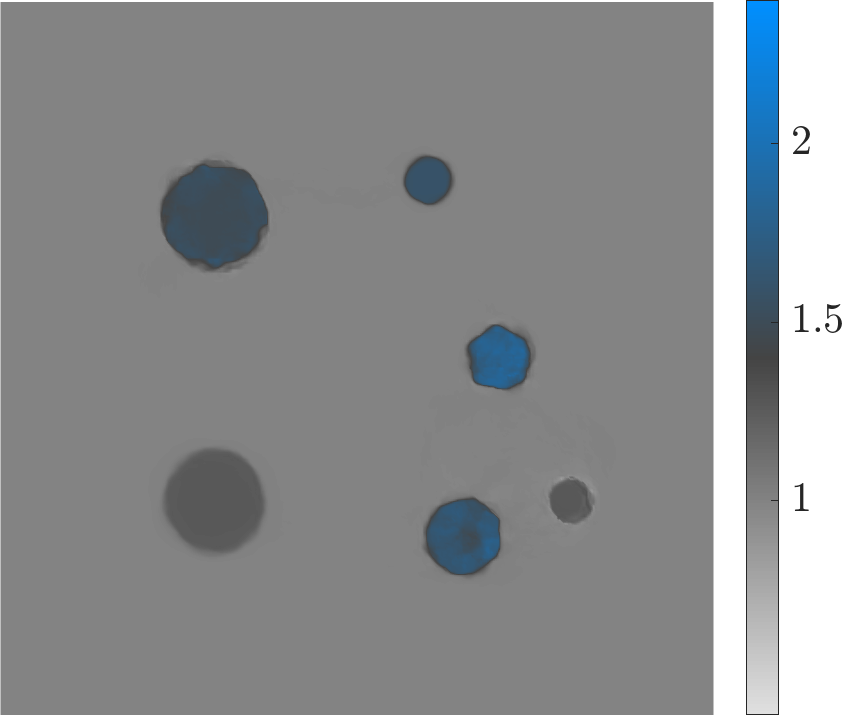}
		\caption{$\text{ASI}_{0}$: $u^{(m_\ast),\delta}$ for $\hat\delta = 1\%$}
	\end{subfigure}
	\hfill
	\begin{subfigure}{0.32\textwidth}
		\centering
		\includegraphics[width=1\textwidth]{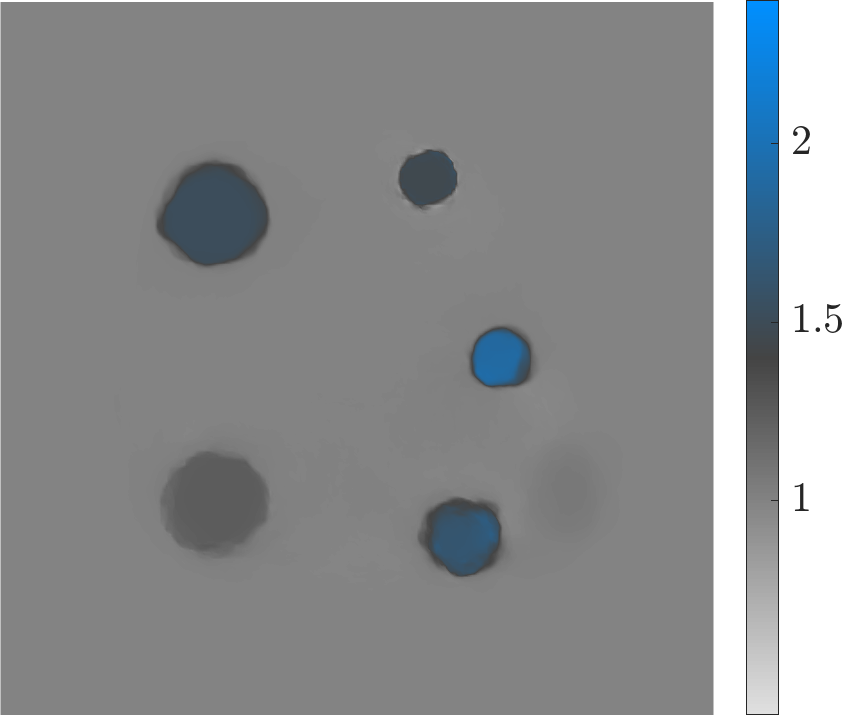}
		\caption{$\text{ASI}_{0}$: $u^{(m_\ast),\delta}$ for $\hat\delta = 2\%$}
	\end{subfigure}
	\hfill
	\begin{subfigure}{0.32\textwidth}
		\centering
		\includegraphics[width=1\textwidth]{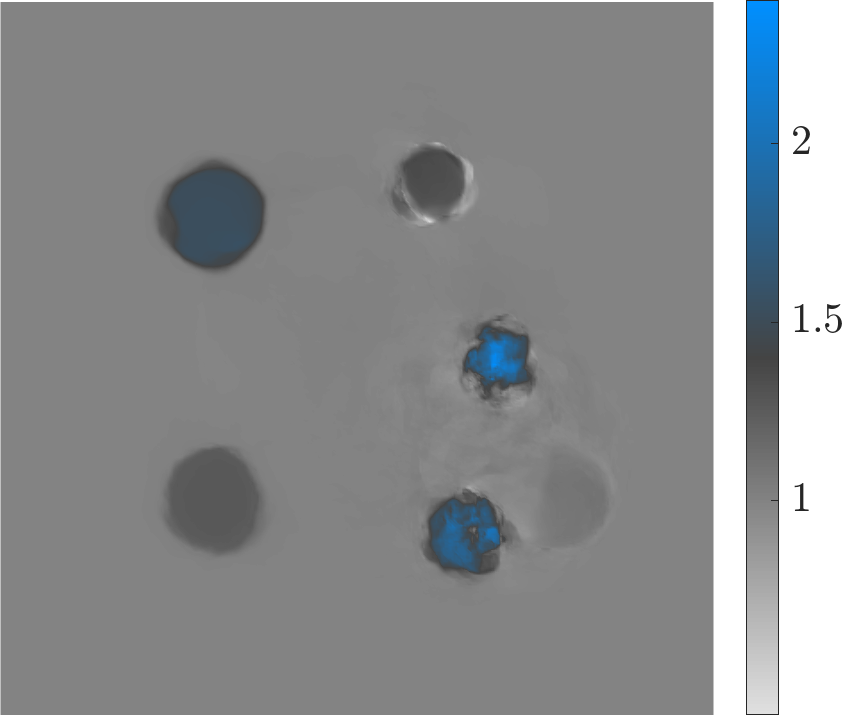}
		\caption{$\text{ASI}_{0}$: $u^{(m_\ast),\delta}$ for $\hat\delta = 5\%$}
	\end{subfigure}
	\caption{Inverse scattering problem, six discs:
				Comparison of ASI (top) and $\text{ASI}_{0}$ (bottom) for different noise levels $\hat\delta$.}
	\label{fig:wave.discs}
\end{figure}
\begin{figure}[ht!]
	\begin{subfigure}[c]{0.49\textwidth}
		\centering
		\includegraphics[width=0.8\textwidth]{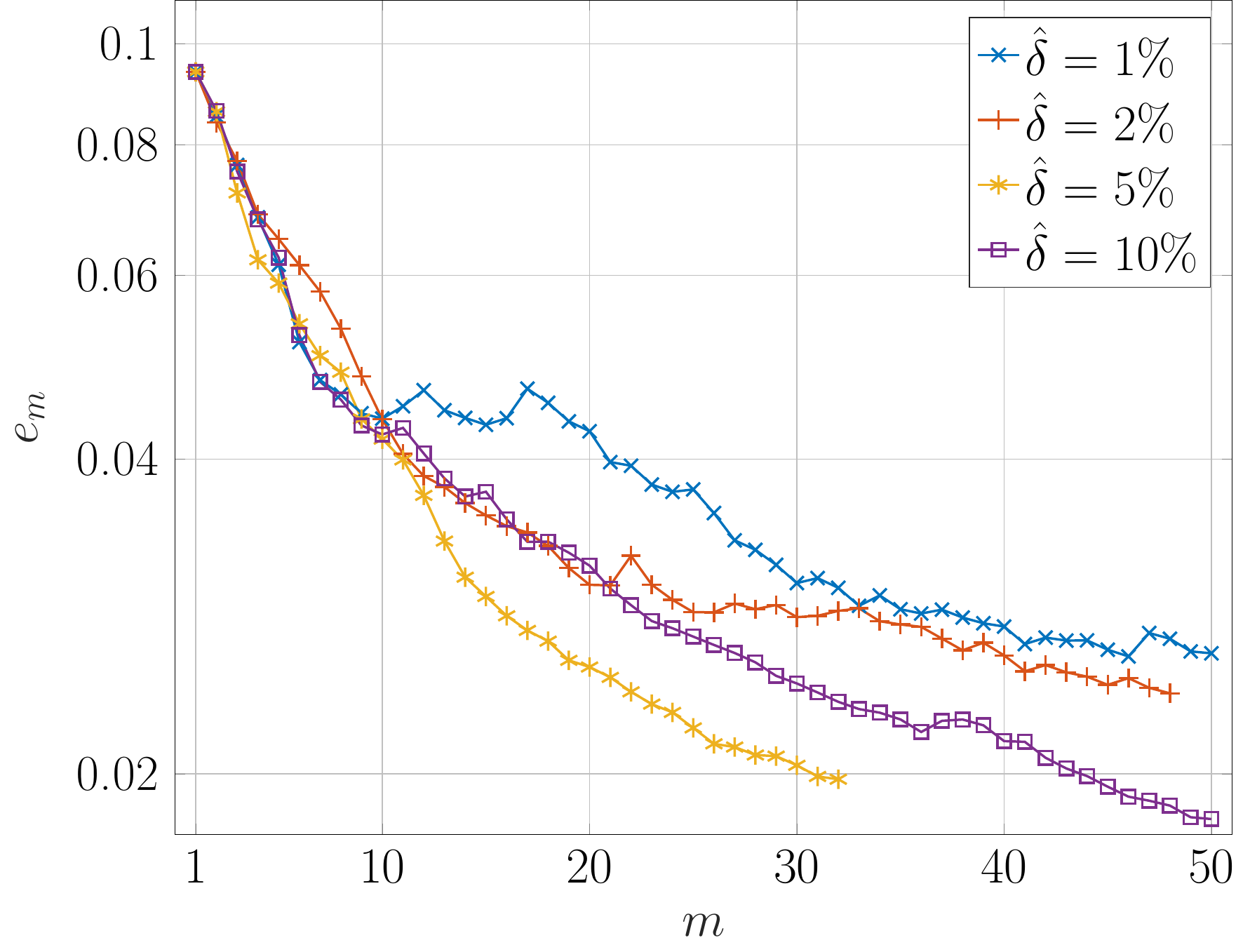}
		\caption{relative $L^2$ error $e_{m}$}
	\end{subfigure}
	\hfill
	\begin{subfigure}[c]{0.49\textwidth}
		\centering
		\includegraphics[width=0.8\textwidth]{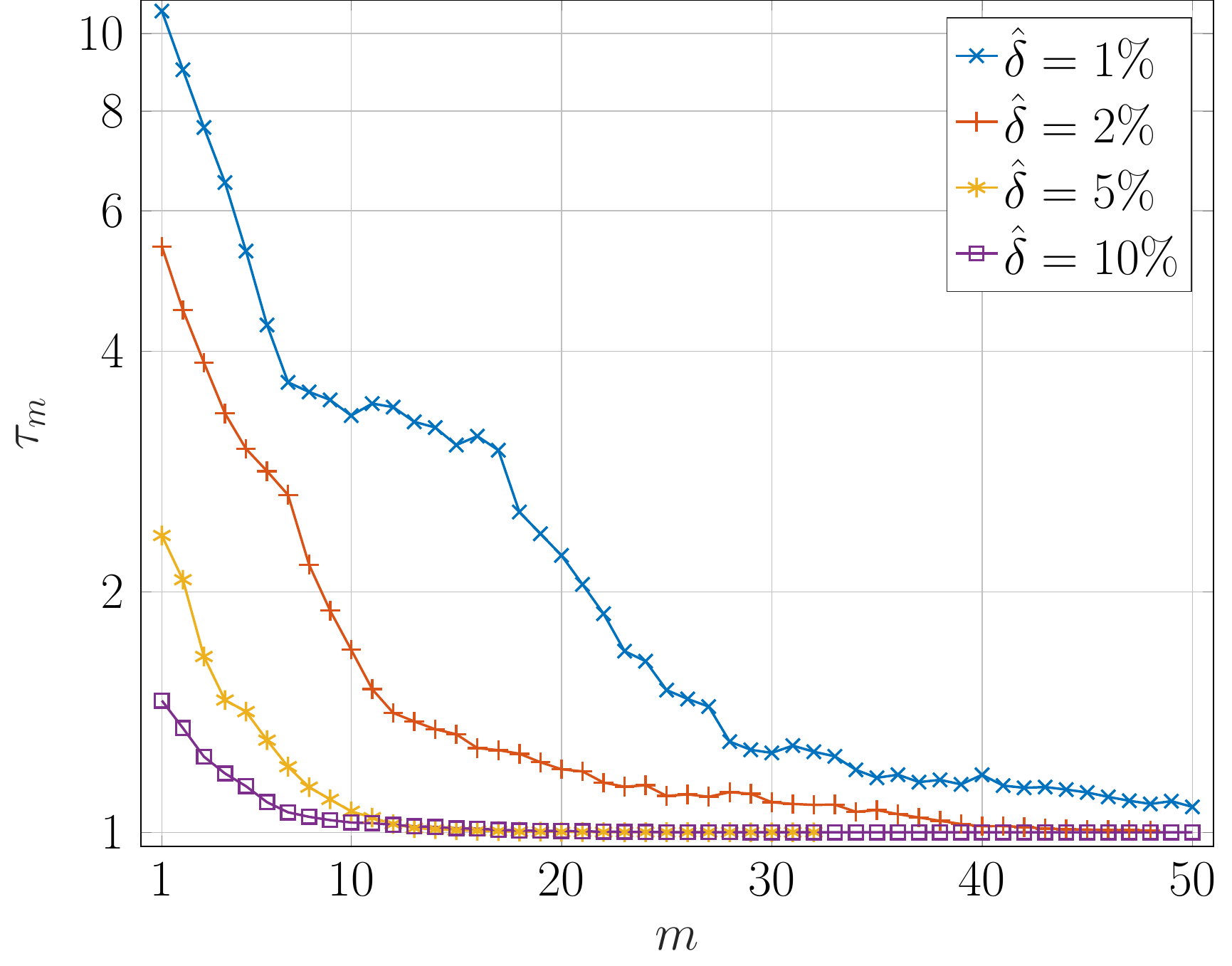}
		\caption{ratio $\tau_{m}$ from the discrepancy principle}
	\end{subfigure}
	\\[2ex]
	\begin{subfigure}[c]{0.49\textwidth}
		\centering
		\includegraphics[width=0.8\textwidth]{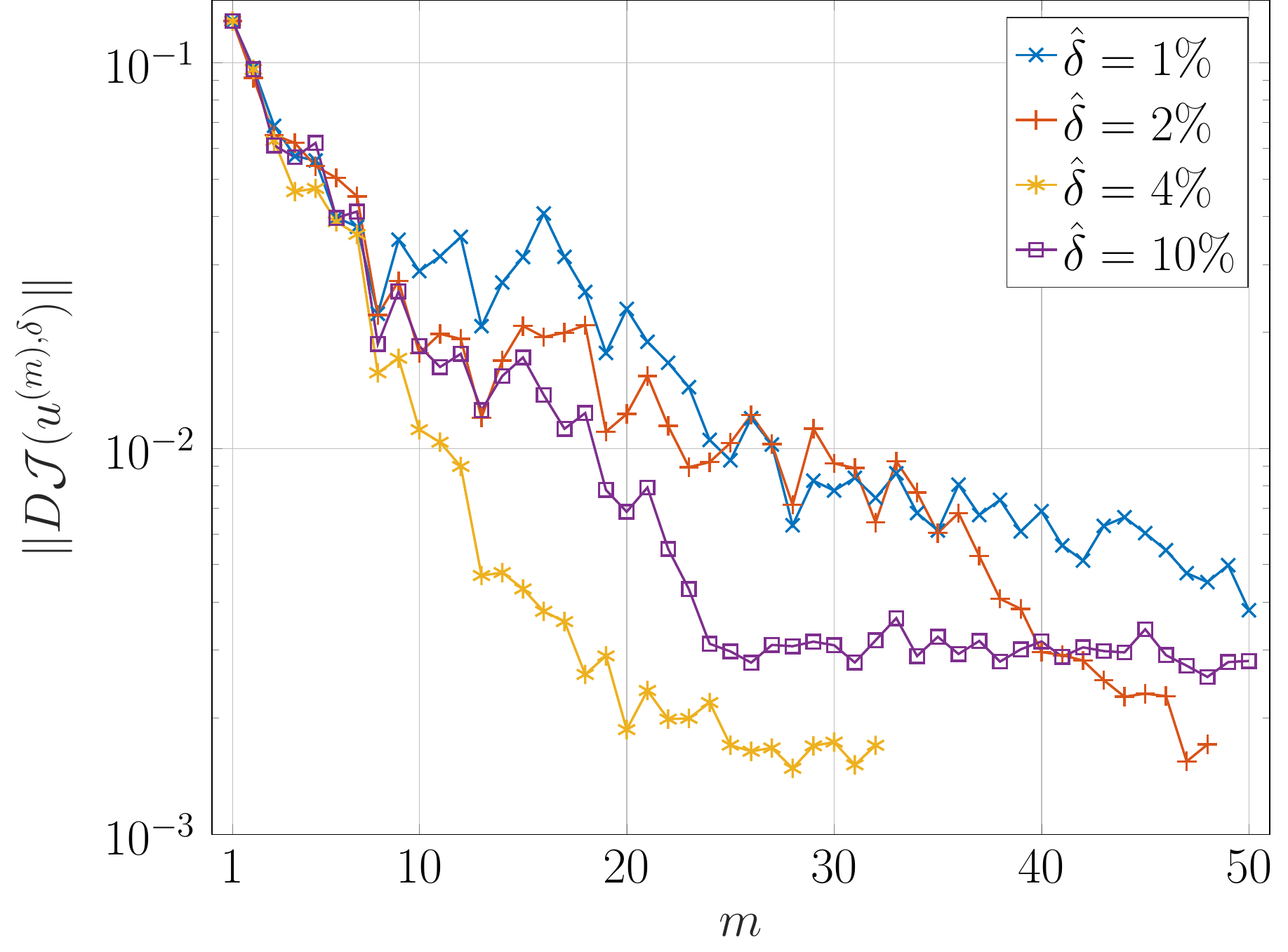}
		\caption{$\|D\misfit(u^{(m),\delta})\|$}
	\end{subfigure}
	\hfill
	\begin{subfigure}[c]{0.49\textwidth}
		\centering
		\includegraphics[width=0.8\textwidth]{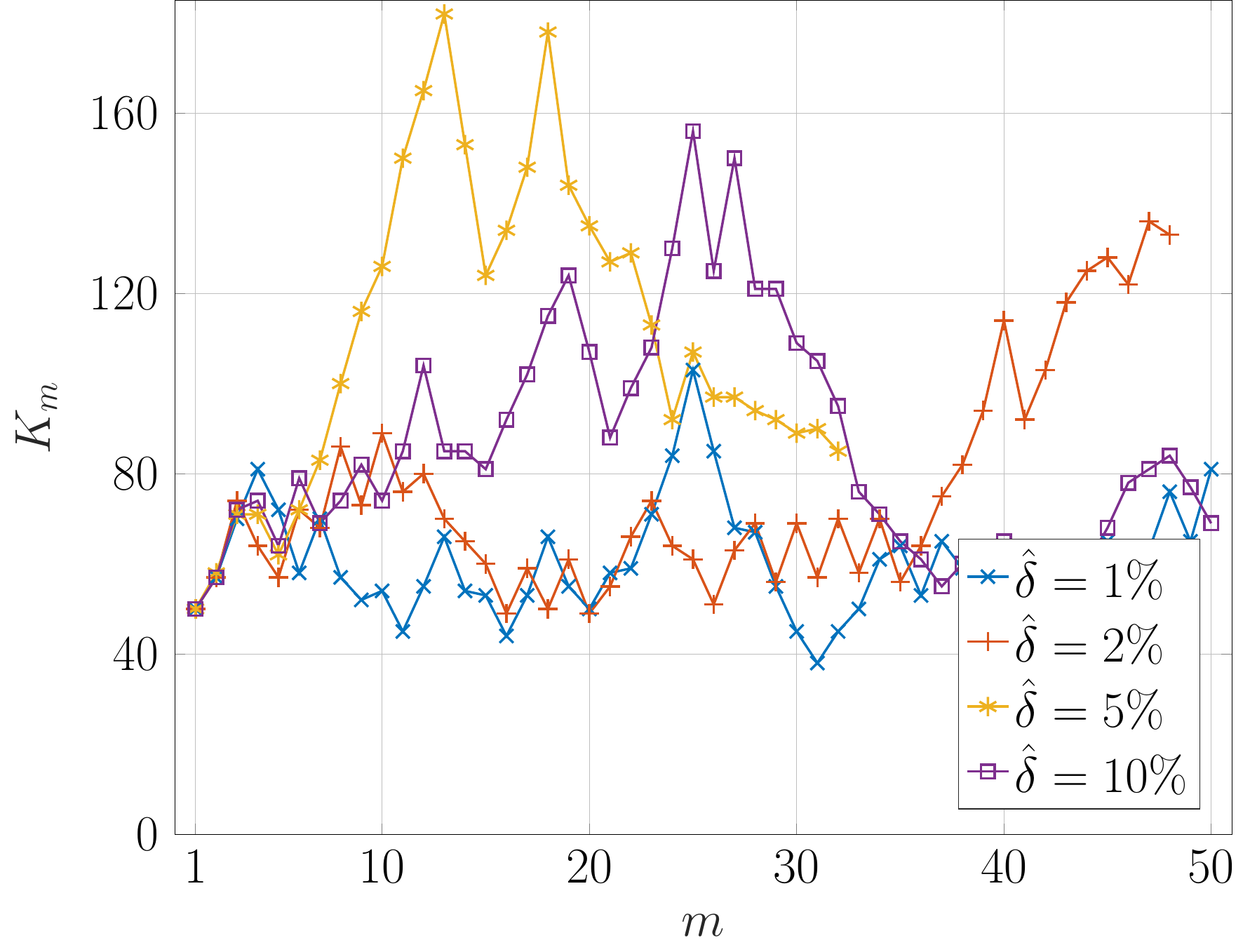}
		\caption{dimension $K_{m}$}
	\end{subfigure}
	\caption{Inverse scattering problem, six discs:
			The relative error \eqref{eq:rel.error}, the ratio $\tau_{m}$ from the discrepancy principle \eqref{eq:tau.discrepancy.principle}, the norm of the gradient, and the dimension of the search space are shown at every iteration $m$ for the ASI method and the different noise levels $\hat\delta$.}
	\label{fig:wave.discs.data}
\end{figure}

Finally, we consider the medium $u^\dagger$ shown in Figure \ref{fig:InverseProblems.uDagger} with three geometric inclusions. 
As shown in Figure \ref{fig:three.inclusions}, the ASI Algorithm recovers the shape and height of all three inclusions with high fidelity and regardless of the noise level $\hat\delta$.
In Table \ref{tab:three.inclusions} and Figure \ref{fig:three.inclusions.data}, we observe that the relative error $e_{m}$ remains low for all $\hat\delta$ and decreases throughout all iterations.
Again, the ratio $\tau_{m}$ from \eqref{eq:tau.discrepancy.principle} tends to $1$ while $\|D\misfit^\delta(u^{(m),\delta})\|$ decreases.
For all noise levels $\hat\delta$, the number of basis functions $K_{m}$ never exceeds $140$, which greatly reduces the computational effort compared to a standard nodal based optimization approach with $160'000$ control variables. Since the ASI and $\text{ASI}_{0}$ Algorithms performed similarly, the results from the latter are omitted here. 

\begin{table}[ht!]
	\begin{center}
		\caption{Inverse scattering problem, three inclusions:
				the relative error $e_{m_\ast}$, the total number of iterations $m_\ast$, and the dimension $K_{m_\ast}$ of the search space are shown for the ASI method.}
		\label{tab:three.inclusions}
		\begin{tabular}{l||cccc}
			Method & \multicolumn{4}{c}{ASI} \\\hline\hline
			$\hat\delta$ & $1\%$ & $2\%$ & $5\%$ & $10\%$ \\\hline
			$e_{m_\ast}$ & $4.0\%$ & $3.4\%$ & $3.2\%$ & $3.4\%$ \\
			$m_\ast$ & $50$ & $49$ & $50$ & $50$ \\
			$K_{m_\ast}$ & $114$ & $83$ & $61$ & $45$ \\
			$\tau_{m_\ast}$ & $1.075$ & $1.015$ & $1.005$ & $1.002$ \\
			$\delta$ & $0.18\%$ & $0.35\%$ & $0.88\%$ & $1.46\%$
		\end{tabular}
	\end{center}	
\end{table}
\begin{figure}[ht!]
	\begin{subfigure}{0.32\textwidth}
		\centering
		\includegraphics[width=1\textwidth]{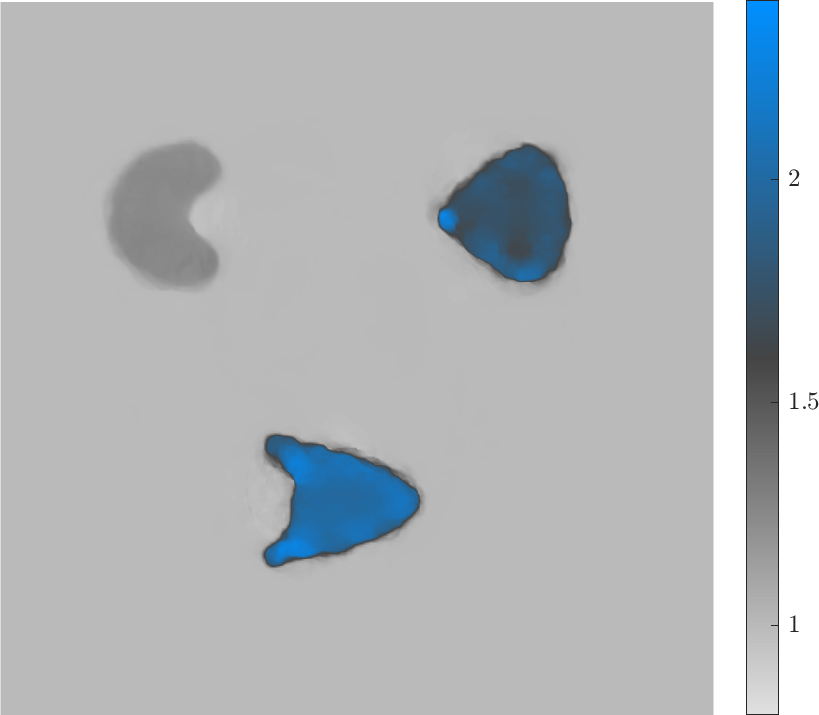}
		\caption{ASI: $u^{(m_\ast),\delta}$ for $\hat\delta = 1\%$}
	\end{subfigure}
	\hfill
	\begin{subfigure}{0.32\textwidth}
		\centering
		\includegraphics[width=1\textwidth]{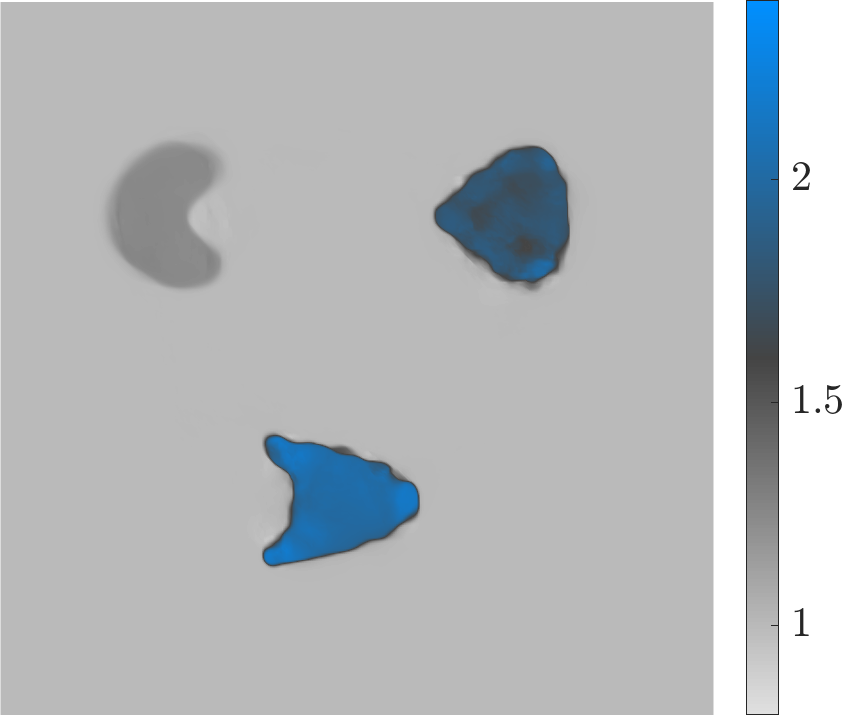}
		\caption{ASI: $u^{(m_\ast),\delta}$ for $\hat\delta = 2\%$}
	\end{subfigure}
	\hfill
	\begin{subfigure}{0.32\textwidth}
		\centering
		\includegraphics[width=1\textwidth]{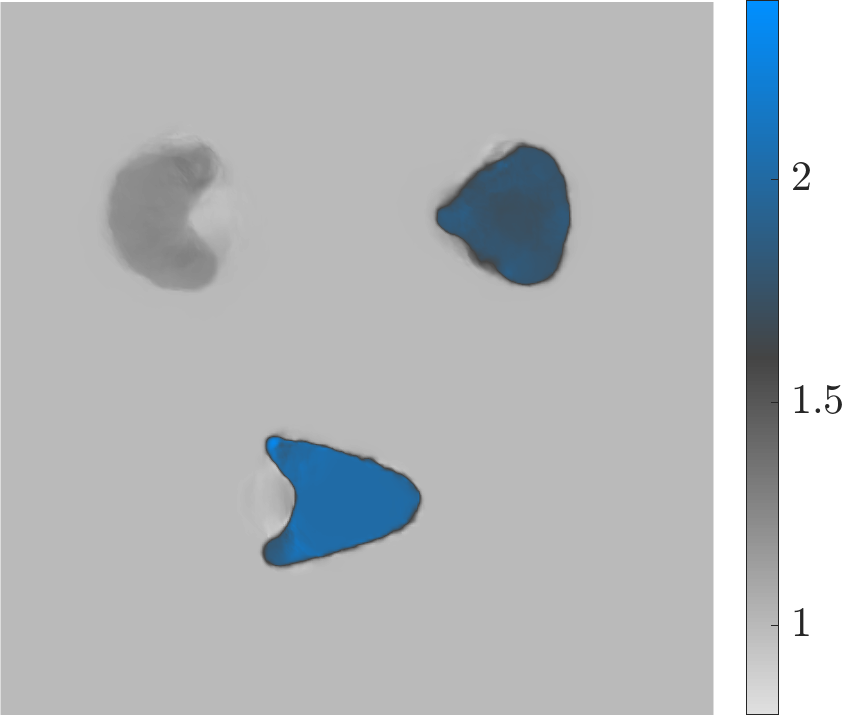}
		\caption{ASI: $u^{(m_\ast),\delta}$ for $\hat\delta = 5\%$}
	\end{subfigure}
	\caption{Inverse scattering problem, three inclusions:
				reconstructed medium using the ASI method for different noise levels $\hat\delta$.}
	\label{fig:three.inclusions}
\end{figure}
\begin{figure}[ht!]
	\begin{subfigure}[c]{0.49\textwidth}
		\centering
		\includegraphics[width=0.8\textwidth]{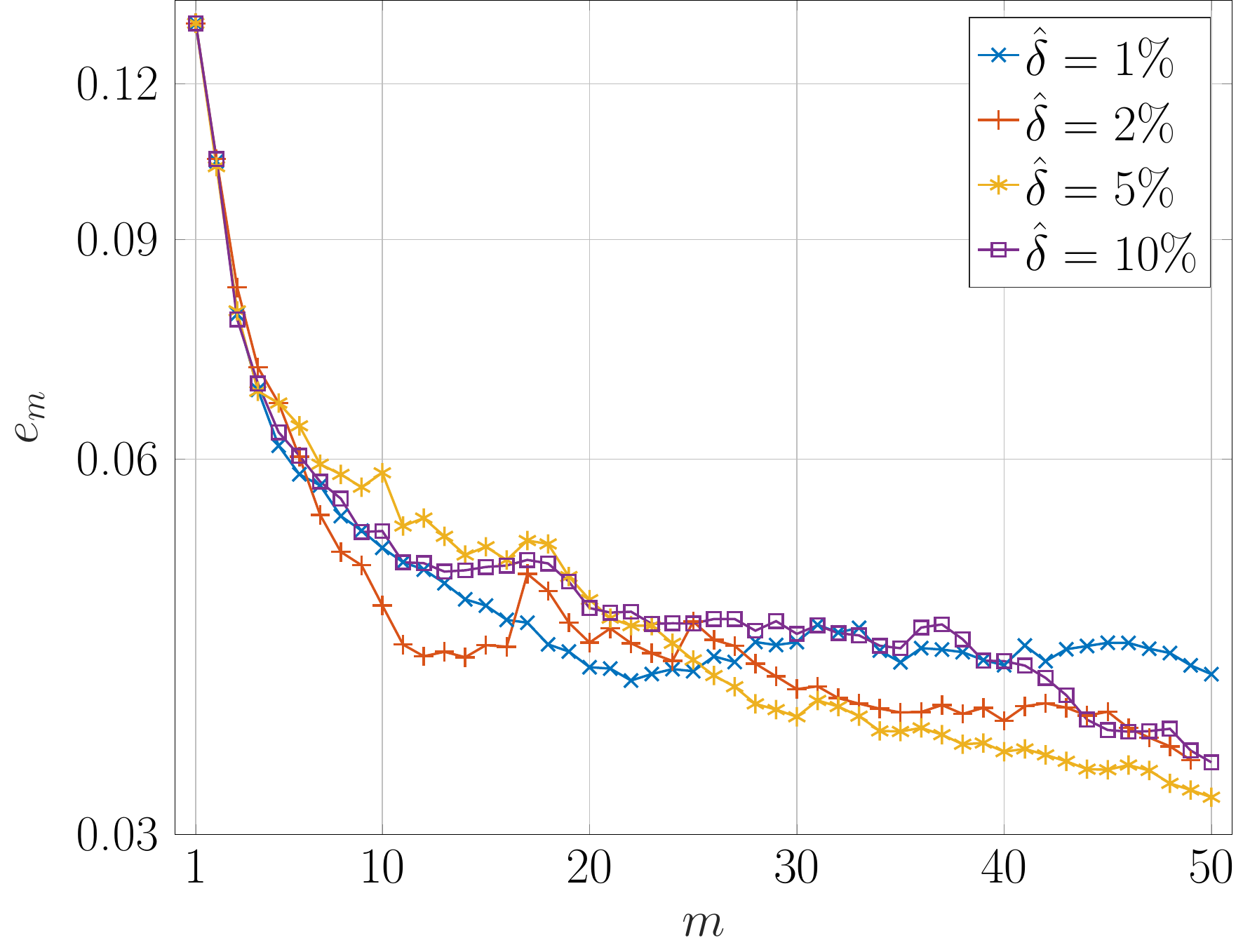}
		\caption{relative $L^2$ error $e_{m}$}
	\end{subfigure}
	\hfill
	\begin{subfigure}[c]{0.49\textwidth}
		\centering
		\includegraphics[width=0.8\textwidth]{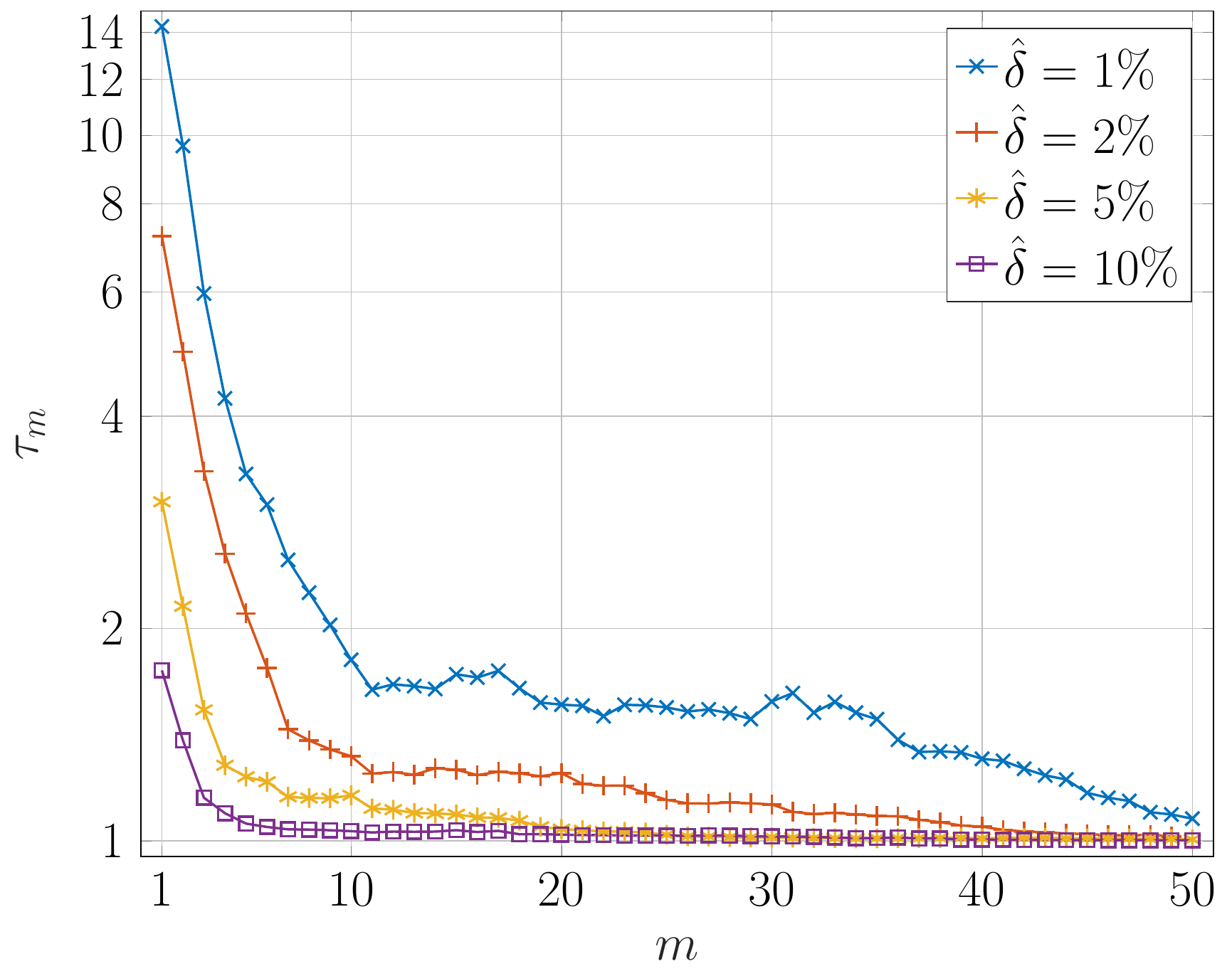}
		\caption{ratio $\tau_{m}$ from the discrepancy principle}
	\end{subfigure}
	\\[2ex]
	\begin{subfigure}[c]{0.49\textwidth}
		\centering
		\includegraphics[width=0.8\textwidth]{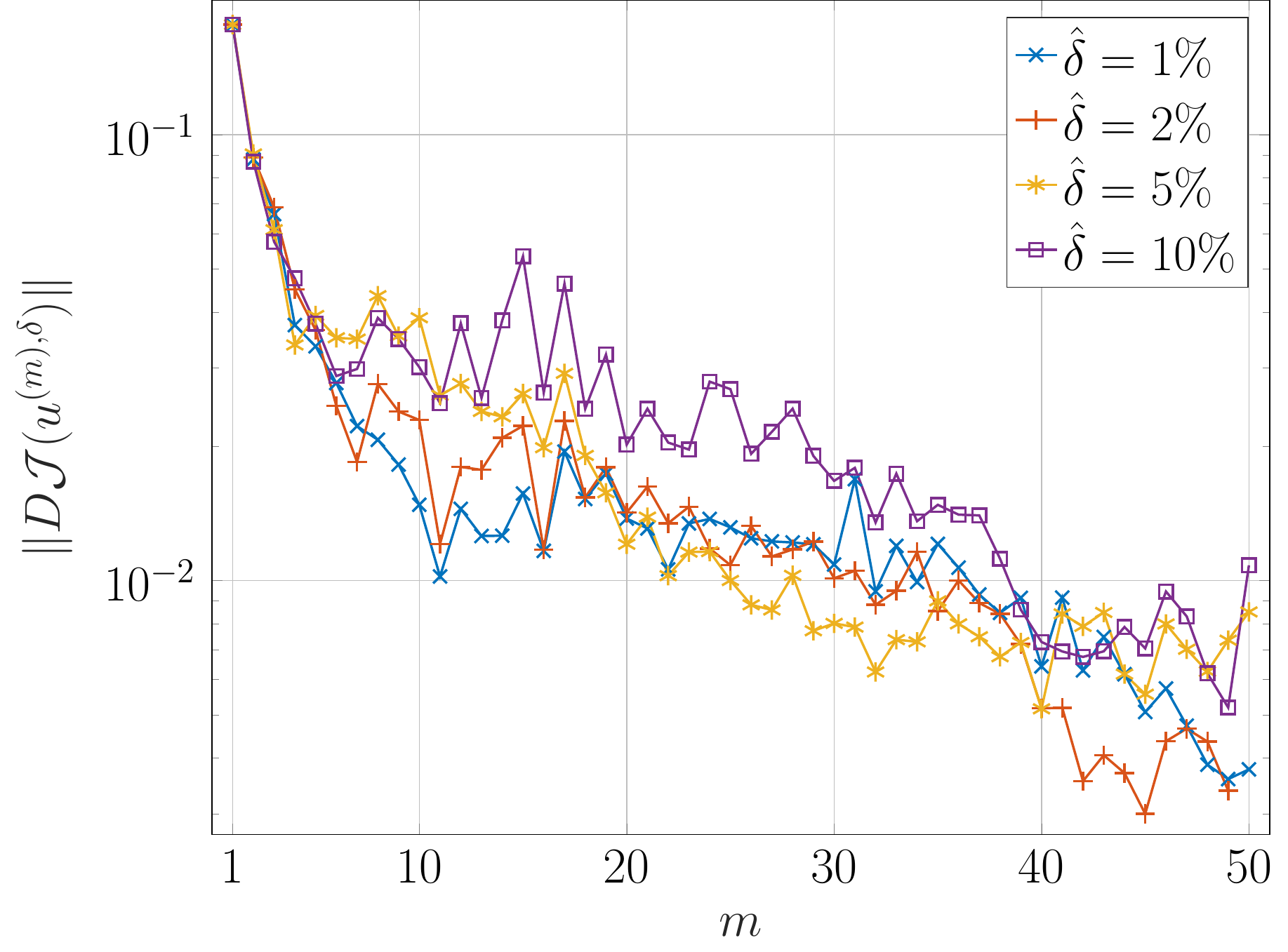}
		\caption{$\|D\misfit(u^{(m),\delta})\|$}
	\end{subfigure}
	\hfill
	\begin{subfigure}[c]{0.49\textwidth}
		\centering
		\includegraphics[width=0.8\textwidth]{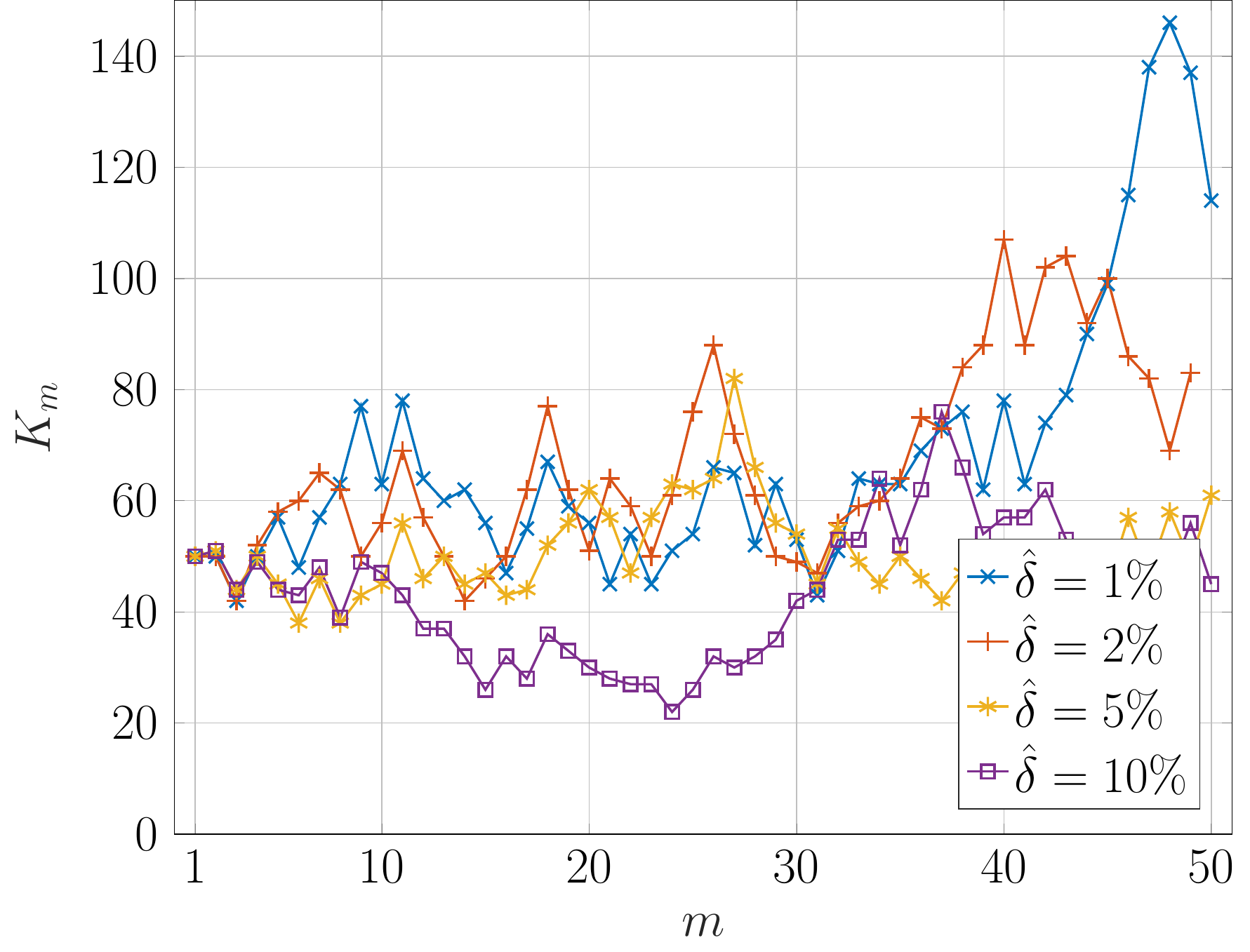}
		\caption{dimension $K_{m}$}
	\end{subfigure}
	\caption{Inverse scattering problem, three inclusions:
			The relative error \eqref{eq:rel.error}, the ratio $\tau_{m}$ from the discrepancy principle \eqref{eq:tau.discrepancy.principle}, the norm of the gradient, and the dimension of the search space, at iteration $m$ for the ASI method and different noise levels $\hat\delta = 1\%,\;2\%\;5\%,\;10\%$.}
	\label{fig:three.inclusions.data}
\end{figure}
\section{Concluding Remarks}
The Adaptive Spectral Inversion (ASI) method has proved remarkably effective for the solution of PDE-constrained inverse medium problems. At the $m$-th iteration, the ASI method
minimizes the data misfit with added noise $\delta$ in a small subspace $\Psi^{(m)}$, which inherits key information from the previous step while adding promising new search directions from the first few eigenfunctions
of the elliptic operator $L_\varepsilon[u^{(m),\delta}]$ in \eqref{eq:as.eigenfunctions}. Since the operator $L_\varepsilon[u^{(m),\delta}]$ itself depends on the current iterate, $u^{(m),\delta}$, so do its eigenfunctions and thus also the new search space $\Psi^{(m+1)}$. The convergence of the ASI method hinges upon the angle condition \eqref{eq:angle.condition} -- see Theorem \ref{thm:convergence.gradient} -- which has been newly included as a final step into each iteration. The full ASI Algorithm \ref{algo:ASI} is listed in Section \ref{sec:asi}.

Under suitable assumptions, the ASI iteration stops after (finitely many) $m_*$ iterations when the discrepancy principle \eqref{eq:discrepancy.principle} is satisfied. Hence, the ASI
Algorithm yields a genuine regularization method whose solution $u^{(m_*),\delta}$ converges to
the exact (noise-free) solution as $\delta \rightarrow 0$, without the need for extra Tikhonov regularization 
-- see Theorem \ref{thm:regularization}.
By adapting the search space at each iteration while keeping its dimension low, the ASI method achieves more accurate reconstructions than standard grid-based Tikhonov $L^2$-regularization together with a thousandfold decrease in the number of unknowns. Thanks to the newly incorporated angle condition \eqref{eq:angle.condition}, 
the ASI Algorithm is able to detect even the smallest inclusions in the medium, which previous versions  of the algorithm \cite{grote2019adaptive} at times failed to identify with increasing noise.

The added cost from the numerical solution of the eigenvalue problem \eqref{eq:as.eigenfunctions}
is rather small. On the one hand, the Galerkin FE discretization leads to a generalized eigenvalue problem which is sparse, symmetric, and positive definite. On the other hand, we only require a few eigenfunctions which can be efficiently computed by a standard Lanczos iteration. A further reduction 
in the computational cost can easily be achieved by using an adaptively refined FE mesh when 
solving \eqref{eq:as.eigenfunctions}, as  in \cite{grote2017adaptive}, which need not coincide with the mesh used for discretizing the governing PDE. As the higher eigenfunctions become increasingly localized, they can easily be ``sparsified'' simply by setting to zero the smallest coefficients in their discrete FE representation \cite{grote2017adaptive}.

Although we have concentrated here on two-dimensional inverse medium problems, the AS decomposition in fact applies to arbitrary many space dimensions \cite{baffet2022error}. 
When the target medium is not piecewise constant but instead smoothly varying, adaptive spectral bases resulting
from different elliptic operators may be more effective \cite{grote2019adaptive}. Clearly, the ASI 
approach would probably prove useful for inverse source problems, for instance, or could be 
be combined with alternative (globally convergent) inversion methods \cite{beilina2012approximate}.
It could also be applied to other inverse problems, unrelated to wave scattering, where the state variable is governed by a different partial differential, or possibly even integral, equation.

\paragraph{Acknowledgments.}
We thank Daniel Baffet for useful comments and suggestions.
%
%
%
\bibliographystyle{siam}
\bibliography{literature.bib}

\begin{thebibliography}{10}

\bibitem{baffet2022error}
{\sc D.~H. Baffet, Y.~G. Gleichmann, and M.~J. Grote}, {\em {E}rror estimates
  for aptive spectral decompositions}, Journal of Scientific Computing, 93
  (2022).

\bibitem{baffet2021adaptive}
{\sc D.~H. Baffet, M.~J. Grote, and J.~H. Tang}, {\em {A}daptive spectral
  decompositions for inverse medium problems}, Inverse Problems, 37 (2021),
  p.~025006.

\bibitem{bakushinsky2004iterative}
{\sc A.~B. Bakushinsky and M.~Y. Kokurin}, {\em {I}terative {M}ethods for
  {A}pproximate {S}olution of {I}nverse {P}roblems}, Springer Netherlands,
  2004.

\bibitem{bauschke2017convex}
{\sc H.~H. Bauschke and P.~L. Combettes}, {\em Convex {A}nalysis and {M}onotone
  {O}perator {T}heory in {H}ilbert {S}paces}, Springer International
  Publishing, 2017.

\bibitem{beilina2012approximate}
{\sc L.~Beilina and M.~V. Klibanov}, {\em {A}pproximate {G}lobal {C}onvergence
  and {A}daptivity for {C}oefficient {I}nverse {P}roblems}, Springer {US},
  2012.

\bibitem{burger2016spectral}
{\sc M.~Burger, G.~Gilboa, M.~Moeller, L.~Eckardt, and D.~Cremers}, {\em
  Spectral decompositions using one-homogeneous functionals}, {SIAM} Journal on
  Imaging Sciences, 9 (2016), pp.~1374--1408.

\bibitem{cohen2001higher}
{\sc G.~Cohen, P.~Joly, J.~E. Roberts, and N.~Tordjman}, {\em {H}igher order
  triangular finite elements with mass lumping for the wave equation}, {SIAM}
  Journal on Numerical Analysis, 38 (2001), pp.~2047--2078.

\bibitem{DDM2004}
{\sc I.~Daubechies, M.~Defrise, and C.~De~Mol}, {\em An iterative thresholding
  algorithm for linear inverse problems with a sparsity constraint},
  Communications on Pure and Applied Mathematics, 57 (2004), pp.~1413--1457.

\bibitem{deBuhan2017numerical}
{\sc M.~de~Buhan and M.~Darbas}, {\em {N}umerical resolution of an
  electromagnetic inverse medium problem at fixed frequency}, Computers \&
  Mathematics with Applications, 74 (2017), pp.~3111--3128.

\bibitem{deBuhan2013aNewApproach}
{\sc M.~de~Buhan and M.~Kray}, {\em {A} new approach to solve the inverse
  scattering problem for waves: combining the {TRAC} and the adaptive inversion
  methods}, Inverse Problems, 29 (2013), p.~085009.

\bibitem{deBuhan2010logaritmic}
{\sc M.~de~Buhan and A.~Osses}, {\em {L}ogarithmic stability in determination
  of a 3d viscoelastic coefficient and a numerical example}, Inverse Problems,
  26 (2010), p.~095006.

\bibitem{engl1996regularization}
{\sc H.~W. Engl, M.~Hanke, and A.~Neubauer}, {\em {R}egularization of {I}nverse
  {P}roblems}, vol.~375, Springer Science \& Business Media, 1996.

\bibitem{faucher2020eigenvector}
{\sc F.~Faucher, O.~Scherzer, and H.~Barucq}, {\em Eigenvector models for
  solving the seismic inverse problem for the {H}elmholtz equation},
  Geophysical Journal International,  (2020).

\bibitem{gilboa2016nonlinear}
{\sc G.~Gilboa, M.~Moeller, and M.~Burger}, {\em Nonlinear spectral analysis
  via one-homogeneous functionals: Overview and future prospects}, Journal of
  Mathematical Imaging and Vision, 56 (2016), pp.~300--319.

\bibitem{graff2019howToSolve}
{\sc M.~Graff, M.~J. Grote, F.~Nataf, and F.~Assous}, {\em How to solve inverse
  scattering problems without knowing the source term: a three-step strategy},
  Inverse Problems, 35 (2019), p.~104001.

\bibitem{groetsch1988convergence}
{\sc C.~W. Groetsch and A.~Neubauer}, {\em Convergence of a general projection
  method for an operator equation of the first kind}, Houston J. Math., 14
  (1988), pp.~201--208.

\bibitem{grote2017adaptive}
{\sc M.~J. Grote, M.~Kray, and U.~Nahum}, {\em Adaptive eigenspace method for
  inverse scattering problems in the frequency domain}, Inverse Problems, 33
  (2017), p.~025006.

\bibitem{grote2019adaptive}
{\sc M.~J. Grote and U.~Nahum}, {\em Adaptive eigenspace for multi-parameter
  inverse scattering problems}, Computers \& Mathematics with Applications, 77
  (2019), pp.~3264--3280.

\bibitem{haber2012anEffective}
{\sc E.~Haber, M.~Chung, and F.~Herrmann}, {\em An effective method for
  parameter estimation with {PDE} constraints with multiple right-hand sides},
  {SIAM} Journal on Optimization, 22 (2012), pp.~739--757.

\bibitem{hamarik2002on}
{\sc U.~H\"{a}marik, E.~Avi, and A.~Ganina}, {\em On the solution of ill-posed
  problems by projection methods with a posteriori choice of the discretization
  level}, Math. Model. Anal., 7 (2002), pp.~241--252.

\bibitem{hanke1995aConvergence}
{\sc M.~Hanke, A.~Neubauer, and O.~Scherzer}, {\em A convergence analysis of
  the {L}andweber iteration for nonlinear ill-posed problems}, Numerische
  Mathematik, 72 (1995), pp.~21--37.

\bibitem{HH2008}
{\sc F.~J. Herrmann and G.~Hennenfent}, {\em Non-parametric seismic data
  recovery with curvelet frames}, Geophys. J. Int.,  (2008), pp.~233--248.

\bibitem{hofmann2007regularization}
{\sc B.~Hofmann, P.~Math{\'{e}}, and S.~V. Pereverzev}, {\em Regularization by
  projection: Approximation theoretic aspects and distance functions}, Journal
  of Inverse and Ill-posed Problems, 15 (2007).

\bibitem{kaltenbacher2000regularization}
{\sc B.~Kaltenbacher}, {\em Regularization by projection with a posteriori
  discretization level choice for linear and nonlinear ill-posed problems},
  Inverse Problems, 16 (2000), pp.~1523--1539.

\bibitem{kaltenbacher2008iterative}
{\sc B.~Kaltenbacher, A.~Neubauer, and O.~Scherzer}, {\em {I}terative
  {R}egularization {M}ethods for {N}onlinear {I}ll-{P}osed {P}roblems}, De
  Gruyter, 2008.

\bibitem{kaltenbacher2012aConvergence}
{\sc B.~Kaltenbacher and J.~Offtermatt}, {\em A convergence analysis of
  regularization by discretization in preimage space}, Math. Comput., 81
  (2012), pp.~2049--2069.

\bibitem{kirsch1996an}
{\sc A.~Kirsch}, {\em {A}n {I}ntroduction to the {M}athematical {T}heory of
  {I}nverse {P}roblems}, vol.~120 of Applied Mathematical Sciences,
  Springer-Verlag, New York, 1996.

\bibitem{landweber1951anIterationFormula}
{\sc L.~Landweber}, {\em An iteration formula for {F}redholm integral equations
  of the first kind}, American Journal of Mathematics, 73 (1951), pp.~615--624.

\bibitem{LAH2012}
{\sc Y.~Lin, A.~Abubakar, and T.~M. Habashy}, {\em Seismic full-waveform
  inversion using truncated wavelet representations},  (2012), pp.~1--6.
\newblock SEG Annual meeting 2012, Las Vegas.

\bibitem{lines2005aTimeDomain}
{\sc C.~D. Lines and S.~N. Chandler-Wilde}, {\em A time domain point source
  method for inverse scattering by rough surfaces}, Computing, 75 (2005),
  pp.~157--180.

\bibitem{LDNDR2010}
{\sc I.~Loris, H.~Douma, G.~Nolet, I.~Daubechies, and C.~Regone}, {\em
  Nonlinear regularization techniques for seismic tomography}, Journal of
  Computational Physics,  (2010), pp.~890--905.

\bibitem{mulder2001higher}
{\sc W.~A. Mulder}, {\em Higher-order mass-lumped finite elements for the wave
  equation}, Journal of Computational Acoustics, 09 (2001), pp.~671--680.

\bibitem{natterer1977regularisierung}
{\sc F.~Natterer}, {\em {R}egularisierung schlecht gestellter {P}robleme durch
  {P}rojektionsverfahren}, Numerische Mathematik, 28 (1977), pp.~329--341.

\bibitem{clay2021anAdaptive}
{\sc C.~Sanders, M.~Bonnet, and W.~Aquino}, {\em An adaptive eigenfunction
  basis strategy to reduce design dimension in topology optimization},
  Internat. J. Numer. Methods Engrg., 122 (2021), pp.~7452--7481.

\bibitem{scherzer1993optimal}
{\sc O.~Scherzer, H.~W. Engl, and K.~Kunisch}, {\em Optimal a posteriori
  parameter choice for {T}ikhonov regularization for solving nonlinear
  ill-posed problems}, {SIAM} Journal on Numerical Analysis, 30 (1993),
  pp.~1796--1838.

\bibitem{nocedal2006numerical}
{\sc S.~Wright and J.~Nocedal}, {\em {N}umerical {O}ptimization}, Springer New
  York, 2006.

\end{thebibliography}

\end{document}